\documentclass[10pt, a4paper]{article} 
\usepackage[a4paper, left=3cm, right=3cm, top=3.5cm, bottom=3.5cm]{geometry}
\usepackage{amsmath}
\usepackage{amssymb}
\usepackage{mathrsfs}
\usepackage{amsthm}
\usepackage[english]{babel}
\usepackage{xcolor}
\usepackage{verbatim}
\usepackage{stmaryrd}
\usepackage{hyperref}
\usepackage[nameinlink,capitalize]{cleveref}
\usepackage{cancel}
\usepackage{enumitem}
\usepackage{csquotes}
\usepackage{dsfont}
\usepackage[
backend=biber,
style=alphabetic,
sorting=nyt,
maxnames=99,
doi=true,
isbn=false,
url=false,
]{biblatex}
\def\ud{\mathrm d}

\newcommand{\tnorm}[1]{\left\vert\kern-0.25ex\left\vert\kern-0.25ex\left\vert #1 
    \right\vert\kern-0.25ex\right\vert\kern-0.25ex\right\vert}
\newcommand{\R}{\mathbb{R}}
\newcommand{\E}{\mathbb{E}}

\newcommand{\esssup}{\mathop{\mathrm{ess\,sup}}}

\theoremstyle{definition}
\newtheorem{defn}{Definition}[section]

\newtheorem{rmk}[defn]{Remark}

\theoremstyle{plain}
\newtheorem{thm}[defn]{Theorem}
\newtheorem*{thm*}{Theorem}
\newtheorem{lemma}[defn]{Lemma}

\newtheorem{prop}[defn]{Proposition}

\newtheorem*{assum*}{Assumptions}

\usepackage[textsize=tiny]{todonotes}
\title{Mean-field games with rough common noise: 
the linear-quadratic case}
\author{Peter K. Friz\thanks{Technische Universit\"at Berlin and Weierstraß-Institut (friz@math.tu-berlin.de)}, Ioannis Gasteratos \thanks{Technische Universit\"at Berlin (i.gasteratos@tu-berlin.de)}, Ulrich Horst\thanks{Humboldt Universit\"at zu Berlin (horst@math.hu-berlin.de)}, Stefanos Theodorakopoulos\thanks{Technische Universit\"at Berlin (stefanos.theodorakopoulos@tu-berlin.de)}          }

\begin{document}

\maketitle

\begin{abstract}
Motivated by mean-field games (MFG) with common noise on the one hand and pathwise stochastic control theory on the other, we formulate here a linear-quadratic (LQ) MFG with rough common noise, along with a satisfactory well-posedness theory for the linear-quadratic case. A novel Volterra-type (or mild) formulation allows us to keep technical (rough stochastic) considerations to a minimum. We derive a characterization of the optimal state and optimal control through a rough forward-backward SDE (rough FBSDE), and obtain existence and uniqueness.
A number of stability estimates are established, and in particular we show continuous dependence of the \emph{It\^o-Lions-Lyons} map (equilibrium law as a function of the rough common noise). In a final section, we discuss independent Brownian randomization of the rough common noise. The resulting stochastic problem can be seen as an MFG extension of pathwise stochastic control problems.
\end{abstract}



\smallskip

\tableofcontents

\section{Introduction}

Mean-field game (MFG) theory provides a mathematical framework for analyzing strategic interactions in large populations of agents. Introduced independently by Lasry and Lions \cite{Lasry2006,Lasry2007} and Huang, Caines, and Malhamé \cite{huang2006,Huang2007}, MFGs describe the asymptotic behavior of Nash equilibria in $N$-player stochastic differential games as $N\to\infty$, when each individual player has negligible influence on the overall system but interacts through aggregate quantities such as the empirical distribution of states. The central idea of MFG theory is to replace the $N$-player game with a representative-agent control problem coupled to a consistency condition: each agent optimizes a cost functional that depends on the time-evolving distribution of the population and, in equilibrium, this distribution must coincide with the law of the optimally controlled state. This yields a characterization of equilibria as a fixed-point problem in the space of probability measures. 

Over the past two decades, two complementary approaches to MFG theory have emerged: an analytic (PDE) approach and a probabilistic (stochastic) approach. In their original formulation \cite{Lasry2006,Lasry2007}, Lasry and Lions characterized MFG equilibria through a coupled system of nonlinear partial differential equations consisting of a backward Hamilton--Jacobi--Bellman (HJB) equation for the representative agent's value function and a forward Fokker--Planck (Kolmogorov) equation describing the evolution of the population distribution. The probabilistic approach was pioneered by Carmona and Delarue (see their two-volume monograph \cite{carmona2018probabilistic,carmona2018probabilistic_2}). In this framework, the representative agent's state follows a stochastic differential equation whose coefficients depend on the law of the state process itself, and equilibria are characterized via forward--backward SDEs (FBSDEs) of McKean--Vlasov type. The probabilistic formulation is particularly powerful in the presence of common noise, where the population distribution becomes stochastic and evolves as a random measure. Recent advances in rough stochastic analysis, control, and mean-field models address rough semimartingales, pathwise control and SPDEs, rough SDEs, Malliavin calculus, McKean--Vlasov BSDEs, and modern control principles \cite{friz2023rough,FHL21,bugini2024malliavin,papapantoleon2024existenceuniquenesspropagationchaos,FLZ24,CHT24,horst2025pontryagin,FLZ25}. 

In this paper, we study a class of linear--quadratic mean-field games with rough common noise. Our goal is to take the common noise to be a rough path $\pmb{\eta}$, build a well-posedness and stability theory in the LQ case, and study the dependence of the equilibrium law on the common signal through the map
\[
    \pmb{\eta}\longmapsto \mu^{\pmb{\eta}}.
\]
Although rough path theory has a substantial history in applications to stochastic optimal control, mean-field dynamics and finance, see, e.g., \cite{Cass2014,DiehlFrizGassiat2017,AC20,bailleul2020,Bailleul2021,Allan2023,FLZ24,CHT24,bank2025rough,friz2025mckeanvlasovequationsroughcommon,horst2025pontryagin}, its use in mean-field games is only beginning. The desirability of bringing these topics together was already pointed out by Carmona and Delarue \cite{carmona2014master}:

\smallskip
{\it ``When analyzed within the probabilistic framework of the stochastic maximum principle, MFGs with a common noise [...] best understood [with] rough paths. Indeed integrals and differentials with respect to the conditioned paths can be interpreted in the sense of rough paths while the meaning of the others can remain in the classical It\^o calculus framework.''}
\smallskip

A sufficiently flexible hybrid rough It\^o theory arguably did not take shape until \cite{FHL21,friz2023rough,bugini2024malliavin}, despite related aspects in rough BSDE and stochastic filtering theory \cite{Diehl2012,CDFO13,DOR15}; see also \cite{allan2025roughsdesrobustfiltering,bugini2025rough}. In the present work, in a setting of linear but multivariate agent dynamics, we prove existence and uniqueness of equilibria and establish stability of the equilibrium with respect to the rough common noise. While our work draws on recent advances in rough and stochastic analysis, its implementation remains surprisingly self-contained due to an original mild, or Volterra-type, formulation of the relevant dynamics.

\subsection{Problem formulation}\label{subsec:theProblem}

Let $W$
be a Brownian motion, assume for the moment that $\pmb{\eta}$ is a (deterministic) {\sl regular} path. In this case the question is if there exist a flow of probability measures $\{\mu_t\}_{t\in[0,T]}$ for which
\begin{enumerate}
    \item[(i)] the representative-agent control problem admits a unique minimizer $\alpha^*$ in a suitable admissible set $\mathcal{A}$ for the functional
    \begin{align*}
        \mathcal{J}(\alpha;\pmb{\eta},\mu)
        :=\mathbb{E}\Big[\int_0^T \ell\big(t,X_t^{\pmb{\eta},\alpha,\mu},\mu_t,\alpha_t\big)\ud t
        +\Phi\big(X_T^{\pmb{\eta},\alpha,\mu},\mu_T\big)\Big],
    \end{align*}
    where $\alpha$ denotes an admissible control and $\xi$ denotes the starting value of the controlled state dynamics 
    \begin{align*}
        X_t^{\pmb{\eta},\alpha,\mu}
        =\xi
        &+\int_0^t b\big(s,X_s^{\pmb{\eta},\alpha,\mu},\mu_s,\alpha_s\big)\ud s
        +\int_0^t \Sigma\big(s,X^{\pmb{\eta},\alpha,\mu},\mu_s,\alpha_s\big)\ud W_s\\
        &+\int_0^t b^1\big(s,X^{\pmb{\eta},\alpha,\mu},\mu_s\big)\ud\pmb{\eta}_s;
    \end{align*}
    \item[(ii)] the mean-field equilibrium condition holds: for all $t\in[0, T]$
    \begin{align*}
        \mu_t=\mathcal{L}\big(X_t^{\pmb{\eta},\alpha^*(\mu),\mu}\big),
        \qquad
        \alpha^*(\mu)\in\operatorname*{argmin}_{\alpha\in\mathcal{A}}
        \mathcal{J}(\alpha;\pmb{\eta},\mu).
    \end{align*}
\end{enumerate}

If for any (deterministic) regular path $\pmb{\eta}$ a mean-field equilibrium $\mu^{\pmb{\eta}}$ exists and is unique, then the mapping 
\[
    \pmb{\eta}\longmapsto \mu^{\pmb{\eta}}
\]
is well-defined and we will call it {\it It\^o–Lions–Lyons map}\footnote{Actually, unless $\pmb{\eta}$ is a true rough path, see \cite[Section 6.2]{friz2020course} and $\xi, A^1, C^1$ are fixed, the It\^o–Lions–Lyons map shall be properly viewed as a function $\mathscr{C}^{0,\gamma}_g([0,T];\mathbb{R}^p) \times  \bigcap_{k=2}^{\infty}\mathbb{L}^k(\mathcal{F}_0,\mathbb{R}^d) \times \mathscr{D}^{\beta,\beta'}_{\eta}\big([0,T];L(\mathbb{R}^d,\mathbb{R}^{d\times p})\big)\times \mathscr{D}^{\beta,\beta'}_{\eta}\big([0,T];L(\mathbb{R}^d,\mathbb{R}^{d\times p})\big)\ni \big(\pmb{\eta}, \xi,({A}^{1},(A^1)'),({C}^{1},(C^1)')\big)\longmapsto \mu\in  \mathcal{W}_{k}(C^\gamma([0,T];\mathbb{R}^d))$ (see Section \ref{sec:Stability} for more details). For simplicity reasons we present it here as a function of the rough common noise alone. Nevertheless, our stability results showcase that the full map is locally Lipschitz continuous.}. In this case it is
a natural question whether this mapping is continuous, that is, if the mean-field equilibrium is robust in ${\pmb{\eta}}$. 

The problem we propose in this article is to allow for $\pmb{\eta}$ to be a genuine rough path, inspired by MFGs with common noise. In a standard MFG with common noise, the common noise represents randomness that affects all agents simultaneously. This leads to a stochastic coupling where both the population distribution and individual strategies evolve together under shared uncertainty. 

In this paper we consider a novel type of common noise MFGs where the representative agent observes the common path at the beginning of the game and establish stability of equilibria w.r.t.~the noise path. It turns out that in such settings one needs to complement the equilibrium condition (ii) with a additional equilibrium condition on the  Gubinelli derivative (see Section \ref{subseec:RPpreliminaries} for more details) on the equilibrium measure flow. Such a condition (which will become more manageable in the LQ setting, cf. \eqref{eq:rConsistency} below) is new and has no counterpart in classical MFG theory.

\subsection{Rough Common Noise in Stress Testing}

A natural interpretation of an initially revealed common rough path arises in
scenario-based stress testing. In a dynamic stress test, the object of interest
is typically not merely the magnitude of an adverse shock, but its entire
evolution over the testing horizon. The timing, persistence, ordering, and
clustering of adverse movements may have substantially different consequences
for a dynamic population, even when two scenarios have the same initial and terminal values or comparable overall magnitudes. It is therefore natural to specify the common stress by a complete trajectory $\pmb{\eta}$ and to study the equilibrium response of the population conditional on this trajectory.

The assumption that the common path is known at the initial time should not be interpreted as an assumption that agents can forecast future common shocks perfectly. Rather, the path is an input to the stress test. It may be obtained from a historical episode, generated by a calibrated stochastic model, or constructed deliberately to represent an adverse scenario. Once a particular scenario has been selected, it is supplied to the agents and the model asks how they would behave if they were allowed to plan optimally against that scenario.

Giving agents the complete stress path also provides a useful benchmark. If a system remains vulnerable even when agents can anticipate the entire adverse trajectory and adapt their decisions optimally, the resulting vulnerability cannot be attributed merely to surprise or forecasting errors. Instead, it reveals a structural feature of the system, such as insufficient capacity, strategic congestion, or an endogenous amplification of the common stress.
Thus, perfect knowledge of the prescribed stress can make the stress test more, rather than less, informative.

The use of a rough path is particularly appropriate when stress scenarios are constructed from high-frequency data or from stochastic models with irregular sample paths. A prescribed trajectory does not become smooth merely because it is known in advance. Rough-path theory allows the stress scenario to retain its fine temporal structure and provides a stable pathwise description of the response of nonlinear state dynamics. In this interpretation, randomness of the rough path describes variation \emph{across stress scenarios}, while additional stochastic drivers may describe uncertainty that remains unresolved \emph{within a given scenario}.

A concrete example is a population of energy-storage operators subjected to a common grid-stress scenario. The rough path $\pmb{\eta}$ may represent a high-frequency trajectory of renewable-generation or grid imbalances, while $W$ captures battery-specific demand, efficiency, or operational shocks. Agents know the prescribed grid scenario and can schedule their charging and discharging decisions accordingly, but they continue to face idiosyncratic uncertainty. Their actions interact through prices, congestion, or other aggregate effects represented by $(m,m')$.

For each prescribed scenario $\pmb{\eta}$, the model therefore produces a scenario-contingent mean field equilibrium
$$
    \pmb{\eta}
    \longmapsto
   \big(\alpha^*,(m^*,(m^*)')\big)^{\pmb{\eta}}.
$$
Comparing these equilibria across different rough stress paths makes it
possible to distinguish the effects of persistent, transient, clustered, or rapidly fluctuating stresses. The model thereby separates two conceptually different sources of uncertainty: 1) variation across common stress scenarios represented by $\pmb{\eta}$, and 2) idiosyncratic uncertainty within a fixed scenario represented by $W$.
This separation is particularly useful in stress testing: the common adverse environment is controlled by the design of the test, whereas individual risks continue to materialize during the simulated stress episode.

\subsection{Setting, main results and ideas of proof} 

In our LQ setting the representative agent controls linear rough stochastic dynamics for the form
\begin{equation}\label{eq rough_affineI}
\begin{aligned}
    X^{\pmb{\eta},\alpha,(m,m')}_t
    &=\xi+\int_0^t\big(A_sX^{\pmb{\eta},\alpha,(m,m')}_s+B_s\alpha_s+C_sm_s\big)\ud s
      +\int_0^t\Sigma_s\ud W_s\\
    &\quad+\int_0^t\big(A_s^1X^{\pmb{\eta},\alpha,(m,m')}_s+C_s^1m_s\big)\ud\pmb{\eta}_s,
\end{aligned}
\end{equation}
where $m$ is a fixed controlled rough path, later related to the mean of the state through \eqref{eq:rConsistency}. The agent aims to minimize the quadratic cost functional
\begin{equation}\label{eq:rough_costI}
\begin{aligned}
    \mathcal{J}^{\pmb{\eta},(m,m')}(\alpha)
    &=\E\!\left[\int_0^T \ell_t^m\big(X_t^{\pmb{\eta},\alpha,(m,m')},\alpha_t\big)\ud t
    +\Phi^m\big(X_T^{\pmb{\eta},\alpha,(m,m')}\big)\right],\\
    \ell_t^m(x,\alpha)
    &:=\frac12\Big(x^\top Q_tx+\alpha^\top R_t\alpha+(x-S_tm_t)^\top\bar Q_t(x - S_t m_t)\Big),\\
    \Phi^m(x)
    &:=\frac12\Big(x^\top Qx+(x - S m_T)^\top\bar Q(x-S m_T)\Big).
\end{aligned}
\end{equation}

As pointed out above, in our setting, the mean-field equilibrium condition does not only reduce to equilibrium condition on the mean flow $m$ but also on its Gubinelli derivative $m'$, a feature not seen in classical theory. The resulting equilibrium condition (also referred to as \textit{consistency condition} in the sequel) is, for all $t\in[0,T]$,
\begin{equation}\label{eq:rConsistency}
    m_t=\mathbb{E}\Big[X_t^{\pmb{\eta},\alpha^*,(m,m')}\Big],
    \qquad
    m_t'=(A_t^1+C_t^1)\mathbb{E}\Big[X_t^{\pmb{\eta},\alpha^*,(m,m')}\Big].
\end{equation}
This is a genuinely rough consistency condition: the equilibrium parameter is a controlled rough path, and the fixed-point condition therefore identifies both the mean flow $m$ and its Gubinelli derivative $m'$.
Indeed, when $\eta$ is smooth, our enhanced equilbirium condition reduces to the classical MFG consistency condition since Gubinelli derivatives carry no additional information.


\begin{defn}[Solution of rough LQ MFG]\label{dfn 1.1}
Let $\mathcal{A}$ denote the space of admissible controls. A pair
\begin{equation*}
\big(\alpha^*,(m^*,(m^*)')\big)
\in
\mathcal{A}
\times
\mathscr{D}^{\beta,\beta'}_{\eta}
\big([0,T];\mathbb{R}^d\big)\footnotemark
\end{equation*}
\footnotetext{See Definition \ref{def:contro-rp} for the detailed introduction of this space.}
is said to be a solution of the rough LQ MFG
\eqref{eq rough_affineI}--\eqref{eq:rConsistency}
if and only if
\begin{equation*}
    \alpha^*
    \in
    \operatorname*{argmin}_{\alpha\in\mathcal{A}}
    \mathcal{J}^{\pmb{\eta},(m^*,(m^*)')}(\alpha),
    \quad
    \text{and}
    \quad
    (m^*,(m^*)')
    \text{ satisfies \eqref{eq:rConsistency}.}
\end{equation*}
\end{defn}

Under natural assumptions on the model parameters we prove that our rough MFG admits a unique solution. As is customary in LQ MFG theory, we solve the fixed-point problem at the level of the mean flow rather than directly in the space of probability measures. Of course, in presence of the appropriate uniqueness properties, the two fixed point problems are equivalent. In the present setting this has two advantages: it yields a natural rough extension of the classical LQ theory, and it avoids the introduction of measure-valued controlled rough paths \cite{bailleul2020}. 

The following informal theorem summarizes our main findings.

\begin{thm}\label{thm:main}
\begin{itemize}
  \item[(i)] Under natural assumptions on the coefficient data we derive the stochastic maximum principle (SMP) of the optimal control problem \eqref{eq rough_affineI}-\eqref{eq:rough_costI} in terms of rough FBSDE.

  \item [(ii)] Under further assumptions on the coefficient data (in agreement with
  classical ones in absence of $\pmb{\eta}$) we have global well-posedness of the
  rough Riccati equation and thus a well-posed rough LQ MFG \eqref{eq rough_affineI}-\eqref{eq:rConsistency}, with explicit
  formulae for the equilibrium measure flow and equilibrium actions. Moreover, the It\^o--Lions--Lyons map $\pmb{\eta}\longmapsto \mu^{\pmb{\eta}}$ is (locally Lipschitz) continuous in rough path metric.
  \item[(iii)] Upon (independent) Brownian rough path randomization of the rough path $\pmb{\eta}$ we solve a MFG with pathwise\footnote{Terminology inspired by pathwise stochastic control, cf. Section 1.3 for references.} common noise.
\end{itemize}
\end{thm}

\medskip
\noindent 
{\bf Ideas of proof.}
The first step in solving our rough LQ MFG is to exploit the affine-linear rough noise structure in \eqref{eq rough_affineI}. Theorem \ref{thm equivalence} shows that the rough term can be absorbed into a multiplicative two-parameter kernel $\mathcal{K}^{\pmb{\eta}}(t,s)$ and a deterministic path $M^{\pmb{\eta}}((m,m'))$, which yields the mild formulation
\begin{equation}\label{eq:auxiliarySVE}
\begin{aligned}
    X_t^{\pmb{\eta},\alpha,m}
    &=\mathcal{K}^{\pmb{\eta}}(t,0)\xi+M^{\pmb{\eta}}\big(m\big)_t\\
    &\quad+\int_0^t\mathcal{K}^{\pmb{\eta}}(t,s)
    \big(A_sX_s^{\pmb{\eta},\alpha,m}+B_s\alpha_s+C_sm_s\big)\ud s
    +\int_0^t\mathcal{K}^{\pmb{\eta}}(t,s)\Sigma_s\ud W_s.
\end{aligned}
\end{equation}
This formulation strikes a balance between a direct rough stochastic analysis of \eqref{eq rough_affineI} and a fully flow-transformed equation, often regarded as non-intrinsic. It allows us to derive the stochastic maximum principle in a Volterra setting, see Theorem \ref{thm Stoch_optimal_control_Volterra}, without rough path considerations at the optimization step. The resulting Volterra FBSDE is shown in Theorem \ref{thm sol_rough_FBSDE} to be equivalent to the corresponding intrinsic rough FBSDE. Combining this rough stochastic maximum principle with the consistency condition \eqref{eq:rConsistency} gives the rough McKean--Vlasov system
\begin{equation}\label{eq 0001}
\left\{
\begin{aligned}
    X_t^{\pmb{\eta}}
    &=\xi+\int_0^t\big(A_s^1X_s^{\pmb{\eta}}+C_s^1\mathbb{E}[X_s^{\pmb{\eta}}]\big)\ud\pmb{\eta}_s\\
    &\quad+\int_0^t\big(A_sX_s^{\pmb{\eta}}-B_sR_s^{-1}B_s^\top Y_s^{\pmb{\eta}}
    +C_s\mathbb{E}[X_s^{\pmb{\eta}}]\big)\ud s
    +\int_0^t\Sigma_s\ud W_s,\\
    Y_t^{\pmb{\eta}}
    &=(Q+\bar Q)X_T^{\pmb{\eta}}-\bar QS\mathbb{E}[X_T^{\pmb{\eta}}]
    +\int_t^T(A_s^1)^\top Y_s^{\pmb{\eta}}\ud\pmb{\eta}_s\\
    &\quad+\int_t^T\big((Q_s+\bar Q_s)X_s^{\pmb{\eta}}
    -\bar Q_sS_s\mathbb{E}[X_s^{\pmb{\eta}}]
    +A_s^\top Y_s^{\pmb{\eta}}\big)\ud s
    -\int_t^T Z_s^{\pmb{\eta}}\ud W_s.
\end{aligned}
\right.
\end{equation}
The precise equivalence between the rough LQ MFG and the rough McKean--Vlasov FBSDE \eqref{eq 0001} is stated in Theorem \ref{thm:VolterraEquivalence}. In order to solve equation \eqref{eq 0001} we employ the fixed point strategy.
By taking expectations in the rough FBSDE and imposing the equilibrium condition \eqref{eq:rConsistency} is shown in Theorem \ref{thm:RMFGfixedPointEquivalence} that this MFG fixed-point problem is reduced to the deterministic forward--backward rough differential equation \eqref{eq:FBRDE}. The latter is solved through a linear decoupling ansatz and new rough Riccati equations. Global non-explosion of these Riccati equations yields the global existence and uniqueness result and hence global well-posedness of the rough LQ MFG, see Theorem \ref{thm:FBRDEwellposedness}. Subsequently, we establish stability of the optimal state, optimal control, value functional, and equilibrium law with respect to the rough common noise and the initial condition. In particular, for the maps
\[
    (\pmb{\eta},\xi)\longmapsto
    \mu^{\pmb{\eta},\xi}
    =\mathcal{L}\big(X_\cdot^{\pmb{\eta},\alpha^*(\pmb{\eta},\xi),\xi}\big),
    \qquad
    (\pmb{\eta},\xi)\longmapsto\alpha^*(\pmb{\eta},\xi),
\]
and the corresponding optimal value, Theorem \ref{thm:StabilityMain} gives local Lipschitz continuity in the rough path metric and the appropriate norm on the initial condition. In particular, the It\^o--Lions--Lyons map $\pmb{\eta}\mapsto\mu^{\pmb{\eta}}$ is locally Lipschitz.

\medskip
\noindent 
{\bf From rough MFG to MFG with pathwise common noise.}
Our rough path approach allows us to study a novel class of conditional mean-field games where the representative agent observes the common noise path at the beginning of the game. 

Specifically, on a product space $\Omega=\Omega'\times\Omega''$, let $W=W(\omega')$ be the idiosyncratic Brownian motion and let $W^\perp=W^\perp(\omega'')$ be an independent common Brownian motion, lifted to a Brownian rough path $\textbf{W}^\perp$. Evaluating the pathwise equilibrium at
\[
    \pmb{\eta}=\textbf{W}^\perp(\omega'')
\]
produces a conditional stochastic MFG with a suitably modified admissibility condition on the controls and a {\sl random} conditional objective function of the form
\begin{equation}\label{eq:rough_costI2}
\begin{aligned}
    \mathcal{J}^{\textbf{W},m}(\alpha)
    &=\E\!\left[\int_0^T \ell_t^m\big(X_t^{\textbf{W},\alpha,m},\alpha_t\big)\ud t
    +\Phi^m\big(X_T^{\textbf{W},\alpha,m}\big) \bigg|\mathcal{F}^{\perp}_T \right].
\end{aligned}
\end{equation}
More precisely, assuming $A^1 , C^1$ are continuous functions of $\pmb{\eta}$ (see also \ref{(S6-R)}), the continuity obtained in Section \ref{sec:Stability} provides the 
necessary measurability properties,
while Theorem \ref{thm 6.7} identifies the pathwise equilibria and equilibria in the conditional stochastic setting. Consequently, there exists a controlled rough path $(\overline{m^*},(\overline{m^*})')$ and an admissible control $\overline{\alpha^*}$ such that   
\begin{multline*}
     \mathbb{E}\left[\int_0^T \ell^{\overline{m^*}}_t\big(\overline{X}^{\overline{\alpha^*}}_t,\overline{\alpha^*}_t\big)\ud t + \Phi^{\overline{m^*}}\big(\overline{X}^{\overline{\alpha^*}}_T\big)\bigg|\mathcal{F}^{\perp}_T\right]\\ = \left(\textrm{inf}_{\alpha\in\mathcal{A}}\mathbb{E}'\!\left[\int_0^T \ell^{m(\pmb{\eta})}_t\big({X}^{\pmb{\eta}}_t,\alpha_t\big)\ud t + \Phi^{m(\pmb{\eta})}\big({X}^{\pmb{\eta}}_T\big)\right]\right)\Bigg|_{\pmb{\eta}=\textbf{W}^{\perp} (\omega'')}
 \end{multline*}
 and such that the equilibrium condition 
\begin{equation*}
 \overline{m^*}_t = \mathbb{E}\Big[ \overline{X}^{\overline{\alpha^*}}_t\Big|\mathcal{F}^{\perp}_T\Big], \qquad (\overline{m^*})'_t = ( A_t^1  + C_t^1)|_{\pmb{\eta}=\textbf{W}^{\perp} (\omega'')}  \mathbb{E}\Big[ \overline{X}^{\overline{\alpha^*}}_t\Big|\mathcal{F}^{\perp}_T\Big]
\end{equation*}
holds in terms of conditional means. 
\subsection{Contributions}

The paper contributes to both MFG theory and rough path theory.
\medskip

\noindent

On the MFG side, together with the independent and contemporaneous work \cite{bayraktar2026meanfieldgamesroughcommon}, this is the first work to formulate an MFG with rough-path common noise, realizing the program suggested in \cite{carmona2014master}. The approaches are complementary: \cite{bayraktar2026meanfieldgamesroughcommon}, which appeared shortly after our first version on the arXiv, uses relaxed solutions and compactness methods to obtain existence, whereas we obtain strong global existence, uniqueness, and stability in the linear--quadratic setting. The equilibrium condition is posed on the controlled mean flow $(m,m')$, rather than directly on a measure flow. The consistency relation \eqref{eq:rConsistency} is a new rough equilibrium condition that has no counterpart in standard MFG theory. It provides a direct extension of classical LQ MFG theory without introducing measure-valued controlled rough paths. Global well-posedness of the rough LQ MFG is reduced to global non-explosion for new rough Riccati equations. This gives a rough counterpart of the Riccati mechanism underlying classical LQ control and LQ mean-field games. We also construct a well-defined, locally Lipschitz It\^o--Lions--Lyons map
    \[
        \pmb{\eta}\longmapsto\mu^{\pmb{\eta}},
    \]
    providing a first quantitative stability result for MFG equilibria with respect to the common noise and the robustness naturally inherited from rough path theory. Finally, under independent Brownian randomization, the pathwise rough MFG becomes a conditional stochastic MFG. This gives the stochastic meaning of the rough formulation and extends pathwise stochastic control \cite{LS98,LS98b,BM07,AC20,CHT24,FLZ24,horst2025pontryagin} from a single-agent optimization problem to a mean-field equilibrium problem.

\medskip

\noindent

On the rough path side, we first  
establish a global well-posedness result for a class of rough McKean--Vlasov FBSDEs. 
These equations arise naturally as the optimality system of the rough MFG and are reduced to a deterministic FBRDE through the mean-flow consistency condition. We use rough FBSDEs to extend the stochastic maximum principle to rough stochastic optimal control. The optimization is first carried out in the Volterra formulation, and the resulting optimality system is then transferred back to the intrinsic rough setting. To solve our MFG we introduce and solve new symmetric and non-symmetric rough Riccati equations, building on related equations from robust Kalman--Bucy filtering \cite{bugini2025rough}. Their global non-explosion is the key to the well-posedness of the FBRDE and the rough LQ MFG. The Volterra-type, or mild, formulation absorbs the $\ud\pmb{\eta}$ term into the two-parameter flow kernel $\mathcal{K}^{\pmb{\eta}}(t,s)$. It balances direct rough stochastic analysis against a fully flow-transformed equation and streamlines the moment conditions that arise in rough It\^o calculus. The same approach should also work for non-linear $b, \Sigma$, under suitable strict convexity assumptions on the Hamiltonian. Our model provides a genuinely hybrid rough--It\^o system, driven by idiosyncratic Brownian noise and rough common noise. The equivalence between the rough and Volterra formulations relies on a rough stochastic product rule. When $\Sigma$ is not invertible the standard conditions for existence of (global, strong) solutions of MFG (in absence of rough drift) via SMP are rather stringent, see e.g.  \cite[Assumption (SMP), p. 276]{carmona2018probabilistic} and \cite[Assumptions (Sufficient SMP in Random Environment), p. 99, (MFG with a Common Noise SMP), p. 210]{carmona2018probabilistic_2}, but very important. Hence, 
by smoothly approximating $\pmb{\eta}$ and absorbing the rough integral in the drift term, one can see that assuming affine dependence on state in the rough integral is naturally aligned with the aforementioned sets of conditions. The affine dependence of the rough coefficient on the state is structurally important: it preserves the convexity of the Hamiltonian, whereas in the non-affine case the uncontrolled sign of ``$\dot\eta$'' generally destroys the global sufficient conditions supplied by the stochastic maximum principle; compare \cite{DiehlFrizGassiat2017}.
\subsection{Structure of the paper}

Section \ref{sec:prelim} collects the notation and the required preliminaries on rough paths, controlled rough paths, and rough stochastic integration. In Section \ref{sec:LQVolterraControl}, we study the auxiliary LQ stochastic control problem with Volterra dynamics \eqref{eq:auxiliarySVE}, derive its stochastic maximum principle, and prove well-posedness of the corresponding Volterra FBSDE. Section \ref{sec 3} connects the Volterra and rough formulations through Duhamel's principle and a rough stochastic product formula. The rough MFG fixed-point problem is then reduced to the deterministic FBRDE \eqref{eq:FBRDE}, which is solved through novel assymetric rough Riccati equations. Section \ref{sec:Stability} establishes local Lipschitz stability of the equilibrium law, optimal control, and value with respect to the rough common noise and initial condition. Section \ref{sec:Randomization} studies independent Brownian randomization and identifies the resulting conditional stochastic MFG. Auxiliary results on linear RDEs, backward rough Riccati equations, and the rough stochastic product rule are collected in Appendix \ref{sec ap.}.

\section{Preliminaries}\label{sec:prelim}
\subsection{Frequently used notation}\label{subsec:FUN} 
For two matrices $A, B$ we write $A\succeq B$ (resp. $A\succ B$) to indicate that $A-B$ is symmetric and non-negative (resp. positive) definite. For an $n-$tensor $A,$ we use $A^\top$ to denote matrix transposes when $n=2$ and (one of the) well-defined pairwise tensor transposes for $n>2.$  By slight abuse of notation, different pairwise tensor transposes will not be distinguished when they can be inferred from the context.

For every pair $n,n'\in\mathbb{N}$ we shall denote by $\mathbb{R}^{n\times n'}$ the space of $n\times n'$-matrices with real entries endowed with its Euclidean norm, \emph{i.e.}, for $A\in\mathbb{R}^{n\times n'}$ $|A|^2:=\textrm{Tr}[AA^\top]$, where $A^{\top}$ is the transpose of $A$. We shall frequently identify $\mathbb{R}^n$ with $\mathbb{R}^{n \times 1}$.

Throughout this work we fix a complete filtered probability space $(\Omega,\mathcal F,(\mathcal F_t)_{t\in\mathbb{R}_+},\mathbb{P})$ which satisfies the usual conditions, $k\in(0,\infty),T\in\mathbb{R}_+, d,q \in\mathbb{N}$ and a Euclidean space $(V,|\cdot|)$. We use ${L}(V_1, V_2)$ to denote the space of linear maps between Euclidean spaces $V_1, V_2.$ The Borel $\sigma-$algebra on a Polish space $\mathcal{X}$ is denoted by $\mathcal{B}(\mathcal{X}).$ For our analysis we are going to need the following spaces:

    \begin{align*}
\mathbb{L}^k(\mathcal{F}_t;V)&:=\Big\{Z:\Omega\rightarrow V:\mathcal{F}_t-\text{measurable and}\hspace{.2cm}\|Z\|_k:=\mathbb{E}\big[|Z|^k\big]^{\frac{1}{k}}<\infty\Big\},\\
\mathbb{L}^\infty(\mathcal{F}_t;V)&:=\Big\{Z:\Omega\rightarrow V:\mathcal{F}_t-\text{measurable and}\hspace{.2cm}\|Z\|_\infty:=\esssup_\omega |Z(\omega) |<\infty\Big\},\\
    \mathbb{S}^k([0,T];V)&:=\Big\{Z: [0,T]\times\Omega\to V: \text{ prog. measurable and}\hspace{.2cm} \|Z\|_{\mathbb{S}^k}:=\mathbb{E}\Big[\sup_{t\in[0,T]}|Z_t|^k\Big]^{\frac{1}{k}}<\infty\Big\},
\end{align*}
and
\begin{align*}\label{eq:bbHspace} \mathbb{H}^k([0,T];\mathbb{R}^{d\times q}):=\bigg\{Z&: [0,T]\times\Omega\to\R^{d\times q}:\text{ prog. measurable and}\\&\|Z\|_{\mathbb{H}^k}:=\mathbb{E}\bigg[\Big(\int_0^T |Z_t|^2\ud t\Big)^{\frac{k}{2}}\bigg]^{\frac{1}{k}}<\infty\bigg\}.
  \end{align*}

\subsection{Rough path theory} \label{subseec:RPpreliminaries}
Let us recall some basic results on rough paths and rough stochastic integrals. We refer to the textbook \cite{friz2020course} for more details on rough analysis. 

In what follows, $\gamma \in (\frac13, \frac12)$ is a fixed H\"older exponent corresponding to the underlying rough path.  
For a one-parameter function/path $\eta:[0,T] \rightarrow V, s,t\in[0,T]$ we set $\delta \eta_{s,t}:= \eta_t - \eta_s.$ The space of $V$-valued $\gamma$-H\"older continuous functions is denoted $C^\gamma([0,T];V)$. For any two-parameter function $A: \Delta_{[s,t]} \rightarrow V,$ where
\[
    \Delta_{[s,t]} := \{(u,v) \in [s, t]^2 ~ | ~ s \leq u < v \leq t\},
\]    
the difference operator $\delta A_{s,u,t}$ is defined by
$$
    \delta A_{s,u,t}:= A_{s,t}-A_{s,u}- A_{u,t},\ \ \ s \leq u \leq t,
$$
and the $\gamma$-H\"older seminorm is denoted
$$
\lvert A \rvert_{\gamma,[s,t]} := \sup_{(u, v) \in \Delta_{[s, t]}} \frac{\lvert A_{u,v} \rvert}{|v - u|^\gamma}.
$$
The set of $\gamma$-H\"older continuous two-parameter functions is denoted $C^\gamma_2([0,T];V)$. For a random two-parameter $V$-valued function $A: \Omega \times \Delta_{[s,t]} \rightarrow V$ we set 
$$
\lVert A \rVert_{\gamma, k,[s,t]} := \sup_{(u,v) \in \Delta_{[s, t]}} \frac{\lVert A_{u, v} \rVert_k}{\lvert v - u \rvert^\gamma}.
$$ 

The space of all functions for which the above seminorm is finite is denoted $C^\gamma_2([0,T];\mathbb{L}^k)$ and $C^\gamma_2 \mathbb{L}^k$ denotes the subspace of all $\mathcal{F} \otimes \mathcal{B}(\Delta_{[0,T]})/\mathcal{B}(V)$-measurable elements of $C^\gamma_2([0,T];\mathbb{L}^k )$. 

In the following, we usually omit the time interval in our notations and norms if $[s,t]=[0,T]$ and write `$\lesssim_{\theta}$' to indicate that an inequality holds up to a generic constant that depends on a parameter $\theta$.

\begin{rmk}
For every $\gamma\in(0,1]$ and Banach space $V$, the space $C^{\gamma}([0,T];V)$ equipped with the norm $\|\cdot\|_{\gamma}:= \|\cdot\|_{\infty; [0,T]}+|\cdot|_{\gamma,[0,T]}$ is a Banach space.
\end{rmk}

\begin{defn} 

\begin{itemize}
\item[$(1)$] A pair of (deterministic) continuous functions $\pmb{\eta}:=(\eta, \eta^{(2)}):[0,T] \times \Delta_{[0,T]} \rightarrow \R^p \times \R^{p \times p}$ is called a $\gamma$-H\"older continuous {\it rough path}, if the following two conditions are satisfied.
\[
(\text{i})\  |\delta \eta |_{\gamma;[0,T]},\ |\eta^{(2)} |_{2\gamma;[0,T]} < \infty; \ \ (\text{ii})\  \eta^{(2)}_{s,t}- \eta^{(2)}_{s,u}- \eta^{(2)}_{u,t}= \delta \eta_{s,u} \otimes \delta \eta_{u,t},\ \text{ for any $s\leq u \leq t$.}
\]  
The space of $\gamma$-H\"older continuous rough paths is denoted  $\mathscr{C}^{\gamma}([0,T];\R^p)$ and is equipped with the following rough path norm and rough path metric: for any $\pmb{\eta}=(\eta,\eta^{(2)}), \bar{\pmb{\eta}}=(\bar \eta, {\bar \eta}^{(2)} ) \in \mathscr{C}^\gamma([0,T];\R^p),$ 
$$
\tnorm{\pmb{\eta}}_\gamma :=|\eta_0|+ \lvert \delta \eta \rvert_{\gamma;[0,T]} +\lvert \eta^{(2)}\rvert_{2 \gamma;[0,T]}^{\frac{1}{2}} , \ \ \ 
\rho_{\gamma}(\pmb{\eta}, \bar{\pmb{ \eta}}) := |\eta_0-\bar{\eta}_0|+\lvert \delta \eta - \delta \bar \eta \rvert_{\gamma;[0,T]} + \lvert \eta^{(2)} - \bar \eta^{(2)} \rvert_{2\gamma;[0,T]}.
$$

\item[$(2)$] $\mathscr{C}^{0,\gamma}_g([0,T];\R^p)$ is the space of $\gamma$-H\"older geometric rough paths, defined as the completion of canonically lifted smooth paths (i.e. if $\eta$ is smooth, $\pmb{\eta}:=(\eta, \int \eta  d\eta)$ is the canonical lift of $\eta$) under the rough path metric $\rho_\gamma$. Note that $\mathscr{C}^{0,\gamma}_g([0,T];\R^p)$ is Polish and $\mathscr{C}^{0,\gamma}_g([0,T];\R^p) \subset \mathscr{C}^{\gamma}([0,T];\R^p)$ (see e.g. \cite{friz2020course} for further details).

\end{itemize}
\end{defn}

\begin{rmk}
We shall always endow $\mathscr{C}^{0,\gamma}_g([0,T];\R^p)$with the rough path metric $\rho_\gamma$. Moreover, we 
may write $\rho_\gamma(\pmb{\eta})$ instead of $\rho_\gamma(\pmb{\eta},0)$.
\end{rmk}

To introduce stochastic controlled rough paths and to define rough integrals, we introduce for any integrable random two-parameter process $A$ the two-parameter conditional expectations process $\E_\bullet A$ by
$$
(s,t;\omega)\mapsto \E[A_{s,t}|\mathcal{F}_s](\omega) =: \E_s[A_{s,t}](\omega).
$$

\begin{defn}[{\cite[Definition 3.1]{FHL21}}]\label{def:contro-rp} 
Let $\pmb{\eta} \in \mathscr{C}^{\gamma}([0,T];\R^p)$ be a rough path.
\begin{itemize}

\item[$(1)$.] 
A pair \((Z,Z')\) is called a \textit{$\eta$-stochastic controlled rough path} of $k$-integrability and $(\beta,\beta')$-H\"older regularity with $\beta, \beta ' \in (0,\gamma],$ if the following hold: 

\begin{enumerate}  
\item[$(\text{i})$] The process \(Z=(Z_t)_{t \in [0,T]}\) is $V$-valued, progressively measurable and 
\[
      \|Z_0\|_k < \infty 
    \quad \text{and} \quad
    \|\delta Z\|_{\beta,k} := \sup_{0 \leq s < t \leq T} \frac{ \|\delta Z_{s,t} \|_k }{|t-s|^\beta} < \infty.
\]
\item[$(\text{ii})$] The process \(Z'=(Z'_t)_{t \in [0,T]}\) takes values in ${L}(\R^{p},V)$ (throughout this work the codomain is finite-dimensional and thus we will not distinguish operator norms in the notation), it is progressively measurable and
\[
      \|Z'_0\|_k < \infty \quad \text{and} \quad \|\delta Z'\|_{\beta' ,k} := \sup_{0 \leq s < t \leq T} \frac{ \| \delta Z'_{s,t} \|_k }{|t-s|^{ \beta'}} < \infty.
\]
\item[$(\text{iii})$] The two-parameter process 
\[
    R^Z_{s,t} := \delta Z_{s,t} - Z'_s \delta \eta_{s,t}, \quad (s,t)\in\Delta_{[0,T]}
\]
satisfies
\[
    \|\mathbb{E}_\bullet R^Z\|_{\beta+\beta', k } := \sup_{0 \leq s < t \leq T} \frac{\|\mathbb{E}_s R^Z_{s,t} \|_k}{|t-s|^{\beta+\beta'}} < + \infty.
\]
\end{enumerate}
In this case we write $(Z,Z') \in \textbf{D}_{\eta}^{\beta,\beta'} \mathbb{L}^{k}([0,T];V)$ and  $(Z,Z') \in \mathbf{D}_{\eta}^{2\beta} \mathbb{L}^{k }([0,T];V)$ if $\beta = \beta'$. Throughout the text, we refer to the process $Z'$ as the \textit{(generalized) Gubinelli derivative} of $Z$ even though it is not unique in general. We equip the space $\mathbf{D}_{\eta}^{\beta,\beta'} L^{k}$ with the norms 
\begin{align*}
\llbracket (Z, Z')\rrbracket_{\eta, \beta,\beta', k} &:=  \left \lVert Z_0 \right \rVert_{k}+\lVert \delta Z \rVert_{\beta, k} +   \left \lVert Z'_0 \right \rVert_{k} + \lVert \delta Z' \rVert_{\beta', k} + \lVert \E_\bullet R^Z \rVert_{\beta + \beta', k},\\
\|(Z, Z')\|_{\eta, \beta,\beta', k} &:=  \left \lVert Z_0 \right \rVert_{k}+\lVert \delta Z \rVert_{\beta, k} +   \left \lVert Z'_0 \right \rVert_{k} + \lVert \delta Z' \rVert_{\beta', k} + \lVert R^Z \rVert_{\beta + \beta', k}.
\end{align*}
Moreover, for any $\pmb{\eta}, \bar{\pmb{\eta}} \in \mathscr{C}^\gamma,$ and pairs $(Z,Z') \in \textbf{D}^{\beta, \beta'}_{  \eta} \mathbb{L}^{k },$ $(\bar Z, \bar Z') \in  \textbf{D}^{\beta, \beta'}_{\bar{\eta}} \mathbb{L}^{k }$ we define the distance
    \begin{gather*}
            \llbracket (Z,Z'); (\bar Z, \bar{Z}') \rrbracket_{\eta, \bar{\eta}, \beta, \beta',k} :=   \|\delta (Z - \bar Z)\|_{\beta,k } + \|\delta (Z' - \bar Z')\|_{\beta',k}
            + \|Z_0- \bar Z_0\|_k 
            \\\hspace{5cm}
             + \|Z'_0- \bar Z'_0\|_k 
             + \|\mathbb{E}_\bullet R^Z- \mathbb{E}_\bullet \bar R^{\bar Z}\|_{\beta+ \beta',k},
    \end{gather*}
    where $$\bar R^{\bar Z}_{s,t}:= \delta \bar Z_{s,t} - \bar Z'_s \delta \bar \eta_{s,t}, \quad s,t \in \Delta_{[0,T]}.$$

\item[$(2)$] For any pair $(Z,Z') \in \mathbf{D}_{\eta}^{\beta,\beta'} \mathbb{L}^{k},$ if $R^{Z} \in C^{\beta+\beta'}_2 \mathbb{L}^k$, we write $(Z,Z') \in \mathscr{D}^{\beta,\beta'}_{\eta} \mathbb{L}^k$, and define for any $(\bar{Z}, \bar{Z}') \in \mathscr{D}_{\bar{\eta}}^{\beta,\beta'} \mathbb{L}^{k}$, the distance
\begin{align*}
    \| (Z,Z'); (\bar Z, \bar{Z}') \|_{\eta, \bar{\eta}, \beta, \beta',k}  & ~ := \|\delta (Z - \bar Z)\|_{\beta,k } + \|\delta (Z' - \bar Z')\|_{\beta',k } \\ & \quad \quad + \|Z_0- \bar Z_0\|_k + \|Z'_0- \bar Z'_0\|_k + \|  R^Z-   \bar R^{\bar Z}\|_{\beta+ \beta',k}.
\end{align*}
\end{itemize}
\end{defn} 

\begin{rmk}
We emphasize that in the case where $(Z,Z')\in \mathbf{D}_{{\eta}}^{\beta,\beta'} \mathbb{L}^{k}$ is deterministic then immediately we get $(Z,Z')\in \mathscr{D}_{\eta}^{\beta,\beta'} \mathbb{L}^{\infty}.$ The space of deterministic controlled rough paths is denoted by $\mathscr{D}_{\eta}^{\beta,\beta'}$, or $\mathscr{D}_{\eta}^{\beta,\beta'}([0,T];V)$ when we need to emphasize the codomain $V.$ Furthermore, in the deterministic setting we abbreviate the notation and we simply write $ \| (\cdot,\cdot) \|_{\eta, \beta, \beta'}, \| (\cdot,\cdot); (\cdot, \cdot) \|_{\eta, \bar{\eta}, \beta, \beta'}$.
\end{rmk}

Since the coefficients $A^1, C^1$ in the state dynamics \eqref{eq rough_affineI} are controlled rough paths, we frequently need to deal with products of (stochastic) controlled rough paths. The following lemma is implicitly used in several places and establishes useful estimates for such products. The proof follows immediately from the definitions.

\begin{lemma}\label{productest}
    Let $(Z,Z') \in \textbf{D}^{\beta,\beta'}_{\eta} \mathbb{L}^{\infty} $ and $(z,z') \in \textbf{D}^{\beta,\beta' }_{\eta} \mathbb{L}^{k}$ for some $k \in [2,\infty]$. Then 
    $$
    (Zz,(Zz)' ):= (Zz, Z'z+Zz') \in \textbf{D}^{\beta,\beta' }_{\eta} \mathbb{L}^{k}.
    $$
    Moreover, if
    $$
        \esssup_{t,\omega}|Z_t| \vee   \|(Z,Z')\|_{{\eta}, \beta,\beta',\infty}   \leq K
        \quad \mbox{and} \quad \|z_0\|_k \vee \|z'_0\|_k \leq K_0
    $$
    then, for some implied constants that depend increasingly on $T$, 
\begin{align*}\label{Fy}
    \|Zz\|_{ \beta,k} &\lesssim K\|z\|_{ \beta,k} + KK_0,\\ 
   \|(Zz)'\|_{ \beta',k} &\lesssim K(\|z\|_{ \beta,k}+\|z'\|_{ \beta',k}   + K_0), \\  
  \|   R^{Zz}\|_{ \beta ,k} &\lesssim K (\|z\|_{ \beta,k}+\|z'\|_{ \beta',k}  +  K_0)(1\vee \|  \eta\|_{\beta}), \\ 
    \| \E_{\bullet} R^{Zz}\|_{ \beta+\beta',k} &\lesssim K (\|z\|_{ \beta,k}+ \| \E_{\bullet} R^z\|_{ \beta+\beta',k}  + K_0),
\end{align*}
where $a \vee b:= \max{(a,b)}$.
\end{lemma}

Having introduced rough paths, we proceed to introduce rough stochastic integrals. 

\begin{prop}[{\cite[Theorem 3.4]{FHL21}}] \label{prop:roughinteg}
Let $p\in\mathbb{N},$ $\pmb{\eta} \in \mathscr{C}^{\gamma}([0,T];\R^p)$ a rough path, with $\gamma\in(\frac{1}{3},\frac{1}{2})$ and
 $0<\beta'\leq \beta\leq\gamma$ such that $\gamma+\beta+\beta'>1$. For any $k \in [2,\infty )$ and any $\pmb{\eta}$-stochastic controlled rough path $(Z,Z') \in \textbf{D}_{\eta}^{\beta,\beta'}\mathbb{L}^k$, the following rough stochastic integral is well-defined:
\begin{equation*}
  \int_0^t Z_s \ud\pmb{\eta}_s :=  \int_0^t (Z_s,Z'_s) \ud\pmb{\eta}_s :=  \lim_{|\Pi| \rightarrow 0} \sum_{[u,v] \in \Pi} ( Z_u \delta \eta_{u,v} + Z'_u \eta^{(2)}_{u,v}), \quad t\in[0,T].
\end{equation*}
Here, $\Pi$ is any partition of $[0,t]$ with mesh $|\Pi|$ and the limit is taken in the sense of convergence in probability. Moreover, for any $(s,t) \in \Delta,$ we have
\begin{align*} 
    &\| \int_s^t Z_r \ud\pmb{\eta}_r - Z_s \delta \eta_{s,t} \|_k \lesssim    (\tnorm{\pmb{\eta}}_\gamma \| \delta Z \|_{\beta,k} + \tnorm{\pmb{\eta}}_\gamma^2 \sup_{r\in [s,t]} \|Z'\|_k )|t-s|^{\gamma+\beta} \\  
    & \qquad \qquad \qquad \qquad \qquad \ \ \  +   ( \tnorm{\pmb{\eta}}_\gamma \|\E_{\bullet} R^Z\|_{\beta+\beta',k} + \tnorm{\pmb{\eta}}_\gamma^2 \| \E_{\bullet} \delta Z'\|_{\beta',k}) |t-s|^{\gamma+\beta+\beta'},\\  
   & \lVert \mathbb{E}_{s}  ( \int_{s}^{t} Z_{r} \ud\pmb{\eta}_r - Z_{s} \delta \eta_{s,t} - Z'_s \eta^{(2)}_{s,t}  ) \rVert_{k} 
    \lesssim   \left(  \tnorm{\pmb{\eta}}_\gamma \|\E_{\bullet} R^{Z} \|_{ {\beta}+ \beta', k} + \tnorm{\pmb{\eta}}_\gamma^2 \| \E_{\bullet} \delta Z' \|_{\beta',k} \right) |t-s|^{\gamma +   \beta + \beta'}.
\end{align*}
In particular, 
$$
    (X,X'):=\Big(\int_0^. Z_s \ud\pmb{\eta}_s, Z\Big) \in \textbf{D}^{\gamma, \beta}_{\eta} \mathbb{L}^k.
$$
\end{prop}

Due to our hybrid rough stochastic setting, we interpret the state equation \eqref{eq rough_affineI} as a Rough Stochastic Differential Equation (RSDE) per \cite{FHL21}. We now define what we mean by solutions of \eqref{eq rough_affineI}.

\begin{defn}\label{def:solution}
Let $p\in\mathbb{N}, k\in[2,\infty),$ $\gamma\in(\frac{1}{3},\frac{1}{2})$ and 
 $0<\beta'\leq \beta\leq\gamma$ such that $\gamma+\beta+\beta'>1$,
and $\pmb{\eta} \in \mathscr{C}^{0,\gamma}_g([0,T];\mathbb{R}^p)$. Additionally, let $A,B, C:[0,T]\rightarrow \mathbb{R}^{d\times d}$ bounded and Borel measurable on \([0,T]\), $\big({A}^{1},(A^1)'\big), \big({C}^{1},(C^1)'\big)\in \mathscr{D}^{\beta,\beta'}_{\eta}\big([0,T];L(\mathbb{R}^d,\mathbb{R}^{d\times p})\big), (m,m')\in \mathscr{D}^{\beta,\beta'}_{\eta}\big([0,T];\mathbb{R}^d\big)$ deterministic controlled rough paths with respect to $\eta$, $\xi\in \mathbb{L}^k(\mathcal{F}_0,\mathbb{R}^d)$, and $\alpha(\omega,t), \Sigma(\omega,t)$ progressively measurable processes such that $\sup_{t\in[0,T]}\mathbb{E}[|\Sigma_t|^k] + \sup_{t\in[0,T]}\mathbb{E}[|\alpha_t|^k]<\infty$.
 We shall call a continuous adapted process $X^{\pmb{\eta},\alpha,(m,m'),\xi}$ an $k$-integrable solution to the RSDE \eqref{eq rough_affineI} in the space $\textbf{D}_{\eta}^{\beta,\beta'}\mathbb{L}^k([0,T]; \mathbb{R}^{d})$, if the following holds:

\begin{enumerate}
\item[(i)] The integral $\int_{0}^t\big(A_s X^{\pmb{\eta},\alpha,(m,m'),\xi}_s + B_s \alpha_s + C_s m_s \big)\ud s$ is finite a.s..

\item[(ii)] We have
\begin{gather*}
\big(A^1 {X}^{\pmb{\eta},\alpha,(m,m'),\xi} + C^1 m, (A^1)' {X}^{\pmb{\eta},\alpha,(m,m'),\xi} + A^1(A^1 {X}^{\pmb{\eta},\alpha,(m,m'),\xi} + C^1 m)+ (C^1)' m + C^1 m'\big)\\\in  \textbf{D}^{\beta,\beta'}_{\eta}\mathbb{L}^k\big([0,T];\mathbb{R}^d\big).
\end{gather*}
\item[(iii)]  The following integral equation holds $\mathbb{P}-$a.s.~for all $t \in [0,T]$:
\begin{equation*}
\begin{aligned}
    X^{\pmb{\eta},\alpha,(m,m'),\xi}_t&=\xi+\int_{0}^t\big(A_s X^{\pmb{\eta},\alpha,(m,m'),\xi}_s + B_s \alpha_s + C_s m_s \big)\ud s +\int_0^t \Sigma_s\ud W_s\\&+\int_{0}^t \big(A^1_s X^{\pmb{\eta},\alpha,(m,m'),\xi}_s + C^1_s m_s\big)\ud\pmb{\eta}_s.
\end{aligned}
\end{equation*}
\end{enumerate}
\end{defn}

The next result is from \cite[Thm 3.11]{bugini2024malliavin} in case $\beta'=\beta$, and \cite[Thm 3.5]{horst2025pontryagin} otherwise. It establishes well-posedness of \eqref{eq rough_affineI}.
\begin{thm}\label{thm sol_rough_affine}
Let $p\in\mathbb{N},$ $\gamma\in(\frac{1}{3},\frac{1}{2})$ and 
 $0<\beta'\leq \beta\leq\gamma$ such that $\gamma+\beta+\beta'>1$,
and $\pmb{\eta} \in \mathscr{C}^{0,\gamma}_g([0,T];\mathbb{R}^p)$. Additionally, let $A,B, C:[0,T]\rightarrow \mathbb{R}^{d\times d}$ bounded and Borel measurable on \([0,T]\), $\big({A}^{1},(A^1)'\big), \big({C}^{1},(C^1)'\big)\in \mathscr{D}^{\beta,\beta'}_{\eta}\big([0,T];L(\mathbb{R}^d,\mathbb{R}^{d\times p})\big), (m,m')\in \mathscr{D}^{\beta,\beta'}_{\eta}\big([0,T];\mathbb{R}^d\big)$ deterministic controlled rough paths with respect to ${\eta}$, $\xi\in \bigcap_{k=2}^{\infty}\mathbb{L}^k(\mathcal{F}_0,\mathbb{R}^d)$, and $\alpha(\omega,t), \Sigma(\omega,t)$ progressively measurable processes such that $\sup_{t\in[0,T]}\mathbb{E}[|\Sigma_t|^k] + \sup_{t\in[0,T]}\mathbb{E}[|\alpha_t|^k]<\infty$ for every $k\in[2,\infty)$. Then, there exists a unique solution $X^{\pmb{\eta},\alpha,(m,m'),\xi}$ of
\begin{equation*}
\begin{aligned}
    X^{\pmb{\eta},\alpha,(m,m'),\xi}_t&=\xi+\int_{0}^t\big(A_s X^{\pmb{\eta},\alpha,(m,m'),\xi}_s + B_s \alpha_s + C_s m_s \big)\ud s +\int_0^t \Sigma_s\ud W_s\\&+\int_{0}^t \big(A^1_s X^{\pmb{\eta},\alpha,(m,m'),\xi}_s + C^1_s m_s\big)\ud\pmb{\eta}_s
\end{aligned}
\end{equation*}
with $(X^{\pmb{\eta},\alpha,(m,m'),\xi},A^1X^{\pmb{\eta},\alpha,(m,m'),\xi}+C^1m)\in \bigcap_{k =2}^\infty\textbf{D}^{\gamma,\beta}_{\eta}\mathbb{L}^k\big([0,T];\mathbb{R}^d\big)$.
\end{thm}

\section{A linear-quadratric  Volterra stochastic control problem }\label{sec:LQVolterraControl}
This section is devoted to the study of an LQ stochastic control problem with stochastic Volterra state dynamics. The Volterra kernel $\mathcal{K}$ is regular and chosen to satisfy a certain flow property which appears naturally in the rough LQ MFG setting of Section \ref{sec 3}; see also Remark \ref{rem:Kernelchoice} below. Throughout this section, we treat the flow $m$ as a parameter i.e. a fixed continuous path.
\subsection{Assumptions, state dynamics and cost functional}\label{subsection: SVEcontrolIntro}
From this point on (and throughout the rest of the paper), we assume that \((\Omega,\mathcal F,(\mathcal F_t)_{t\in\mathbb{R}_+},\mathbb{P})\) is the usual augmentation of a filtration generated by a $q-$dimensional Brownian motion $W$ and a $\sigma$-algebra $\mathcal{F}_0^0$ independent of $W$\footnote{This assumption is made to avoid technicalities with martingale representation arguments that are used in the derivation of the Volterra FBSDEs.}.

For $\kappa\in\mathbb{N},$ 
$\widehat{\mathcal A}:=\mathbb{H}^2([0,T];\mathbb{R}^\kappa)$ 
denotes the admissible control space, equipped with the norm $\|\cdot\|_{\mathbb{H}^2}.$
For every $\alpha\in\widehat{\mathcal{A}}, m\in C([0,T];\mathbb{R}^d)$ and $\xi\in\mathbb{L}^2(\mathcal{F}_0;\mathbb{R}^d)$, a representative agent controls a process $X^{\alpha,m,\xi}:\Omega\times[0,T] \rightarrow \R^d$ evolving as
\begin{align}\label{eq:rough_state_2}
   {X}^{\alpha,m,\xi}_t = \mathcal{K}(t,0)\xi+M(m)_t + \int_0^t \mathcal{K}(t,s)\big(A_s {X}^{\alpha,m,\xi}_s + B_s \alpha_s + C_s m_s\big)\,\ud s
   + \int_0^t \mathcal{K}(t,s)\Sigma_s\ud W_s.
\end{align}

For a fixed path $m\in C([0,T];\mathbb{R}^d)$ and $\xi\in\mathbb{L}^2(\mathcal{F}_0;\mathbb{R}^d)$, the agent seeks to minimize the quadratic cost functional 
\begin{align}\label{eq:rough_cost}
 \mathcal{J}^{m,\xi}(\alpha) = \E\!\left[\int_0^T \ell^m_t(X_t^{\alpha,m,\xi},\alpha_t)\ud t + \Phi^m(X_T^{\alpha,m,\xi})\right],
\end{align}
over all admissible controls $\alpha\in\widehat{\mathcal{A}}$ with $\ell^m, \Phi^m$ as in \eqref{eq:rough_costI}.
Notice that this cost functional coincides with the original functional \eqref{eq:rough_costI} with the only difference that the latter also depends on an underlying rough path via the rough dynamics.

From now on and for the rest of this section, we will work under the following assumptions:
\begin{enumerate}[label=(S\arabic*), ref=S\arabic*]
\item\label{S1} We are given $\mathcal{K}:[0,T]\times[0,T]\rightarrow\mathbb{R}^{d\times d}$ with the following properties:
\begin{enumerate}
    \item[(i)] For every $t\in[0,T]$ we have $\mathcal{K}(t,t)=\text{Id}_{d\times d}$.
    
    \item[(ii)] For every $s,u,t\in[0,T]$ we have $\mathcal{K}(t,s)=\mathcal{K}(t,u)\mathcal{K}(u,s)$.

    \item[(iii)] The process $\mathcal{K}(t,0):[0,T]\rightarrow \mathbb{R}^{d\times d}$ is continuous.
\end{enumerate}

\item\label{S2} Next, let $A, C:[0,T]\rightarrow \mathbb{R}^{d\times d}, B:[0,T]\rightarrow \mathbb{R}^{d\times \kappa}$ continuous on \([0,T]\), and $M:C([0,T];\mathbb{R}^d)\rightarrow C([0,T];\mathbb{R}^d)$. Furthermore, let us also fix a $\Sigma\in \mathbb{H}^2([0,T];\mathbb{R}^{d\times q})$.

\item\label{S3} Finally, let $ Q,\bar{Q}\in\mathbb{R}^{d\times d}$ with $Q, \bar Q\succeq 0$, and $S_t,Q_t,\bar Q_t,:[0,T]\rightarrow \mathbb{R}^{d\times d}, R_t:[0,T]\rightarrow \mathbb{R}^{\kappa\times \kappa},$ continuous on \([0,T]\), with
\(R_t\succ \lambda \text{Id}_{\kappa\times \kappa}\), for some $\lambda >0$, and \(Q_t,\bar Q_t\succeq 0\).
\end{enumerate}

\begin{rmk}\label{rmk 3.1}
    From (\ref{S1})(i),(ii) it follows that $\mathcal{K}(t,0)^{-1}=\mathcal{K}(0,t)$, for every $t\in[0,T]$.
\end{rmk}

\begin{rmk}\label{rmk 3.2}
We define $\widehat{X}^{\alpha,m,\xi}_t:=  \mathcal{K}(0,t)\big({X}^{\alpha,m,\xi}_t-M(m)_t\big), t\in[0,T]$. Then equation \eqref{eq:rough_state_2} becomes equivalent to 
\begin{equation*}
    \widehat{X}^{\alpha,m,\xi}_t = \xi + \int_0^t \mathcal{K}(0,s)\big(A_s \mathcal{K}(s,0)\widehat{X}^{\alpha,m,\xi}_s + B_s \alpha_s + A_s M(m)_s+ C_s m_s\big)\,\ud s
   + \int_0^t \mathcal{K}(0,s)\Sigma_s\ud W_s.
\end{equation*}
The above is a classical forward affine SDE so we know that it has a unique (global) solution $\widehat{X}^{\alpha,m,\xi}\in  \mathbb{S}^2([0,T];\mathbb{R}^d)$, see for example \cite[Theorem 4.21]{carmona2018probabilistic} for an even stronger result. From this we deduce that \eqref{eq:rough_state_2} has a unique solution ${X}^{\alpha,m,\xi}=\mathcal{K}(t,0)(\widehat{X}^{\alpha,m,\xi}+M(m)) \in  \mathbb{S}^2([0,T];\mathbb{R}^d)$, for every $\alpha\in\widehat{\mathcal{A}}, m\in C([0,T];\mathbb{R}^d)$ and $\xi\in\mathbb{L}^2(\mathcal{F}_0;\mathbb{R}^d)$.
\end{rmk}

\begin{rmk} Linear-quadratic stochastic optimal control problems for more general classes of stochastic Volterra processes (but without a mean-field component) have been considered e.g. in \cite{pham2009continuous}. The approach followed therein involves so called Markovian lifts of Volterra processes to infinite-dimensional spaces; such techniques arise in the setting of singular (matrix-valued) kernels (see e.g. \cite{gasteratos2025kolmogorov} and references therein) and allow access to classical Markovian tools. In contrast, the kernels $\mathcal{K}$ appearing in our Volterra formulation of the rough MFG are quite well-behaved. Indeed, the flow property (\ref{S1}) guarantees that $\mathcal{K}$ is finite on the diagonal and $\mathcal{K}(0,t)$ is an invertible matrix for each $t.$ This, in turn, allows us to solve the Volterra LQ problem by transforming the state dynamics into a classical (semimartingale-type) form and without resorting to infinite-dimensional arguments. 
\end{rmk}

\subsection{The stochastic maximum principle}
Let us introduce the following coupled forward-backward Volterra system of differential equations
\begin{equation}\label{stoch_max_principle Volterra FBSDE}
 \left\{\begin{aligned}
{X}^{m,\xi}_t & ~ = \mathcal{K}(t,0)\xi+M(m)_t \\&\quad\quad+ \int_0^t \mathcal{K}(t,s)\big(A_s {X}^{m,\xi}_s - B_s R_s^{-1} B_s^\top Y^{m,\xi}_s + C_s m_s\big)\,\ud s
   + \int_0^t \mathcal{K}(t,s)\Sigma_s\ud W_s,
   \\
Y^{m,\xi}_t
 & ~ = \mathcal{K}(T,t)^\top\big((Q+\bar Q) X_T^{m,\xi} - \bar Q Sm_T\big)  
 \\
    &\quad\quad+\int_t^T \mathcal{K}(s,t)^\top\Big((Q_s+\bar Q_s) X_s^{m,\xi} - \bar Q_s S_s m_s+A_s^{\top} Y^{m,\xi}_s\Big) \ud s
    \\
   & \quad \quad -\int_t^T \mathcal{K}(s,t)^\top Z^{m,\xi}_s\ud W_s.
\end{aligned}\right.
\end{equation}

The main result  of this subsection, a stochastic maximum principle expressed by \eqref{stoch_max_principle Volterra FBSDE}, is given in Theorem \ref{thm Stoch_optimal_control_Volterra} below. Although the approach is classical we decided to include its proof for completeness and more importantly to showcase exactly how the Volterra formulation allows us to bypass technical difficulties, as it is explained in more detail in Remark \ref{rmk> 3.8} below.

\begin{thm}\label{thm Stoch_optimal_control_Volterra}
    Let (\ref{S1}), (\ref{S2}), (\ref{S3}) hold true and assume the Volterra FBSDE \eqref{stoch_max_principle Volterra FBSDE} has a unique solution $(X^{m,\xi},Y^{m,\xi},Z^{m,\xi}) \in \mathbb{S}^2([0,T];\mathbb{R}^d)\times \mathbb{S}^2([0,T];\mathbb{R}^d)\times \mathbb{H}^2([0,T];\mathbb{R}^{d\times q})$. Then, for every fixed $m\in C([0,T];\mathbb{R}^d)$ and $\xi\in \mathbb{L}^2(\mathcal{F}_0;\mathbb{R}^d)$, the linear-quadratic stochastic optimal control problem described by \eqref{eq:rough_state_2}, \eqref{eq:rough_cost} has a unique solution $(\alpha^*,X^{\alpha^*,m,\xi})$, where $X^{\alpha^*,m,\xi}=X^{m,\xi}$ and the optimal control is explicitly given by the formula
    \begin{equation*}
        \alpha_t^* = - R_t^{-1} B_t^\top Y^{m,\xi}_t,
    \end{equation*}
where $Y^{m,\xi}$ is the solution of the backward component of \eqref{stoch_max_principle Volterra FBSDE}.
\end{thm}
This result can be viewed as a generalization of the standard SMP in the LQ setting, see for example \cite[eq. 3.49, eq. 3.50, p. 181]{carmona2018probabilistic}.
\begin{proof}
\noindent \textbf{Step 1. (Gateaux differentiability of cost function)}
Let $\alpha,h\in\widehat{\mathcal{A}}$, as a first step we are going to show that $\mathcal{J}^{m,\xi}$ is Gâteaux differentiable in the direction of $h$ at point $a$ denoted by $(\delta \mathcal{J})^{m,\xi,h}(\alpha)$, and compute its formula.
To this end,  let $X,H\in \mathbb{S}^2([0,T];\mathbb{R}^d)$, 
from (\ref{S3}) we have
\begin{equation*}
\begin{aligned}
    (X_t+H_t)^\top Q_t(X_t+H_t) - X_t^\top Q_tX_t\
    &= X_t^\top Q_t H_t+H_t^\top Q_tX_t+
    H_t^\top Q_t H_t\\
    &=2  H_t^\top Q_t X_t+ H_t^\top Q_t H_t
    \end{aligned}
\end{equation*}
and similarly
\begin{equation*}
\begin{aligned}
    (\alpha_t+h_t)R_t(\alpha_t+h_t)- \alpha_tR_t\alpha_t &= 2  h_t^\top R_t \alpha_t+ h_t^\top R_t h_t,\\
    (X_t+H_t-S_t m_t)\bar{Q}_t(X_t+H_t-S_t m_t)&-(X_t-S_t m_t)^\top\bar{Q}_t(X_t-S_t m_t)\\
    &=2H^\top\bar{Q}_t (X_t-S_t m_t) + H_t^\top \bar{Q}_t H_t.
    \end{aligned}
\end{equation*}
Hence, with the same computations for $\Phi^m$, we get
\begin{equation*}
\begin{aligned}
    &\ell^m_t(X+H,\alpha+h)-\ell^m_t(X,\alpha)\\
    &=  H_t^\top \big[(Q_t+\bar{Q}_t)X_t-\bar{Q}_t(S_t m_t) \big] + H_t^\top \Big(\frac{Q_t+\bar{Q}_t}{2}\Big) H_t  + h_t^\top R_t \alpha_t + h_t^\top \frac{R_t}{2} h_t\\
    &\hspace{5.5cm}\text{and}\\
    &\Phi^m(X+H)-\Phi^m(X) \\ & = H_T^\top \big[(Q+\bar{Q})X_T-\bar{Q}(S m_T)\big] + H_T^\top \Big(\frac{Q+\bar{Q}}{2}\Big) H_T. 
    \end{aligned}
\end{equation*}
To finish our preparations, we define
\begin{gather}
    R^m_1(X,H,\alpha,h) :=\nonumber\\\mathbb{E}\bigg[\int_0^T\Big( H_t^\top \big[(Q_t+\bar{Q}_t)X_t-\bar{Q}_t(S_t m_t) \big] \Big)\ud t +  H_T^\top \big[(Q+\bar{Q})X_T-\bar{Q}(S m_T)\big] +\int_0^T(h_t^\top R_t \alpha_t)\ud t\bigg]\label{R1}\\
    \text{and}\nonumber\\
     R^m_2(H,h) :=\mathbb{E}\bigg[\int_0^T\Big( H_t^\top \Big(\frac{Q_t+\bar{Q}_t}{2}\Big) H_t+ h_t^\top \frac{R_t}{2} h_t\Big)\ud t + H_T^\top \Big(\frac{Q+\bar{Q}}{2}\Big) H_T\bigg],\label{R2}\\
     \text{so}\nonumber\\
     \E\!\left[\int_0^T \ell^m_t(X_t+H_t,\alpha_t+h_t)\ud t + \Phi^m(X_T+H_T)\right]- \E\!\left[\int_0^T \ell^m_t(X_t,\alpha_t)\ud t + \Phi^m(X_T)\right]\nonumber\\
     = R^m_1(X,H,\alpha,h) + R^m_2(H,h). \label{R3}
\end{gather}
\noindent 
For $\varepsilon>0$, define the perturbed control
$\alpha^\varepsilon := \alpha + \varepsilon h$
and the corresponding state $ X^{\alpha^\varepsilon,m,\xi}$ solving \eqref{eq:rough_state_2}.
Since the dependence of \eqref{eq:rough_state_2} on $(x,\alpha)$ is affine, we immediately get
\begin{equation*}
    \frac{X^{\alpha^\varepsilon,m,\xi}_t-X^{\alpha,m,\xi}_t}{\varepsilon}= \int_0^t \mathcal{K}(t,s)\Big(A_s \frac{X^{\alpha^\varepsilon,m,\xi}_s-X^{\alpha,m,\xi}_s}{\varepsilon} + B_s h_s \Big)\,\ud s,
\end{equation*}
the right-hand side above is independent of $\varepsilon$, hence the G\^ateaux derivative
\[
  {(\delta X)}_t^{\alpha,h} := \lim_{\varepsilon\rightarrow 0}\frac{X^{\alpha^\varepsilon,m,\xi}_t-X^{\alpha,m,\xi}_t}{\varepsilon},
\]
exists and satisfies the affine \emph{variational equation}
\begin{equation}\label{eq:variation_appendix}
{(\delta X)}_t^{\alpha,h}
 = \int_0^t \mathcal{K}(t,s)\big(A_s {(\delta X)}_s^{\alpha,h} + B_s h_s \big)\,\ud s.
\end{equation}

Therefore, the map $\alpha\mapsto X^{\alpha,m,\xi}$ is Gâteaux differentiable.
Moreover, notice that, again due to the affine dynamics of \eqref{eq:rough_state_2}, ${(\delta X)}_t^{\alpha,h}$ does not depend on $\alpha$ either. Hence the notation ${(\delta X)}_t^{h}$ is justified and for every $\varepsilon>0,$ $h\in\widehat{\mathcal{A}},$ we have
\begin{equation*}
    X^{\alpha+\varepsilon h,m,\xi}_t = X^{\alpha,m,\xi}_t +\varepsilon {(\delta X)}_t^{h}.
\end{equation*}
Furthermore, $ \mathcal{K}(t,s),A_s,B_s$ are uniformly bounded over $s,t\in [0,T]$, hence from Tonelli's theorem and Gr\"onwall's and Jensen's inequalities, we also get the energy estimate

\begin{gather}
|{(\delta X)}_t^{h}|^2 \leq   2T^2\int_0^t |\mathcal{K}(t,s)|^2\big(|A_s|^2 |{(\delta X)}_s^{h}|^2 + |B_s|^2 |h_s|^2 \big)\,\ud s \nonumber\\
\Longrightarrow \mathbb{E}\Big[\sup_{u\in[0,t]}\big|{(\delta X)}_u^{h}\big|^2\Big]
\leq C_T \|h\|_{\mathcal A}^2+C_T\int_0^t\mathbb{E}\Big[\big|{(\delta X)}_s^{h}\big|^2\Big]\ud s\nonumber\\
\Longrightarrow \mathbb{E}\Big[\sup_{u\in[0,t]}\big|{(\delta X)}_t^{h}\big|^2\Big]
\leq C'_T \|h\|_{\mathbb{H}^2}^2,\label{eq:delta_est}
\end{gather}
for some constants $C_T,C'_T$ depending on $\mathcal{K},A,B,T$.

We can now move forward to compute the Gâteaux derivative of $\mathcal{J}^{m,\xi}$ with respect to $\widehat{\mathcal{A}}$. So first, we define $H^{\varepsilon}_t:= \varepsilon{(\delta X)}_t^{h}\ $ and by \eqref{R1} and \eqref{R2} we get
\begin{align*}
    &  \mathcal{J}^{m,\xi}(\alpha+\varepsilon h)-\mathcal{J}^{m,\xi}(\alpha)\\
    = & ~ R^m_1(X^{\alpha,m,\xi},(\delta X)^h,\alpha,\varepsilon h) + R^m_2((\delta X)^h,\varepsilon h)\\
    = & ~ \varepsilon R^m_1(X^{\alpha,m,\xi},(\delta X)^h,\alpha, h) + \varepsilon^2 R^m_2((\delta X)^h,h).
\end{align*}
In view of \eqref{eq:delta_est}
\begin{equation*}
    \big|R^m_2((\delta X)^h,h)\big| \leq C''_T \|h\|_{\mathbb{H}^2}^2,
\end{equation*}
for some constant $C''$ depending on $\mathcal{K},A,B,T,Q_t,\bar Q_t,R_t,Q,\bar Q.$ Hence,
\begin{align}
  (\delta \mathcal{J})^{m,\xi,h}(\alpha)
  &:= \lim_{\varepsilon\rightarrow 0}\frac{\mathcal{J}^{m,\xi}(\alpha+\varepsilon h)-\mathcal{J}^{m,\xi}(\alpha)}{\varepsilon} \nonumber\\&\hspace{.12cm}=  R^m_1(X^{\alpha,m,\xi},(\delta X)^h,\alpha,h)\nonumber\\
&\hspace{.12cm}=\mathbb{E}\bigg[\int_0^T\big((\delta X)^h_t\big)^\top\big[ (Q_t+\bar{Q}_t)X^{\alpha,m,\xi}_t-\bar{Q}_t (S_t m_t)\big]\ud t\nonumber\\&\qquad+\big((\delta X)^h_T\big)^\top\big[ (Q+\bar{Q})X^{\alpha,m,\xi}_T-\bar{Q}(S m_T)\big] +\int_0^T(\alpha_t^\top R_t h_t)\ud t\bigg]\label{eq 08},
  \end{align}
which proves G\^ateaux differentiability of $\mathcal{J}^{m,\xi}(\cdot)$ with derivative $(\delta \mathcal{J})^{m,\xi,h}(\cdot)$.

\noindent \textbf{Step 2. (Reduced form of $(\delta \mathcal{J})^{m,\xi,h}(\alpha)$ and adjoint process) }
At this point we eliminate the ${(\delta X)}^{h}$ terms in \eqref{eq 08}
using an integration by parts argument and an adjoint process. So, let $Y^{\alpha,m,\xi}$ be an adapted process satisfying the stochastic backward (adjoint) Volterra differential equation 
\begin{equation}
\begin{aligned}
Y^{\alpha,m,\xi}_t
 &= \mathcal{K}(T,t)^{\top}\big((Q+\bar Q) X_T^{\alpha,m,\xi} - \bar Q Sm_T\big) 
 \\
 &\quad+\int_t^T \mathcal{K}(s,t)^{\top} \Big((Q_s+\bar Q_s) X_s^{\alpha,m,\xi} - \bar Q_s S_s m_s + A_s^{\top} Y^{\alpha,m,\xi}_s\Big) \ud s
    \\
   &\quad-\int_t^T \mathcal{K}(s,t)^{\top} Z^{\alpha,m,\xi}_s\ud W_s.\label{eq:adjoint_appendix}
   \end{aligned}
\end{equation}
Regarding well posedeness of $Y^{\alpha,m,\xi}$ see Remark \ref{rmk 3.4} below.

Now, consider the scalar process $(Y^{\alpha,m,\xi}_t)^\top {(\delta X)}_{t\wedge\tau_i}^{h}$.
The integration-by-parts formula for semimartingales gives
\begin{align}
(Y^{\alpha,m,\xi}_T)^\top {(\delta X)}_{T}^{h} &~ = \big(\mathcal{K}(T,0)^\top Y^{\alpha,m,\xi}_T\big)^\top \big(\mathcal{K}(0,T){(\delta X)}_{T}^{h}\big)
\nonumber \\
 & ~ = \int_0^T \big(\mathcal{K}(t,0)^\top Y^{\alpha,m,\xi}_t\big)^\top \ud \big(\mathcal{K}(0,t){(\delta X)}_t^{h}\big)\label{eq. 00015}\\ 
 &\qquad + \int_0^T \big(\mathcal{K}(0,t){(\delta X)}_{t}^{h}\big)^\top \ud \big(\mathcal{K}(t,0)^{\top}Y^{\alpha,m,\xi}_t\big).\nonumber
\end{align}
By \eqref{eq:variation_appendix} and
\eqref{eq:adjoint_appendix} we have 
\begin{align*}
\int_0^T \big(\mathcal{K}(t,0)^\top Y^{\alpha,m,\xi}_t\big)^\top \ud \big(\mathcal{K}(0,t){(\delta X)}_t^{h}\big) = \int_0^T\Big((Y^{\alpha,m,\xi}_t)^\top A_t {(\delta X)}_t^{h}\Big)\ud t + \int_0^T\Big((Y^{\alpha,m,\xi}_t)^\top B_t h_t\Big)\ud t,
\end{align*}
and 
\begin{align*}
& \int_0^T \big(\mathcal{K}(0,t){(\delta X)}_t^{h} \big)^\top \ud \big(\mathcal{K}(t,0)^{\top}Y^{\alpha,m,\xi}_t\big) 
\\ 
= & ~  \int_0^T\Big(\big({(\delta X)}_t^{h}\big)^\top A_t^{\top} Y^{\alpha,m,\xi}_t\Big)\ud t\\&
-\int_0^T \big({(\delta X)}_t^{h} \big)^\top \big( (Q_t+\bar Q_t) X_t^{\alpha,m,\xi} - \bar Q_t S_t m_t \big) \ud t + \int_0^T \big({(\delta X)}_t^{h}\big)^\top Z^{\alpha,m,\xi}_t \ud W_t
\end{align*}
respectively.
By \eqref{eq:delta_est} $\int_0^\cdot \big({(\delta X)}_t^{h}\big)^\top Z^{\alpha,m,\xi}_t \ud W_t$ is a true martingale, hence
\begin{equation*}
    \begin{aligned}
   (Y^{\alpha,\pmb{\eta},m}_T)^\top {(\delta X)}_T^{h}
   &= - \int_0^T \big({(\delta X)}_t^{h}\big)^\top \big( (Q_t+\bar Q_t) X_t^{\alpha,m,\xi} - \bar Q_t S_t m_t \big)  \ud t \\& \quad +\int_0^T\Big((Y^{\alpha,m,\xi}_t)^\top B_t h_t\Big) \ud t+ M_T.
    \end{aligned}
\end{equation*}
Taking expectations (the martingale term $M_T$ vanishes) gives
\begin{gather*}
\E\Big[(Y^{\alpha,m,\xi}_T)^\top {(\delta X)}_T^{h}\Big]
 = ~ -\E\!\left[
     \int_0^T \big({(\delta X)}_t^{h}\big)^\top\big( (Q_t+\bar Q_t) X_t^{\alpha,m,\xi} - \bar Q_t S_t m_t \big) \ud t
     \right]
  \\ \quad + ~ \E\!\left[
       \int_0^T \Big((Y^{\alpha,m,\xi}_t)^\top B_t h_t\Big)\ud t
     \right].
\end{gather*}
Using $Y^{\alpha,m,\xi}_T = (Q+\bar Q) X_T^{\alpha,m,\xi} - \bar Q Sm_T$, we obtain
\begin{gather}
\E\!\Big[
   \big({(\delta X)}_T^{h}\big)^\top \big((Q+\bar Q) X_T^{\alpha,m,\xi} - \bar Q Sm_T\big)
 +  \int_0^T \big({(\delta X)}_t^{h}\big)^\top\big( (Q_t+\bar Q_t) X_t^{\alpha,m,\xi} - \bar Q_t S_t m_t \big) \ud t
 \Big]\nonumber\\
 = \E\!\left[
       \int_0^T \Big((Y^{\alpha,m,\xi}_t)^\top B_t h_t \Big)\ud t
     \right].\label{eq:ibp_appendix}
\end{gather}

Substituting \eqref{eq:ibp_appendix} into \eqref{eq 08} we obtain
\begin{equation}\label{eq 011}
   (\delta \mathcal{J})^{m,\xi,h}(\alpha)
  = \E\!\left[
      \int_0^T
        \big( B_t^\top Y^{\alpha,m,\xi}_t + R_t\alpha_t \big)^\top h_t\,\ud t
    \right].
\end{equation}

\noindent \textbf{Step 3. (First order condition)} Using \eqref{eq 011} we are going to derive a condition for a point $\alpha^*$ of $\widehat{\mathcal{A}}$ to be critical.
\noindent Since $h$ is arbitrary in $\widehat{\mathcal A}$, the necessary condition for optimality is
\[
  (\delta \mathcal{J})^{m,\xi,h}(\alpha)=0, \quad \forall\,h\in\widehat{\mathcal{A}}
  \quad\Longleftrightarrow\quad
 B_t^\top Y^{\alpha,m,\xi}_t + R_t\alpha_t  = 0 \quad \text{a.s., for a.e. }t\in[0,T].
\]
Therefore, 
\begin{equation}\label{eq:opt_control_appendix}
\alpha_t^* = - R_t^{-1} B_t^\top Y^{\alpha^*,m,\xi}_t,
\end{equation}
and one arrives at the coupled forward--backward stochastic Volterra system of differential equations \eqref{stoch_max_principle Volterra FBSDE} that an optimal quadruple $(\alpha^*,X^{\alpha^*,m,\xi},Y^{\alpha^*,m,\xi}, Z^{\alpha^*,m,\xi})$ must satisfy together with \eqref{eq:opt_control_appendix}.

\noindent \textbf{Step 4. (Sufficiency)} Finally, we show that if a triplet $(X^{m,\xi},Y^{m,\xi},Z^{m,\xi})\in \mathbb{S}^2([0,T];\mathbb{R}^d)\times \mathbb{S}^2([0,T];\mathbb{R}^d)\times \mathbb{H}^2([0,T];\mathbb{R}^{d\times q})$ satisfies system \eqref{stoch_max_principle Volterra FBSDE}, and one defines $\alpha^*$ by \eqref{eq:opt_control_appendix} using $Y^{m,\xi}$, then
in fact $\alpha^*\in\widehat{\mathcal{A}}$ and is a global unique minimizer of $\mathcal{J}^{m,\xi}(\cdot)$ over $\widehat{\mathcal{A}}$.
This result is stated and proved as Lemma \ref{lem:AlphaOptimality} below.
Hence, the proof of Theorem \ref{thm Stoch_optimal_control_Volterra} is now complete.
\end{proof}

\begin{lemma}[Optimality of $\alpha^*$]\label{lem:AlphaOptimality}
Let $m\in C([0,T];\mathbb{R}^d)$, $\xi\in \mathbb{L}^2(\mathcal{F}_0;\mathbb{R}^d)$ and assume that the Volterra system \eqref{stoch_max_principle Volterra FBSDE} has a solution $(X^{m,\xi},Y^{m,\xi},Z^{m,\xi})$ on $[0,T]$. Then, for every $h\in\widehat{\mathcal{A}}$ 
we have 
    \begin{equation*}
    \mathcal{J}^{m,\xi}(\alpha^*)+ \lambda\|h\|_{\mathbb{H}^2}^2 \leq \mathcal{J}^{m,\xi}(\alpha^*+ h).
\end{equation*}
\end{lemma}
\begin{proof}
As we have already seen, by \eqref{R1}, \eqref{R2} and \eqref{R3}, for every $h\in \widehat{\mathcal{A}}$ we have
\begin{gather*}
    \mathcal{J}^{m,\xi}(\alpha^*+h) - \mathcal{J}^{m,\xi}(\alpha^*)
=    R^m_1({X}^{m,\xi},{(\delta X)}^{h},\alpha^*,h)+  R^m_2({(\delta X)}^{h},h),
\end{gather*}
and because, by \eqref{eq 011} and \eqref{eq:opt_control_appendix} we have by design $R^m_1({X}^{m,\xi},{(\delta X)}^{h},\alpha^*,h)=0$, we actually have
\begin{gather*}
    \mathcal{J}^{m,\xi}(\alpha^*+h) - \mathcal{J}^{m,\xi}(\alpha^*)
=      R^m_2({(\delta X)}^{h},h).
\end{gather*}

So, by (\ref{S3}), $R_t\succ \lambda \text{Id}_{\kappa\times \kappa}$ and $Q_t,\bar Q_t,Q, \bar Q\succeq 0$, we get
\begin{gather*}
    \E\Bigg[
     \int_0^T \frac{1}{2}\Big[H_t^\top (Q_t+\bar{Q}_t) H_t\Big] \ud t
 + \frac{1}{2} H_T^\top (Q +\bar Q)H_T
   \Bigg]\geq 0\\
   \text{and}\\
  \lambda\|h\|_{\mathbb{H}^2}^2 \leq\E\Bigg[\int_0^Th_t^\top R_t h_t\ud t\Bigg].
\end{gather*}
From the above we deduce that for every $H\in  \mathbb{S}^2([0,T];\mathbb{R}^d)$ and $h\in\widehat{\mathcal{A}}$ we have
\begin{gather}
\lambda\|h\|_{\mathbb{H}^2}^2 \leq R^m_2(H,h)\label{ineq 018}
\end{gather}
and thus
   $\mathcal{J}^m(\alpha^*)+ \lambda\|h\|_{\mathbb{H}^2}^2 \leq \mathcal{J}^m(\alpha^*+ h).$ The proof is complete.
\end{proof}

\begin{rmk}\label{rmk 3.4}
By defining $\widehat{Y}^{\alpha,m,\xi}_t:= \mathcal{K}(t,0)^\top {Y}^{\alpha,m,\xi}_t$ and $\widehat{Z}^{\alpha,m,\xi}_t:= \mathcal{K}(t,0)^{\top} Z^{\alpha,m,\xi}_t$, equation \eqref{eq:adjoint_appendix} becomes equivalent to
\begin{align*}
    \widehat{Y}^{\alpha,m,\xi}_t
    &~ = \mathcal{K}(T,0)^{\top}\big((Q+\bar Q) X_T^{\alpha,m,\xi} - \bar Q Sm_T\big) \nonumber
    \\ &\quad+\int_t^T \mathcal{K}(s,0)^{\top} \Big((Q_s+\bar Q_s) X_s^{\alpha,m,\xi} - \bar Q_s S_s m_s + A_s^{\top}\mathcal{K}(0,s)^{\top} \widehat{Y}^{\alpha,m,\xi}_s\Big) \ud s 
    \\
   & \quad -\int_t^T \widehat{Z}^{\alpha,m,\xi}_s\ud W_s.
\end{align*}
The latter is a classical affine backward SDE with a unique solution $$(\widehat{Y}^{\alpha,m,\xi},\widehat{Z}^{\alpha,m,\xi})\in \mathbb{S}^2([0,T];\mathbb{R}^d)\times\mathbb{H}^2([0,T];\mathbb{R}^{d\times q}),$$ see for example \cite[Theorem 4.23]{carmona2018probabilistic}, hence, for every $\xi\in \mathbb{L}^2(\mathcal{F}_0), \alpha \in\widehat{\mathcal{A}}$ and $m\in C([0,T];\mathbb{R}^d)$, \eqref{eq:adjoint_appendix} also has a unique solution $$({Y}^{\alpha,m,\xi},{Z}^{\alpha,m,\xi})\in \mathbb{S}^2([0,T];\mathbb{R}^d)\times\mathbb{H}^2([0,T];\mathbb{R}^{d\times q}).$$
\end{rmk}

\begin{rmk}\label{rmk> 3.8}
Due to the nature of the cost function, $Y^{\alpha,m,\xi}$ was introduced as a solution of a tailor-made Volterra BSDE. This allowed us to use the standard product formula for semimartingales to calculate the product $(Y^{\alpha,m,\xi}_T)^\top {(\delta X)}_T^{h}$, see \eqref{eq. 00015} above. If we worked directly in a rough setting the analogue of $Y^{\alpha,m,\xi}$ would have been introduced as a solution of a rough BSDE, and at this step we would needed to use a rough product formula such as Theorem \ref{rough_product_formula}. As seen therein, this requires strong integrability conditions for the diffusion that in general one cannot get from the martingale representation theorems. Here, in the linear-quadratic setting, we have available (global) decoupling fields, so we can guarantee these integrability conditions (a posteriori) and offer an alternative route. But in more general settings\footnote{Of course with affine state dependence in the rough integral in order to be able to get the mild formulation of the dynamics.}, where (global) decoupling fields are not available, it is not clear if and how one can do this. However, even in these situations we could still push through with the mild formulation. 
\end{rmk}

\subsection{Well posedness of the Volterra FBSDE}\label{subsec:VFBSDE-WP}
In this subsection we provide the proof of  existence and uniqueness of the Volterra FBSDE \eqref{stoch_max_principle Volterra FBSDE}.


\begin{thm}\label{thm 3.9}
 Let (S1), (S2) and (S3) hold true, then for every fixed $m\in C([0,T];\mathbb{R}^d)$ and $\xi\in \mathbb{L}^2(\mathcal{F}_0;\mathbb{R}^d)$, the Volterra FBSDE \eqref{stoch_max_principle Volterra FBSDE} has a unique solution $({X}^{m,\xi},{Y}^{m,\xi},{Z}^{m,\xi})\in \mathbb{S}^2([0,T];\mathbb{R}^d)\times \mathbb{S}^2([0,T];\mathbb{R}^d)\times \mathbb{H}^2([0,T];\mathbb{R}^{d\times q})$. 
Furthermore, there exist continuously differentiable $P:[0,T]\rightarrow\mathbb{R}^{d\times d}, \Pi: [0,T]\rightarrow\mathbb{R}^{d}$ such that
 \begin{gather}
    Y^{m,\xi}_t = \big(\mathcal{K}(0,t)^\top P_t  \mathcal{K}(0,t)\big) {X}^{m,\xi}_t + \Big[\mathcal{K}(0,t)^\top\Pi_t -\big(\mathcal{K}(0,t)^\top P_t  \mathcal{K}(0,t)\big)M(m)_t \Big],
\end{gather}
where $P_t, \Pi_t$ are the unique solutions of the system
\begin{equation}\label{eq symmetric_Riccati}
    \left\{\begin{aligned}
        &\dot{P_t}+P_t \widehat{A}_t  + (\widehat{A}_t)^{\top} P_t - P_t 
        (\widehat{R}_t)^{-1}P_t +(\widehat{Q}_t+\widehat{\bar Q}_t)=0,\\&
        \dot{\Pi}_t + \big((\widehat{A}_t)^{\top} - P_t 
        (\widehat{R}_t)^{-1}\big)\Pi_t+P_t \widehat{D}^m_t +(\widehat{Q}_t+\widehat{\bar Q}_t)\mathcal{K}(0,t)M(m)_t- \widehat{\bar Q}_t\mathcal{K}(0,t) S_t m_t=0,
    \end{aligned}\right.
\end{equation}
with terminal conditions $P_T = \widehat{Q}+\widehat{\bar Q}$ and $\Pi_T = (\widehat{Q}+\widehat{\bar Q}) \mathcal{K}(0,T)M(m)_T- \widehat{\bar Q} \mathcal{K}(0,T)Sm_T$, while
$\widehat{A}_t, 
(\widehat{R}_t)^{-1},\widehat{Q}_t, \widehat{\bar{Q}}_t, \widehat{D}^m_t, \widehat{Q}, \widehat{\bar{Q}}$ are introduced below at \eqref{defn 3.8}.
\end{thm}
This result can be viewed as a generalization of the standard well posedness for the characteristic FBSDE of the SMP in the LQ setting, see for example \cite[p. 184]{carmona2018probabilistic}.
\begin{proof}
\noindent \textbf{Step 1. (Inverse kernel transformation)}
We first introduce the processes:
\begin{equation}\label{defn 2.8}
\begin{split}
    \widehat{X}^{m,\xi}_t & := \mathcal{K}(0,t)\big({X}^{m,\xi}_t-M(m)_t\big), \\
    \widehat{Y}^{m,\xi}_t & := \mathcal{K}(t,0)^\top {Y}^{m,\xi}_t,\\ 
    \widehat{\Sigma}_t & := \mathcal{K}(0,t)\Sigma_t, \\\widehat{Z}_t^{m,\xi} & := \mathcal{K}(t,0)^\top Z_t,
\end{split}
\end{equation}
where $\mathcal{K}(t,0)^{-1}=\mathcal{K}(0,t)$, see Remark \ref{rmk 3.1}. In view of \ref{S3}, the latter transform 
\eqref{stoch_max_principle Volterra FBSDE} into the following standard FBSDE with affine structure:

\begin{equation}\label{stoch_max_principle Volterra FBSDE 3}
\left\{\begin{aligned}
&\widehat{X}^{m,\xi}_t = \xi+ \int_0^t \mathcal{K}(0,s)\big(A_s\mathcal{K}(s,0) \widehat{X}^{m,\xi}_s - B_s R_s^{-1} B_s^\top\mathcal{K}(0,s)^\top \widehat{Y}^{m,\xi}_s +A_s M(m)_s+ C_s m_s\big)\,\ud s
  \\&\quad\quad\quad\quad
  + \int_0^t \widehat{\Sigma}_s\ud W_s,\\&
\widehat{Y}^{m,\xi}_t
 = \mathcal{K}(T,0)^\top\big((Q+\bar Q)\mathcal{K}(T,0) \widehat{X}_T^{m,\xi} +(Q+\bar Q)M(m)_T- \bar Q Sm_T\big)  \\&\quad+\int_t^T \mathcal{K}(s,0)^\top\Big((Q_s+\bar Q_s)\mathcal{K}(s,0) \widehat{X}_s^{m,\xi} +(Q_s+\bar Q_s)M(m)_s- \bar Q_s S_s m_s+A_s^{\top}\mathcal{K}(0,s)^\top \widehat{Y}^{m,\xi}_s\Big) \ud s
   \\&\quad\quad-\int_t^T \widehat{Z}_s^{m,\xi}\ud W_s.
\end{aligned}\right.
\end{equation}
\noindent \textbf{Step 2. (Decoupling field)}
In order to solve \eqref{stoch_max_principle Volterra FBSDE 3} we are going to assume a decoupling field $V$ with affine structure. Subsequently, we verify that the results we obtain indeed satisfy \eqref{stoch_max_principle Volterra FBSDE 3}. The ansatz is
\begin{gather}
V(t,x):= P_t x+\Pi_t,\hspace{.1cm}\forall(t,x)\in[0,T]\times \mathbb{R}^d\nonumber\\
\text{and}\nonumber\\
    \widehat{Y}^{m,\xi}_t = P_t \widehat{X}^{m,\xi}_t+\Pi_t.\label{eq 019}
\end{gather}
Of course, we have so far only postulated the form of $V.$ In order to properly define it, we must specify the coefficients $P:[0,T]\rightarrow\mathbb{R}^{d\times d}, \Pi: [0,T]\rightarrow\mathbb{R}^{d}$. By taking expectations on \eqref{stoch_max_principle Volterra FBSDE 3} and \eqref{eq 019}
we get 
\begin{equation}\label{stoch_max_principle Volterra FBSDE 3_EX}
\left\{\begin{aligned}
&\mathbb{E}[\widehat{X}^{m,\xi}_t] = \mathbb{E}[\xi]+ \int_0^t \mathcal{K}(0,s)\big(A_s\mathcal{K}(s,0) \mathbb{E}[\widehat{X}^{m,\xi}_s] - B_s R_s^{-1} B_s^\top\mathcal{K}(0,s)^\top \mathbb{E}[\widehat{Y}^{m,\xi}_s] \\&\quad\quad\quad\quad+A_s M(m)_s
+ C_s m_s\big)\,\ud s\\&
\mathbb{E}[\widehat{Y}^{m,\xi}_t]
 = \mathcal{K}(T,0)^\top\big((Q+\bar Q)\mathcal{K}(T,0) \mathbb{E}[\widehat{X}_T^{m,\xi}] + (Q+\bar Q)M(m)_T- \bar Q Sm_T\big)  \\&\quad\quad+\int_t^T \mathcal{K}(s,0)^\top\Big((Q_s+\bar Q_s)\mathcal{K}(s,0) \mathbb{E}[\widehat{X}_s^{m,\xi}]+(Q_s+\bar Q_s)M(m)_s\\&\quad\quad\quad\quad - \bar Q_s S_s m_s+A_s^{\top}\mathcal{K}(0,s)^\top \mathbb{E}[\widehat{Y}^{m,\xi}_s]\Big) \ud s.
\end{aligned}\right.
\end{equation}
and
\begin{equation}\label{eq 021}
    \mathbb{E}[\widehat{Y}^{m,\xi}_t] = P_t \mathbb{E}[\widehat{X}^{m,\xi}_t]+\Pi_t.
\end{equation}
For ease of notation we define, for every $t\in[0,T],$ the transformed coefficients
\begin{equation}\label{defn 3.8}
 \begin{aligned}
    &\widehat{A}_t:= \mathcal{K}(0,t)A_t\mathcal{K}(t,0), \hspace{.1cm} 
    (\widehat{R}_t)^{-1}:= [\mathcal{K}(0,t) B_t] R^{-1}_t [\mathcal{K}(0,t) B_t]^\top, 
    \widehat{Q}_t:= \mathcal{K}(t,0)^\top Q_t\mathcal{K}(t,0),\\& \widehat{\bar{Q}}_t:= \mathcal{K}(t,0)^\top \bar{Q}_t\mathcal{K}(t,0),\hspace{.1cm}\widehat{Q}:= \mathcal{K}(T,0)^\top Q\mathcal{K}(T,0),\hspace{.1cm} \widehat{\bar{Q}}:= \mathcal{K}(T,0)^\top \bar{Q}\mathcal{K}(T,0),\\&
\widehat{D}^m_t:=\mathcal{K}(0,t)A_t M(m)_t + \mathcal{K}(0,t)C_t m_t.
  \end{aligned}
\end{equation}
With the help of \eqref{defn 3.8} we rewrite \eqref{stoch_max_principle Volterra FBSDE 3_EX} as 
\begin{equation}\label{stoch_max_principle Volterra FBSDE 3_EX_2}
\left\{\begin{aligned}
&\mathbb{E}[\widehat{X}^{m,\xi}_t] = \mathbb{E}[\xi]+ \int_0^t \big(\widehat{A}_s \mathbb{E}[\widehat{X}^{m,\xi}_s] - 
(\widehat{R}_s)^{-1}\mathbb{E}[\widehat{Y}^{m,\xi}_s] +\widehat{D}^{m}_s\big)\,\ud s\\&
\mathbb{E}[\widehat{Y}^{m,\xi}_t]
 = (\widehat{Q}+\widehat{\bar Q})\mathbb{E}[\widehat{X}_T^{m,\xi}] +(Q+\bar Q)M(m)_T- \widehat{\bar Q} \mathcal{K}(0,T)Sm_T\\&\quad\quad\quad\quad+\int_t^T \Big((\widehat{Q}_s+\widehat{\bar Q}_s) \mathbb{E}[\widehat{X}_s^{m,\xi}]+(\widehat{A}_s)^{\top} \mathbb{E}[\widehat{Y}^{m,\xi}_s] \\&\quad\quad\quad\quad+(\widehat{Q}_s+\widehat{\bar Q}_s)\mathcal{K}(0,s)M(m)_s - \widehat{\bar Q}_s\mathcal{K}(0,s) S_s m_s\Big) \ud s.
\end{aligned}\right.
\end{equation}

Now, by differentiating \eqref{eq 021} and using the forward part of \eqref{stoch_max_principle Volterra FBSDE 3_EX_2}, together with \eqref{eq 021}, we get
\begin{align*}
    &\dot{\mathbb{E}[\widehat{Y}^{m,\xi}_t]}=\dot{P_t}\mathbb{E}[\widehat{X}^{m,\xi}_t]+P_t \dot{ \mathbb{E}[\widehat{X}^{m,\xi}_t]}+\dot{\Pi}_t\\&\quad
    = \dot{P_t}\mathbb{E}[\widehat{X}^{m,\xi}_t]+P_t \big(\widehat{A}_t \mathbb{E}[\widehat{X}^{m,\xi}_t] - 
    (\widehat{R}_t)^{-1} \mathbb{E}[\widehat{Y}^{m,\xi}_t] +\widehat{D}^{m}_t\big)+\dot{\Pi}_t\\&\quad
    = \dot{P_t}\mathbb{E}[\widehat{X}^{m,\xi}_t]+P_t \big(\widehat{A}_t \mathbb{E}[\widehat{X}^{m,\xi}_t] - 
    (\widehat{R}_t)^{-1}\big(P_t \mathbb{E}[\widehat{X}^{m,\xi}_t]+\Pi_t\big) +\widehat{D}^{m}_t\big)+\dot{\Pi}_t\\&\quad
    = \Big[ \dot{P_t}+P_t \widehat{A}_t - P_t 
    (\widehat{R}_t)^{-1}P_t \Big]\mathbb{E}[\widehat{X}^{m,\xi}_t]+ \Big[ \dot{\Pi}_t -P_t 
    (\widehat{R}_t)^{-1}\Pi_t+ P_t\widehat{D}^m_t\Big].
\end{align*}
And by the backward part of \eqref{stoch_max_principle Volterra FBSDE 3_EX_2}, together with \eqref{eq 021}, we get
\begin{gather*}
    \dot{ \mathbb{E}[\widehat{Y}^{m,\xi}_t]}=-\Big[(\widehat{Q}_t+\widehat{\bar Q}_t) + (\widehat{A}_t)^{\top} P_t \Big] \mathbb{E}[\widehat{X}_t^{m,\xi}]-\Big[(\widehat{A}_t)^{\top} \Pi_t - \widehat{\bar Q}_t\mathcal{K}(0,t) S_t m_t\Big],
\end{gather*}
hence 
\begin{align*}
    &0= \Big[ \dot{P_t}+P_t \widehat{A}_t  + (\widehat{A}_t)^{\top} P_t - P_t 
    (\widehat{R}_t)^{-1}P_t +(\widehat{Q}_t+\widehat{\bar Q}_t)\Big]\mathbb{E}[\widehat{X}^{m,\xi}_t]\\&\hspace{.1cm}+ \Big[ \dot{\Pi}_t + \big((\widehat{A}_t)^{\top} - P_t 
    (\widehat{R}_t)^{-1} \big)\Pi_t+P_t\widehat{D}^m_t +(\widehat{Q}_s+\widehat{\bar Q}_t)\mathcal{K}(0,t)M(m)_t- \widehat{\bar Q}_t\mathcal{K}(0,t) S_t m_t\Big].
\end{align*}
From the above equation, and because $\mathbb{E}[\xi]$ may take any value of $\mathbb{R}^d$ at \eqref{stoch_max_principle Volterra FBSDE 3_EX_2}, we derive the system of backward matrix Ordinary Differential Equations
\begin{equation}
    \left\{\begin{aligned}
        &\dot{P_t}+P_t \widehat{A}_t  + (\widehat{A}_t)^{\top} P_t - P_t 
        (\widehat{R}_t)^{-1} P_t +(\widehat{Q}_t+\widehat{\bar Q}_t)=0,\\&
        \dot{\Pi}_t + \big((\widehat{A}_t)^{\top} - P_t 
        (\widehat{R}_t)^{-1}\big)\Pi_t + P_t \widehat{D}^m_t +(\widehat{Q}_t+\widehat{\bar Q}_t)\mathcal{K}(0,t)M(m)_t- \widehat{\bar Q}_t\mathcal{K}(0,t) S_t m_t=0,
    \end{aligned}\right.
\end{equation}
with terminal conditions $P_T = \widehat{Q}+\widehat{\bar Q}$ and $\Pi_T = (\widehat{Q}+\widehat{\bar Q}) \mathcal{K}(0,T)M(m)_T- \widehat{\bar Q} \mathcal{K}(0,T)Sm_T$.

The first equation of the above system is a symmetric (see Remark \ref{rmk 3.8}) Riccati matrix differential equation with terminal condition, and the second is a non-homogeneous linear matrix differential equation with terminal condition. 

Assuming that the Riccati equation has a unique solution on $[0,T]$, then the second equation for $\Pi_t$ is obviously also uniquely solvable on $[0,T]$, and the decoupling field $V(t,x)$ is properly defined. Next,
using \eqref{eq 019} we write the forward equation of \eqref{stoch_max_principle Volterra FBSDE 3} as

\begin{equation*}
    \begin{aligned}
\widehat{X}^{m,\xi}_t = \xi+ \int_0^t &\mathcal{K}(0,s)\big(A_s\mathcal{K}(s,0) \widehat{X}^{m,\xi}_s - B_s R_s^{-1} B_s^\top\mathcal{K}(0,s)^\top (P_s \widehat{X}^{m,\xi}_s+\Pi_s)\,\ud s\\& +\int_0^t      \mathcal{K}(0,s)\big(A_s M(m)_s+ C_s m_s\big)\,\ud s+ \int_0^t \widehat{\Sigma}_s\ud W_s,
    \end{aligned}
\end{equation*}

which is an affine equation with respect to $\widehat{X}^{m,\xi}_t$ that obviously has a unique solution $\widehat{X}^{m,\xi}_t\in \mathbb{S}^2([0,T];\mathbb{R}^d)$. Finally, one can directly verify via It\^o's formula (see for example \cite[Corollary 6.4]{medvegyev2007stochastic} ($V(t,x)\in C^{1,2}([0,T]\times\mathbb{R}^d;\mathbb{R}^d)$), and the terminal conditions for $P_T, \Pi_T$, that $V(t,\widehat{X}^{m,\xi}_t)$ satisfies the backward equation of \eqref{stoch_max_principle Volterra FBSDE 3}. In other words, we have found a solution of \eqref{stoch_max_principle Volterra FBSDE 3} on the whole $[0,T]$, where the backward component is expressed via an affine decoupling field with the help of the forward component. Hence, to get uniqueness for the solution of \eqref{stoch_max_principle Volterra FBSDE 3} we apply \cite[Proposition 4.8]{carmona2018probabilistic}.

\noindent \textbf{Step 3. (Well posedness of Riccati)}
What remains to be proved is that the symmetric Riccati equation 
\begin{equation*}
     \dot{P_t}+P_t \widehat{A}_t  + (\widehat{A}_t)^{\top} P_t - P_t 
     (\widehat{R}_t)^{-1}P_t +(\widehat{Q}_t+\widehat{\bar Q}_t)=0,\hspace{.2cm} P_T = \widehat{Q}+\widehat{\bar Q}
\end{equation*}
has a unique global solution on $[0,T]$. 
But this follows\footnote{We remind again Remark \ref{rmk 3.8}.} from Theorem \ref{thm: sym_rough_riccati} with the rough components equal to zero. 
\end{proof}

\begin{rmk}\label{rmk 3.8}
    It is immediate to see that \(\widehat{Q},\widehat{\bar Q},\widehat{Q}_t,\widehat{\bar Q}_t\succeq 0\).
    Additionally, we remind once more the definition of $ (\widehat{R}_t)^{-1}:= [\mathcal{K}(0,t) B_t] R^{-1}_t [\mathcal{K}(0,t) B_t]^\top$.
\end{rmk}

\section{Solving the rough LQ MFG}\label{sec 3} 

\subsection{Connecting the Volterra and rough frameworks}
In this section we establish the equivalence between the Volterra and the rough settings. For this purpose we are going to restrict the space of admissible controls to
\begin{equation}\label{eq:Arough}
\begin{aligned}
    \mathcal A:=\bigg\{\alpha \in \widehat{\mathcal{A}}:  \sup_{t\in[0,T]}\mathbb{E}[|\alpha_t|^k]<\infty,\hspace{.1cm}\text{for every}\hspace{.1cm}k\in [2,\infty)\bigg\},
\end{aligned}
\end{equation}
while we also assume that $\Sigma \in  \mathbb{H}^2([0,T];\mathbb{R}^{d\times q})$ 
is such that $\sup_{t\in[0,T]}\mathbb{E}[|\Sigma_t|^k] < \infty$ for every $k\in [2,\infty)$.

Let $\pmb{\eta} \in \mathscr{C}^{0,\gamma}_g([0,T];\mathbb{R}^p), p\in\mathbb{N}$, with $\gamma\in(\frac{1}{3},\frac{1}{2})$ and 
 $0<\beta'\leq \beta\leq\gamma$ such that $\gamma+\beta+\beta'>1$. Recall from Theorem \ref{thm sol_rough_affine} that the coefficients $A^1, C^1$ in the state dynamics are taken to be controlled rough paths. Thus we impose:

\begin{enumerate}[label=(S\arabic*-R), ref=S\arabic*-R]
    \item\label{S1-R} There exists $\big({A}^{1},(A^1)'\big), \big({C}^{1},(C^1)'\big)\in \mathscr{D}^{\beta,\beta'}_{\eta}\big([0,T];L(\mathbb{R}^d,\mathbb{R}^{d\times p})\big)$.
\end{enumerate}

Turning to the choice of the Volterra kernel, (\ref{S1-R}) along with Theorem \ref{thm sol_rough_affine}  guarantee that, for any $(m,m')\in \mathscr{D}^{\beta,\beta'}_{\eta}\big([0,T];\mathbb{R}^d\big)$, the deterministic controlled rough paths  $\big(U^{\pmb{\eta}},(U^{\pmb{\eta}})'\big) \in \mathscr{D}^{\gamma,\beta}_{\eta}\big([0,T];\mathbb{R}^{d\times d}\big)$ and $\big(M^{\pmb{\eta},(m,m')},(M^{\pmb{\eta},(m,m')})'\big)\in\mathscr{D}^{\gamma,\beta}_{\eta}\big([0,T];\mathbb{R}^{d}\big)$
respectively solve the linear RDEs
\begin{align}\label{eq:KkernelRDE}
 U^{\pmb{\eta}}_{t\leftarrow s}&=\text{Id}_{d\times d} + \int_s^t A^1_r U^{\pmb{\eta}}_{r\leftarrow s}\ud\pmb{\eta}_r,\\
    M^{\pmb{\eta},(m,m')}_t&= \int_{0}^t \big(A^1_rM^{\pmb{\eta},(m,m')}_r+ C^{1}_rm_r  \big)\ud\pmb{\eta}_r.\nonumber
\end{align}
The matrix $U^{\pmb{\eta}}_{t\leftarrow s}$ is invertible for every $t\in [0,T]$ (this follows e.g. by applying Theorem \ref{rough_product_formula}). Its inverse $(U^{\pmb{\eta}}_{t\leftarrow s})^{-1} =:U^{\pmb{\eta}}_{s\leftarrow t}$, where $\big(U^{\pmb{\eta}}_{s\leftarrow t},(U^{\pmb{\eta}}_{s\leftarrow t})'\big) \in \mathscr{D}^{\gamma,\beta}_{\eta}\big([0,T];\mathbb{R}^{d\times d}\big)$ is the solution of
\begin{equation*}
     U^{\pmb{\eta}}_{s\leftarrow t}=\text{Id}_{d\times d} - \int_s^t U^{\pmb{\eta}}_{s\leftarrow r}A^1_r\ud\pmb{\eta}_r.
\end{equation*}

\begin{rmk}\label{rmk 3.12}
Note that, for every $s,t\in [0,T]$ with $s\leq t$ we have 
\begin{align*}
 (U^{\pmb{\eta}}_{t\leftarrow s})^\top&=\text{Id}_{d\times d} + \int_s^t (U^{\pmb{\eta}}_{r\leftarrow s})^\top(A^1_r)^{\top}\ud\pmb{\eta}_r,\\
    (U^{\pmb{\eta}}_{s\leftarrow t})^\top&=\text{Id}_{d\times d} - \int_s^t (A^1_r)^{\top} (U^{\pmb{\eta}}_{s\leftarrow r})^\top\ud\pmb{\eta}_r.
\end{align*}
\end{rmk}
We shall need the following instance of Duhamel's principle. From Theorem \ref{rough_product_formula}, applied in ${U}^{\pmb{\eta}}_{ 0\leftarrow t }M^{\pmb{\eta},(m,m')}_t$, we get that 
\begin{equation}\label{eq:MRDE}
{U}^{\pmb{\eta}}_{ 0\leftarrow t }M^{\pmb{\eta},(m,m')}_t = \int_0^t U^{\pmb{\eta}}_{0\leftarrow s }C^{1}_sm_s\ud\pmb{\eta}_s.   
\end{equation}

Now, let us define
\begin{equation}\label{eq:RoughVolterraKerneldef}
    \mathcal{K}^{\pmb{\eta},(A^1,(A^1)')}(t,s):= {U}^{\pmb{\eta}}_{t\leftarrow 0}{U}^{\pmb{\eta}}_{0\leftarrow s },
\end{equation}
for every $t,s\in [0,T]$. 
Since $(A^1,(A^1)')$ is fixed throughout this section, we shall omit it from our notation and write $\mathcal{K}^{\pmb{\eta}}(\cdot,\cdot)$ instead of $\mathcal{K}^{\pmb{\eta},(A^1,(A^1)')}(\cdot,\cdot)$.

\begin{rmk}\label{rem:Kernelchoice} The Volterra kernel $\mathcal{K}^{\pmb{\eta}}$ defined above satisfies Assumption \ref{S1}. This can be verified directly from the flow and continuity properties of the linear RDE \ref{eq:KkernelRDE}.
\end{rmk}

We are now ready to state the desired equivalence result for the state processes via Duhamel's principle/variation of constants.

\begin{thm}\label{thm equivalence}
Let (\ref{S1-R}), (\ref{S2}) hold true, $\alpha\in\mathcal{A}$ and $(m,m')\in \mathscr{D}^{\beta,\beta'}_{\eta}\big([0,T];\mathbb{R}^d\big)$. Furthermore, let  ${{X}}^{\pmb{\eta},\alpha,(m,m'),\xi}$ be the unique solution of RSDE \eqref{eq rough_affineI} and $\widetilde{{X}}^{\pmb{\eta},\alpha,(m,m'),\xi}$ the unique solution of Volterra SDE \eqref{eq:rough_state_2}, where now $M(m)_t$ is replaced by $M^{\pmb{\eta},(m,m')}_t$ and $\mathcal{K}(t,s)$ is replaced by $\mathcal{K}^{\pmb{\eta}}(t,s)$.
\footnote{For the well posedness of \eqref{eq rough_affineI} and \eqref{eq:rough_state_2} see Theorem \ref{thm sol_rough_affine} and Remark \ref{rmk 3.2}.} Then, for every $t\in[0,T]$, we have
\begin{equation*}
\mathbb{P}\big[{X}^{\pmb{\eta},\alpha,(m,m'),\xi}_t=\widetilde{{X}}^{\pmb{\eta},\alpha,(m,m'),\xi}_t\big]=1.
\end{equation*}
\end{thm}
\begin{proof}
Direct application of the rough product formula, Theorem \ref{rough_product_formula}, to the product ${U}^{\pmb{\eta}}_{ 0\leftarrow t }{X}^{\pmb{\eta},\alpha,(m,m'),\xi}_t$ with ${U}^{\pmb{\eta}}_{ 0\leftarrow t }$ satisfying \eqref{eq:MRDE}.
\end{proof}

\subsection{The rough linear-quadratic stochastic control problem}
We remind the reader that our goal is to solve the rough LQ MFG problem \eqref{eq rough_affineI}-\eqref{eq:rConsistency}. In view of Theorem \ref{thm equivalence} we see that the analysis and the SMP we derived in Section \ref{sec:LQVolterraControl} can be applied directly to this purpose. So, let us first give the following definition.
\begin{defn}
Under (\ref{S1-R}), for every $(m,m')\in \mathscr{D}^{\beta,\beta'}_{\eta}\big([0,T];\mathbb{R}^d\big)$, we have (see Lemma \ref{productest})
\begin{gather}\label{eq:ZetaControlledPathDef}
\big(Z^{\pmb{\eta},(m,m')},(Z')^{\pmb{\eta},(m,m')}\big):=\nonumber\\\Big(\mathcal{K}^{\pmb{\eta}}(0,\cdot)C^1 m\;,\;\mathcal{K}^{\pmb{\eta}}(0,\cdot)'C^1m+\mathcal{K}^{\pmb{\eta}}(0,\cdot)(C^1)'m+\mathcal{K}^{\pmb{\eta}}(0,\cdot)C^1m'\Big)\in \mathscr{D}^{\beta,\beta'}_{\eta}\big([0,T];\mathbb{R}^{p\times d}\big).
\end{gather}
In turn, we define  
\begin{equation}\label{eq:MroughIntegral}
         M^{\pmb{\eta},(A^1,(A^1)'),(C^1,(C^1)')}\big((m,m')\big)_t:= \mathcal{K}^{\pmb{\eta}}(t,0)\int_{0}^t\big(Z^{\pmb{\eta},(m,m')},(Z')^{\pmb{\eta},(m,m')}\big)_s\ud \pmb{\eta}_s.
    \end{equation}
\end{defn}
Since $(A^1,(A^1)'),(C^1,(C^1)')$ are fixed in this section, for ease of notation, we shall denote $ M^{\pmb{\eta},(A^1,(A^1)'),(C^1,(C^1)')}(\cdot)$ by $ M^{\pmb{\eta}}(\cdot)$.

Hence, for fixed $(m,m')\in \mathscr{D}^{\beta,\beta'}_{\eta}\big([0,T];\mathbb{R}^d\big)$ and $\xi\in \bigcap_{k=2}^{\infty}\mathbb{L}^k(\mathcal{F}_0,\mathbb{R}^d)$, under (\ref{S1-R}), (\ref{S2}) and (\ref{S3}), from the Volterra formulation where $M(m)$ is replaced from $M^{\pmb{\eta}}\big((m,m')\big)$, the rough stochastic optimal control problem \eqref{eq rough_affineI}-\eqref{eq:rough_costI} has a unique solution. Furthermore, by Theorem \ref{thm 3.9} and Theorem \ref{thm equivalence}, for the optimal state and optimal control $\big(\alpha^{*,(m,m'),\xi},X^{(m,m'),\xi}\big)$, of the rough stochastic optimal control problem \eqref{eq rough_affineI}-\eqref{eq:rough_costI} we get 
the Volterra representations
\begin{equation}\label{stoch_max_principle rough FBSDE}
\left\{\begin{aligned}
&X^{(m,m'),\xi}_t = \mathcal{K}^{\pmb{\eta}}(t,0)\xi+M^{\pmb{\eta}}\big((m,m')\big)_t \\&\quad\quad\quad+ \int_0^t \mathcal{K}^{\pmb{\eta}}(t,s)\big(A_s X^{(m,m'),\xi}_s - B_s R_s^{-1} B_s^\top Y^{(m,m'),\xi}_s + C_s m_s\big)\,\ud s
   + \int_0^t \mathcal{K}^{\pmb{\eta}}(t,s)\Sigma_s\ud W_s,\\&
Y^{(m,m'),\xi}_t
 = \mathcal{K}^{\pmb{\eta}}(T,t)^\top\big((Q+\bar Q) X^{(m,m'),\xi}_T - \bar Q Sm_T\big) \\&\quad\quad\quad\quad +\int_t^T \mathcal{K}^{\pmb{\eta}}(s,t)^\top\Big((Q_s+\bar Q_s) X^{(m,m'),\xi}_s - \bar Q_s S_s m_s+A_s^{\top} Y^{(m,m'),\xi}_s\Big) \ud s 
  \\&\quad\quad\quad\quad\quad -\int_t^T \mathcal{K}^{\pmb{\eta}}(s,t)^\top Z^{(m,m'),\xi}_s\ud W_s,\\&
   \alpha^{*,(m,m'),\xi}_t = - R_t^{-1} B_t^\top Y^{(m,m'),\xi}_t.
\end{aligned}\right.
\end{equation}

\begin{rmk}\label{rmk 3.10}
It is important to note that 
\begin{equation*}
    \sup_{t\in[0,T]}\mathbb{E}\Big[\big|X^{(m,m'),\xi}_t\big|^k\Big] +  \sup_{t\in[0,T]}\mathbb{E}\Big[\big|Y^{(m,m'),\xi}_t\big|^k\Big] +  \sup_{t\in[0,T]}\mathbb{E}\Big[\big|\alpha^{*,(m,m'),\xi}_t\big|^k\Big] < \infty, \hspace{.2cm}\text{for every}\hspace{.1cm}k\in[2,\infty).
\end{equation*}
Indeed, by integrability of $\xi$ and the fact that the transformed problem (recall \eqref{defn 2.8}) is solved through a (deterministic) decoupling field, standard SDE estimates, see e.g. \cite[Theorem 3.4.3]{zhang2017backward}, imply 
\begin{equation*}
    \sup_{t\in[0,T]}\mathbb{E}\Big[\big|\widehat{X}^{(m,m'),\xi}_t\big|^k\Big] +  \sup_{t\in[0,T]}\mathbb{E}\Big[\big|\widehat{Y}^{(m,m'),\xi}_t\big|^k\Big]  < \infty, \hspace{.2cm}\text{for every}\hspace{.1cm}k\in[2,\infty).
\end{equation*}
Hence, the desired result follows.
\end{rmk}

We close this section by showing that the SMP \eqref{stoch_max_principle rough FBSDE} obtained from the Volterra formulation for the rough stochastic optimal control problem \eqref{eq rough_affineI}-\eqref{eq:rough_costI} can also be written in a form of rough FBSDE, while we also obtain the rough analogue of Theorem \ref{thm 3.9}.
\begin{thm}\label{thm sol_rough_FBSDE}
Let (\ref{S1-R}), (\ref{S2}), (\ref{S3}) hold true. Then, for every $(m,m')\in \mathscr{D}^{\beta,\beta'}_{\eta}\big([0,T];\mathbb{R}^d\big)$ and $\xi\in \bigcap_{k=2}^{\infty}\mathbb{L}^k(\mathcal{F}_0,\mathbb{R}^d)$, the rough FBSDE
\begin{equation}\label{eq:rFBSDE}
    \left\{\begin{aligned}
&X^{(m,m'),\xi}_t = \xi+ \int_0^t \big(A^1_s X^{(m,m'),\xi}_s+C^1_s m_s\big)\ud\pmb{\eta}_s\\&\quad\quad\quad\quad+ \int_0^t \big(A_s X^{(m,m'),\xi}_s - B_s R_s^{-1} B_s^\top Y^{(m,m'),\xi}_s + C_s m_s\big)\,\ud s
   + \int_0^t \Sigma_s\ud W_s,\\&
Y^{(m,m'),\xi}_t
 = (Q+\bar Q) X^{(m,m'),\xi}_T - \bar Q Sm_T +\int_t^T (A^1_s)^\top Y^{(m,m'),\xi}_s \ud\pmb{\eta}_s \\&\quad\quad\quad\quad+\int_t^T \big((Q_s+\bar Q_s) X^{(m,m'),\xi}_s - \bar Q_s S_s m_s+A_s^{\top} Y^{(m,m'),\xi}_s\big) \ud s 
   -\int_t^T  Z^{(m,m'),\xi}_s\ud W_s,
\end{aligned}\right.
\end{equation}
has a unique solution 
\begin{gather*}
\Big(\big(X^{(m,m'),\xi},(X^{(m,m'),\xi})'\big),\big(Y^{(m,m'),\xi},(Y^{(m,m'),\xi})'\big),Z^{(m,m'),\xi}\Big)\\\in \bigcap_{k =2}^\infty\Big(\textbf{D}^{\gamma,\beta}_{\eta}\mathbb{L}^k\big([0,T];\mathbb{R}^d\big)\times \textbf{D}^{\gamma,\beta}_{\eta}\mathbb{L}^k\big([0,T];\mathbb{R}^d\big)\times\mathbb{H}^k([0,T];\mathbb{R}^{d\times q})\Big).
\end{gather*}
with
\begin{equation*}  \sup_{t\in[0,T]}\mathbb{E}\big[|Z^{(m,m'),\xi}_t|^k\big]<\infty, \hspace{.2cm}\text{for every}\hspace{.1cm}k\in[2,\infty).
\end{equation*}
Furthermore, the solution of \eqref{eq:rFBSDE} almost surely coincides with the one of \eqref{stoch_max_principle rough FBSDE}.
\end{thm}

\begin{proof}
First we are going to establish existence of such a solution. From the discussion above, and specifically Theorem \ref{thm 3.9}, there exist unique processes $X^{(m,m'),\xi}, Y^{(m,m'),\xi},$ $(Z^1)^{(m,m'),\xi}$, with the appropriate integrabilities for $X^{(m,m'),\xi}, Y^{(m,m'),\xi},$ (see Remark \ref{rmk 3.10}) that solve \eqref{stoch_max_principle rough FBSDE}. We are going to show that $X^{(m,m'),\xi}, Y^{(m,m'),\xi}$ are also part of a solution for \eqref{eq:rFBSDE} along with an appropriate $Z^{(m,m'),\xi}$.

Keeping $ Y^{(m,m'),\xi}$ fixed, we consider the decoupled Volterra and RSDEs, respectively given by 
\begin{equation*}
\begin{aligned}
        &X^{(m,m'),\xi}_t = \mathcal{K}^{\pmb{\eta}}(t,0)\xi+M^{\pmb{\eta}}\big((m,m')\big)_t \\&\hspace{2cm}+ \int_0^t \mathcal{K}^{\pmb{\eta}}(t,s)\big(A_s X^{(m,m'),\xi}_s - B_s R_s^{-1} B_s^\top Y^{(m,m'),\xi}_s + C_s m_s\big)\,\ud s
   + \int_0^t \mathcal{K}^{\pmb{\eta}}(t,s)\Sigma_s\ud W_s,\\&
   \widetilde{X}^{(m,m'),\xi}_t = \xi+ \int_0^t \big(A^1_s \widetilde{X}^{(m,m'),\xi}_s+C^1_s m_s\big)\ud\pmb{\eta}_s\\&\hspace{4cm}+ \int_0^t \big(A_s \widetilde{X}^{(m,m'),\xi}_s - B_s R_s^{-1} B_s^\top Y^{(m,m'),\xi}_s + C_s m_s\big)\,\ud s
   + \int_0^t \Sigma_s\ud W_s.
\end{aligned}
\end{equation*}
From Theorem \ref{thm sol_rough_affine} both equations admit unique solutions and an application of Theorem \ref{thm equivalence} yields $\mathbb{P}\big[X^{(m,m'),\xi}=\widetilde{X}^{(m,m'),\xi}\big]=1$.

Now let us argue for the backward component. From the proof of Theorem \ref{thm 3.9}, there exist $P^{\pmb{\eta}}_t, \Pi^{\pmb{\eta},(m,m')}_t$ solutions of \eqref{eq symmetric_Riccati}, recall \eqref{defn 3.8}, such that 
\begin{gather}\label{eq:ansatzRFBSDE}
    Y^{(m,m'),\xi}_t = \mathcal{K}^{\pmb{\eta}}(0,t)^\top P^{\pmb{\eta}}_t  \mathcal{K}^{\pmb{\eta}}(0,t)\big({X}^{(m,m'),\xi}_t -M^{\pmb{\eta}}\big((m,m')\big)_t\big) + \mathcal{K}^{\pmb{\eta}}(0,t)^\top\Pi^{\pmb{\eta},(m,m')}_t.
\end{gather}
For the rest of the proof we are going to denote $P^{\pmb{\eta}}, \Pi^{\pmb{\eta},(m,m')}$ by $P,\Pi$ for ease of notation.
We will apply Theorem \ref{rough_product_formula} several times in order to see that $ Y^{(m,m'),\xi}$ has the desired dynamics. Indeed, we have
\begin{gather*}
    \mathcal{K}^{\pmb{\eta}}(0,t)^\top P_t = P_0 - \int_{0}^t \mathcal{K}^{\pmb{\eta}}(0,s)^\top\big(P_s \widehat{A}_s  + (\widehat{A}_s)^{\top} P_s - P_s 
    (\widehat{R}_s)^{-1} P_s +(\widehat{Q}_t+\widehat{\bar Q}_s)\big) \ud s\\
    - \int_0^t (A^1_s)^{\top}\mathcal{K}^{\pmb{\eta}}(0,s)^\top P_s \ud\pmb{\eta}_s\\
     \mathcal{K}^{\pmb{\eta}}(0,t)^\top \Pi_t = \Pi_0-\int_0^t \mathcal{K}^{\pmb{\eta}}(0,s)^\top\Big(\big(P_s+(\widehat{A}_s)^{\top}-
     (\widehat{R}_s)^{-1}\big)\Pi_s +P_s \widehat{D}^{(m,m')}_s \\+(\widehat{Q}_s+\widehat{\bar Q}_s)\mathcal{K}^{\pmb{\eta}}(0,s)M^{\pmb{\eta}}\big((m,m')\big)_s- \widehat{\bar Q}_s\mathcal{K}^{\pmb{\eta}}(0,s) S_s m_s\Big)\ud s
    - \int_0^t (A^1_s)^{\top}\mathcal{K}^{\pmb{\eta}}(0,s)^\top \Pi_s \ud\pmb{\eta}_s,
    \end{gather*}
and
    \begin{equation*}
        \begin{aligned}
             &\mathcal{K}^{\pmb{\eta}}(0,t)^\top P_t  \mathcal{K}^{\pmb{\eta}}(0,t)\big({X}^{(m,m'),\xi}_t
    -M^{\pmb{\eta}}\big((m,m')\big)_t\big)= P_0\xi \\&-\int_0^t  (A^1_s)^{\top}\mathcal{K}^{\pmb{\eta}}(0,s)^\top P_s\mathcal{K}^{\pmb{\eta}}(0,s)\big({X}^{(m,m'),\xi}_s
    -M^{\pmb{\eta}}\big((m,m')\big)_s\big) \ud\pmb{\eta}_s\\&
    -\int_{0}^t \bigg[\Big(\mathcal{K}^{\pmb{\eta}}(0,s)^\top\big(P_s \widehat{A}_s  + (\widehat{A}_s)^{\top} P_s - P_s 
    (\widehat{R}_s)^{-1}P_s \\&+(\widehat{Q}_t+\widehat{\bar Q}_s)\big)  \mathcal{K}^{\pmb{\eta}}(0,s)\big({X}^{(m,m'),\xi}_s
    -M^{\pmb{\eta}}\big((m,m')\big)_s\big) \Big)\\&
    - \mathcal{K}^{\pmb{\eta}}(0,s)^\top P_s \Big(\widehat{A}_s \mathcal{K}^{\pmb{\eta}}(0,s)\big({X}^{(m,m'),\xi}_s
    -M^{\pmb{\eta}}\big((m,m')\big)_s\big) \\&\quad\quad - 
    (\widehat{R}_s)^{-1}\mathcal{K}^{\pmb\eta}(s,0)^\top Y^{(m,m'),\xi}_s +\widehat{D}^{(m,m')}_s\Big)\bigg]\,\ud s\\&
    + \int_{0}^t  \mathcal{K}^{\pmb{\eta}}(0,s)^\top P_s \mathcal{K}^{\pmb{\eta}}(0,s)\Sigma_s\ud W_s.
        \end{aligned}
    \end{equation*}
Keeping \eqref{eq:ansatzRFBSDE} in mind, we calculate the Riemann integral
\begin{equation*}
\begin{aligned}
     \int_{0}^t &\bigg[\Big(\mathcal{K}^{\pmb{\eta}}(0,s)^\top\big(P_s \widehat{A}_s  + (\widehat{A}_s)^{\top} P_s - P_s 
     (\widehat{R}_s)^{-1} P_s \\&+(\widehat{Q}_t+\widehat{\bar Q}_s)\big)  \mathcal{K}^{\pmb{\eta}}(0,s)\big({X}^{(m,m'),\xi}_s
    -M^{\pmb{\eta}}\big((m,m')\big)_s\big) \Big)\\&
    - \mathcal{K}^{\pmb{\eta}}(0,s)^\top P_s \Big(\widehat{A}_s \mathcal{K}^{\pmb{\eta}}(0,s)\big({X}^{(m,m'),\xi}_s
    -M^{\pmb{\eta}}\big((m,m')\big)_s\big) \\
    &- 
    (\widehat{R}_s)^{-1} \mathcal{K}^{\pmb\eta}(s,0)^\top Y^{(m,m'),\xi}_s +\widehat{D}^{(m,m')}_s\Big)\bigg]\,\ud s\\&
    + \int_0^t \mathcal{K}^{\pmb{\eta}}(0,s)^\top\Big(\big(P_s+(\widehat{A}_s)^{\top}-
    (\widehat{R}_s)^{-1} \big)\Pi_s +P_s \widehat{D}^{(m,m')}_s \\&+(\widehat{Q}_s+\widehat{\bar Q}_s)\mathcal{K}^{\pmb{\eta}}(0,s)M^{\pmb{\eta}}\big((m,m')\big)_s- \widehat{\bar Q}_s\mathcal{K}^{\pmb{\eta}}(0,s) S_s m_s\Big)\ud s\\&
    = \int_0^t \big((Q_s+\bar Q_s) X^{(m,m'),\xi}_s - \bar Q_s S_s m_s+A_s^{\top} Y^{(m,m'),\xi}_s\big)\ud s.
\end{aligned}
\end{equation*}
In view of the previous displays we deduce
\begin{gather*}
     Y^{(m,m'),\xi}_t= (Q+\bar Q) X^{(m,m'),\xi}_T - \bar Q Sm_T +\int_t^T (A^1_s)^\top Y^{(m,m'),\xi}_s \ud\pmb{\eta}_s -  \int_{t}^T \mathcal{K}^{\pmb{\eta}}(0,s)^\top P_s \mathcal{K}^{\pmb{\eta}}(0,s)\Sigma_s\ud W_s\\
    + \int_t^T \big((Q_s+\bar Q_s) X^{(m,m'),\xi}_s - \bar Q_s S_s m_s+A_s^{\top} Y^{(m,m'),\xi}_s\big)\ud s.
\end{gather*}
The existence proof is thus complete upon letting $Z^{(m,m'),\xi}_t :=  \mathcal{K}^{\pmb{\eta}}(0,t)^\top P_t \mathcal{K}^{\pmb{\eta}}(0,t)\Sigma_t.$
Additionally, we immediately get that 
\begin{equation*}
\begin{aligned}
    (X^{(m,m'),\xi})'&= A^1 {X}^{(m,m'),\xi} + C^1 m,\\
    (Y^{(m,m'),\xi})'&= -(A^1)^\top {Y}^{(m,m'),\xi}.
    \end{aligned}
\end{equation*}

\noindent Turning to uniqueness, let $$\Big(\big(X^{(m,m'),\xi},(X^{(m,m'),\xi})'\big),\big(Y^{(m,m'),\xi},(Y^{(m,m'),\xi})'\big),Z^{(m,m'),\xi}\Big)$$ solve \eqref{eq:rFBSDE} with the properties described in the statement of the theorem. By repeating the procedure, from the other direction, i.e. forward component first and then backward, with the help of Theorem \ref{rough_product_formula}, we get that $\big(X^{(m,m'),\xi},Y^{(m,m'),\xi},\mathcal{K}^{\pmb{\eta}}(\cdot,\cdot)^\top Z^{(m,m'),\xi}\big)$ is a solution of \eqref{stoch_max_principle Volterra FBSDE}. Notice that we are allowed to do that because, by definition, the solution space imposes sufficient integrability properties. In particular, $Z^{(m,m'),\xi}$ is sufficiently integrable in order to apply Theorem \ref{rough_product_formula}. Hence, from the uniqueness part of Theorem \ref{thm 3.9}, uniqueness of \eqref{eq:rFBSDE} follows.
\end{proof}

\subsection{Equivalence of rough fixed point step with FBRDE}
Combining Theorem \ref{thm sol_rough_FBSDE}, which gives us the rough SMP of 
\eqref{eq rough_affineI}-\eqref{eq:rough_costI}, with consistency condition \eqref{eq:rConsistency}, we immediately obtain the following result.
\begin{thm}\label{thm:VolterraEquivalence} Let (\ref{S1-R}), (\ref{S2}), (\ref{S3}) hold true. Then, the rough LQ MFG problem \eqref{eq rough_affineI}-\eqref{eq:rConsistency} has a unique solution (per Definition \ref{dfn 1.1}) if and only if the rough McKean-Vlasov FBSDE \eqref{eq 0001} has a unique solution. 
\end{thm}

\begin{equation*}\tag{\ref{eq 0001}}
    \left\{\begin{aligned}
&X^{\xi}_t = \xi+ \int_0^t \big(A^1_s X^{\xi}_s+C^1_s \mathbb{E}\big[X^{\xi}_s\big]\big)\ud\pmb{\eta}_s\\&\quad\quad+ \int_0^t \big(A_s X^{\xi}_s - B_s R_s^{-1} B_s^\top Y^{\xi}_s + C_s \mathbb{E}\big[X^{\xi}_s\big]\big)\,\ud s
   + \int_0^t \Sigma_s\ud W_s,\\&
Y^{\xi}_t
 = (Q+\bar Q) X^{\xi}_T - \bar Q S \mathbb{E}\big[X^{\xi}_T\big] +\int_t^T (A^1_s)^\top Y^{\xi}_s \ud\pmb{\eta}_s \\&\quad\quad+\int_t^T \big((Q_s+\bar Q_s) X^{\xi}_s - \bar Q_s S_s \mathbb{E}\big[X^{\xi}_s\big]+A_s^{\top} Y^{\xi}_s\big) \ud s 
   -\int_t^T  Z^{\xi}_s\ud W_s,
\end{aligned}\right.
\end{equation*}
In order to solve \eqref{eq 0001} we follow the classical two step process. First we substitute $\big(\mathbb{E}\big[X^{\xi}_\cdot\big],(\mathbb{E}\big[X^{\xi}_\cdot\big])'\big)$ by a parameter $(m,m')\in \mathscr{D}^{\beta,\beta'}_{\eta}\big([0,T];\mathbb{R}^d\big)$, solve the new system of equations, and then in the second step search for fixed points with respect to $(m,m')$ and a tailor made-map. The first step has already been dealt with Theorem \ref{thm sol_rough_FBSDE}, so this subsection (and the rest of Section \ref{sec 3}) are concerned with the second step.

To ease the notation and since the initial condition $\xi$ is fixed in this section, we shall omit it from the notation.
Now, by taking expectations on \eqref{stoch_max_principle rough FBSDE} we get 
\begin{equation}\label{eq 029}
\left\{\begin{aligned}
&\mathbb{E}\big[X^{(m,m')}_t\big] = \mathcal{K}^{\pmb{\eta}}(t,0)\mathbb{E}[\xi] + 
M^{\pmb{\eta}}\big((m,m')\big)_t \\&\quad\quad\quad\quad\quad+ \int_0^t \mathcal{K}^{\pmb{\eta}}(t,s)\big(A_s \mathbb{E}\big[X^{(m,m')}_s\big] - B_s R_s^{-1} B_s^\top \mathbb{E}\big[Y^{(m,m')}_s\big]+C_s m_s\big)\,\ud s,\\&
\mathbb{E}\big[Y^{(m,m')}_t\big]
 = \mathcal{K}^{\pmb{\eta}}(T,t)^\top\big((Q+\bar Q)\mathbb{E}\big[X^{(m,m')}_T\big]- \bar Q Sm_T\big) \\&
\quad\quad\quad\quad\quad\quad+\int_t^T \mathcal{K}^{\pmb{\eta}}(s,t)^\top\Big((Q_s+\bar Q_s)\mathbb{E}\big[X^{(m,m')}_s\big]-\bar Q_s S_s m_s+A_s^{\top} \mathbb{E}\big[Y^{(m,m')}_s\big]\Big) \ud s.
\end{aligned}\right.
\end{equation}
From Theorem \ref{thm equivalence}, with $\Sigma=0$ and $\xi$ deterministic (see the proof of Theorem \ref{thm sol_rough_FBSDE}), \eqref{eq 029} is equivalent to the forward-backward RDE (FBRDE)
\begin{equation}\label{eq 0030}
\left\{\begin{aligned}
&\mathbb{E}\big[X^{(m,m')}_t\big] = \mathbb{E}[\xi] + \int_{0}^t\big( A^1_s \mathbb{E}\big[X^{(m,m')}_s\big]+ {C}^{1}_s m_s\big)\ud \pmb{\eta}_s 
 \\&\quad\quad\quad\quad\quad+ \int_0^t \big(A_s \mathbb{E}\big[X^{(m,m')}_s\big] - B_s R_s^{-1} B_s^\top \mathbb{E}\big[Y^{(m,m')}_s\big]+C_s m_s\big)\,\ud s,\\&
\mathbb{E}\big[Y^{(m,m')}_t\big]
 = (Q+\bar Q)\mathbb{E}\big[X^{(m,m')}_T\big]- \bar Q Sm_T + \int_{t}^T (A^1_s)^\top \mathbb{E}\big[Y^{(m,m')}_s\big] \ud \pmb{\eta}_s 
 \\&\quad\quad\quad\quad\quad
 +\int_t^T \Big((Q_s+\bar Q_s)\mathbb{E}\big[X^{(m,m')}_s\big]-\bar Q_s S_s m_s+A_s^{\top} \mathbb{E}\big[Y^{(m,m')}_s\big]\Big) \ud s.
\end{aligned}\right.
\end{equation}

If $\big(m^0,(m^0)'\big)$ is a fixed point of the map $\widetilde{F}:\mathscr{D}^{\beta,\beta'}_{\eta}\big([0,T];\mathbb{R}^d\big) \rightarrow \mathscr{D}^{\beta,\beta'}_{\eta}\big([0,T];\mathbb{R}^d\big)$, where 
\begin{equation}\label{eq:map_widetilde_F}
    \widetilde{F}\big((m,m')\big)_t:= \big(\mathbb{E}[{X}^{(m,m')}_t], A^1_t \mathbb{E}[{X}^{(m,m')}_t] +C^1_tm_t\big),
\end{equation}
then it satisfies the consistency equation
\begin{equation}\label{eq:rRconsistencyCondition}
   (m_t,m'_t) = \big(\mathbb{E}[{X}^{(m,m')}_t], (A^1_t+C^1_t) \mathbb{E}[{X}^{(m,m')}_t] \big).
\end{equation}
Plugging this equation to the system above and setting $$\mathbb{E}\big[X^{(m^0,(m^0)')}_t\big]:= \bar{x}_t, \mathbb{E}\big[Y^{(m^0,(m^0)')}_t\big]:= \bar{y}_t$$
we obtain 
\begin{equation}\label{eq:FBRDE}
\left\{\begin{aligned}
&\bar{x}_t = \mathbb{E}[\xi]+\int_{0}^t (A^1_s+ C^1_s)\bar{x}_s\ud\pmb{\eta}_s+ \int_0^t \big((A_s+C_s) \bar{x}_s - B_s R_s^{-1} B_s^\top \bar{y}_s \big)\,\ud s,\\&
\bar{y}_t
 = (Q+\bar Q- \bar Q S)\bar{x}_T +\int_t^T (A^1_s)^\top \bar{y}_s \ud \pmb{\eta}_s +\int_t^T \big((Q_s+\bar Q_s-\bar Q_s S_s) \bar{x}_s+A_s^{\top} \bar{y}_s\big) \ud s.
\end{aligned}\right.
\end{equation}
\begin{rmk}\label{rmk 4.7}
Note that the solution $\big((\bar{x},\bar{x}'),(\bar{y},\bar{y}')\big)$ of the FBRDE  is an element of $\mathscr{D}^{\gamma,\beta}_{\eta}\big([0,T];\mathbb{R}^d\big) \times \mathscr{D}^{\gamma,\beta}_{\eta}\big([0,T];\mathbb{R}^d\big)$ such that $\bar{x}'= (A^1+ C^1)\bar{x}$ and $\bar{y}'= -(A^1)^\top \bar{y}$.
\end{rmk}

\begin{thm}\label{thm:RMFGfixedPointEquivalence}
Let (\ref{S1-R}), (\ref{S2}), (\ref{S3}) hold true. Then, there exists a one-to-one correspondence between the fixed points of map $\widetilde{F}$ \eqref{eq:map_widetilde_F} and the solutions of  system \eqref{eq:FBRDE}. Specifically, every fixed point is the forward component of a solution of \eqref{eq:FBRDE}, and vise versa, every forward component of a solution of \eqref{eq:FBRDE} is a fixed point of $\widetilde{F}$.    
\end{thm}
This result can be viewed as a generalization of the standard equivalence in the LQ MFG setting, see for example \cite[Thm 3.34]{carmona2018probabilistic}.
\begin{proof}
First we assume that $\widetilde{F}$ has a fixed point $(m^0,(m^0)')$. From the analysis before Remark \ref{rmk 4.7} we have 
\begin{equation*}
\left\{\begin{aligned}
&\mathbb{E}\big[X^{(m^0,(m^0)')}_t\big] = \mathbb{E}[\xi] + \int_{0}^t ( A^1_s + {C}^{1}_s )  \mathbb{E}\big[X^{(m^0,(m^0)')}_s\big]\ud \pmb{\eta}_s 
 \\&\quad\quad\quad\quad\quad\quad+ \int_0^t \big((A_s + C_s) \mathbb{E}\big[X^{(m^0,(m^0)')}_s\big] - B_s R_s^{-1} B_s^\top \mathbb{E}\big[Y^{(m^0,(m^0)')}_s\big]\big)\,\ud s,\\&
\mathbb{E}\big[Y^{(m^0,(m^0)')}_t\big]
 = (Q+\bar Q - \bar Q S)\mathbb{E}\big[X^{(m^0,(m^0)')}_T\big] + \int_{t}^T (A^1_s)^\top \mathbb{E}\big[Y^{(m^0,(m^0)')}_s\big] \ud \pmb{\eta}_s 
 \\&\quad\quad\quad\quad\quad\quad
 +\int_t^T \Big((Q_s+\bar Q_s - \bar Q_s S_s)\mathbb{E}\big[X^{(m^0,(m^0)')}_s\big] + A_s^{\top} \mathbb{E}\big[Y^{(m^0,(m^0)')}_s\big]\Big) \ud s.
\end{aligned}\right.
\end{equation*}
and thus \eqref{eq:FBRDE} is solved by the pair 
\begin{multline*}
\Big(\big(\mathbb{E}[{X}^{(m^0,(m^0)')}_\cdot],(A^1_\cdot + C^1_\cdot)\mathbb{E}[{X}^{(m^0,(m^0)')}_\cdot]\big),\big(\mathbb{E}[{Y}^{(m^0,(m^0)')}_\cdot],-(A^1_\cdot)^\top \mathbb{E}[{Y}^{(m^0,(m^0)')}_\cdot]\big)\Big)\\
=\Big(\big(m^0_\cdot,(A^1_\cdot + C^1_\cdot)m^0_\cdot\big),\big(\mathbb{E}[{Y}^{(m^0,(m^0)')}_\cdot],-(A^1_\cdot)^\top \mathbb{E}[{Y}^{(m^0,(m^0)')}_\cdot]\big)\Big).
\end{multline*}
In other words, we have a mapping
$$(m^0,(m^0)')\mapsto \Big(\big(m^0_\cdot,(A^1_\cdot + C^1_\cdot)m^0_\cdot\big),\big(\mathbb{E}[{Y}^{(m^0,(m^0)')}_\cdot],-(A^1_\cdot)^\top \mathbb{E}[{Y}^{(m^0,(m^0)')}_\cdot]\big)\Big)$$ from fixed points to solutions of the FBRDE.

For the opposite direction, let $\big((\bar{x}^0,(\bar{x}^0)'),(\bar{y}^0,(\bar{y}^0)')\big)$ be a solution of \eqref{eq:FBRDE}. From Theorem \ref{thm sol_rough_FBSDE} there exists a unique triplet $\big(\big(X^{(\bar{x}^0,(\bar{x}^0)')},(X^{(\bar{x}^0,(\bar{x}^0)')})'\big),\big(Y^{(\bar{x}^0,(\bar{x}^0)')},(Y^{(\bar{x}^0,(\bar{x}^0)')})'\big),Z^{(\bar{x}^0,(\bar{x}^0)')}\big)$\footnote{Of the appropriate space described in Theorem \ref{thm sol_rough_FBSDE}.}
such that 
\begin{equation*}
    \left\{\begin{aligned}
&X^{(\bar{x}^0,(\bar{x}^0)')}_t = \xi+ \int_0^t \big(A^1_s X^{(\bar{x}^0,(\bar{x}^0)')}_s + C^1_s \bar{x}^0_s\big)\ud\pmb{\eta}_s\\&\quad\quad\quad\quad+ \int_0^t \big(A_s X^{(\bar{x}^0,(\bar{x}^0)')}_s - B_s R_s^{-1} B_s^\top Y^{(\bar{x}^0,(\bar{x}^0)')}_s + C_s \bar{x}^0_s\big)\,\ud s
   + \int_0^t \Sigma_s\ud W_s,\\&
Y^{(\bar{x}^0,(\bar{x}^0)')}_t
 = (Q+\bar Q) X^{(\bar{x}^0,(\bar{x}^0)')}_T - \bar Q S \bar{x}^0_T +\int_t^T (A^1_s)^\top Y^{(\bar{x}^0,(\bar{x}^0)')}_s \ud\pmb{\eta}_s \\&\quad\quad\quad\quad+\int_t^T \big((Q_s+\bar Q_s) X^{(\bar{x}^0,(\bar{x}^0)')}_s - \bar Q_s S_s \bar{x}^0_s + A_s^{\top} Y^{(\bar{x}^0,(\bar{x}^0)')}_s\big) \ud s 
   -\int_t^T  Z^{(\bar{x}^0,(\bar{x}^0)')}_s\ud W_s,
\end{aligned}\right.
\end{equation*}
We want to show that $\widetilde{F}\big((\bar{x}^0,(\bar{x}^0)')\big)_t:=\big(\mathbb{E}[{X}^{(\bar{x}^0,(\bar{x}^0)')}_t], A^1_t \mathbb{E}[{X}^{(\bar{x}^0,(\bar{x}^0)')}_t] +C^1_t \bar{x}^0_t\big) = (\bar{x}^0,(\bar{x}^0)').$ By taking expectations above\footnote{To avoid rough path considerations one can use \eqref{stoch_max_principle rough FBSDE} and Theorem \ref{thm sol_rough_FBSDE} as in the beginning of the subsection.} we have that the pair $$\Big(\big(\mathbb{E}[X^{(\bar{x}^0,(\bar{x}^0)')}_\cdot],A^1_\cdot \mathbb{E}[X^{(\bar{x}^0,(\bar{x}^0)')}_\cdot] + C^1_\cdot \bar{x}^0_\cdot\big),\big(\mathbb{E}[Y^{(\bar{x}^0,(\bar{x}^0)')}_\cdot],-(A^1_\cdot)^\top \mathbb{E}[Y^{(\bar{x}^0,(\bar{x}^0)')}_\cdot]\big)\Big)$$ 
satisfies
\begin{equation*}
    \left\{\begin{aligned}
&\mathbb{E}[X^{(\bar{x}^0,(\bar{x}^0)')}_t] = \mathbb{E}[\xi] + \int_0^t \big(A^1_s \mathbb{E}[X^{(\bar{x}^0,(\bar{x}^0)')}_s] + C^1_s \bar{x}^0_s\big)\ud\pmb{\eta}_s\\&\qquad\quad\quad\quad+ \int_0^t \big(A_s \mathbb{E}[X^{(\bar{x}^0,(\bar{x}^0)')}_s] - B_s R_s^{-1} B_s^\top \mathbb{E}[Y^{(\bar{x}^0,(\bar{x}^0)')}_s] + C_s \bar{x}^0_s\big)\,\ud s,\\&
\mathbb{E}[Y^{(\bar{x}^0,(\bar{x}^0)')}_t]
 = (Q+\bar Q) \mathbb{E}[X^{(\bar{x}^0,(\bar{x}^0)')}_T] - \bar Q S \bar{x}^0_T +\int_t^T (A^1_s)^\top \mathbb{E}[Y^{(\bar{x}^0,(\bar{x}^0)')}_s] \ud\pmb{\eta}_s \\&\qquad\quad\quad\quad+\int_t^T \big((Q_s+\bar Q_s) \mathbb{E}[X^{(\bar{x}^0,(\bar{x}^0)')}_s] - \bar Q_s S_s \bar{x}^0_s + A_s^{\top} \mathbb{E}[Y^{(\bar{x}^0,(\bar{x}^0)')}_s] \big) \ud s.
\end{aligned}\right.
\end{equation*}
Therefore, the triplet 
$$\Big(\big(\mathbb{E}[X^{(\bar{x}^0,(\bar{x}^0)')}_\cdot] - \bar{x}^0_\cdot,A^1_\cdot (\mathbb{E}[X^{(\bar{x}^0,(\bar{x}^0)')}_\cdot] - \bar{x}^0_\cdot)\big),\big(\mathbb{E}[Y^{(\bar{x}^0,(\bar{x}^0)')}_\cdot] - \bar{y}^0_\cdot,-(A^1_\cdot)^\top (\mathbb{E}[Y^{(\bar{x}^0,(\bar{x}^0)')}_\cdot]- \bar{y}^0_\cdot)\big),0\Big)$$
satisfies
\begin{equation*}
    \left\{\begin{aligned}
&\mathbb{E}[X^{(\bar{x}^0,(\bar{x}^0)')}_t] -\bar{x}^0_t = \int_0^t A^1_s \big(\mathbb{E}[X^{(\bar{x}^0,(\bar{x}^0)')}_s] -\bar{x}^0_s\big)\ud\pmb{\eta}_s\\&\quad\quad\qquad\quad\quad\quad+ \int_0^t \Big(A_s  \big(\mathbb{E}[X^{(\bar{x}^0,(\bar{x}^0)')}_s] -\bar{x}^0_s\big) - B_s R_s^{-1} B_s^\top  \big(\mathbb{E}[Y^{(\bar{x}^0,(\bar{x}^0)')}_s] -\bar{y}^0_s\big)\Big)\,\ud s,\\&
 \mathbb{E}[Y^{(\bar{x}^0,(\bar{x}^0)')}_t] -\bar{y}^0_t
 = (Q+\bar Q)  \big(\mathbb{E}[X^{(\bar{x}^0,(\bar{x}^0)')}_T] -\bar{x}^0_T\big) +\int_t^T (A^1_s)^\top  \big(\mathbb{E}[Y^{(\bar{x}^0,(\bar{x}^0)')}_s] -\bar{y}^0_s\big) \ud\pmb{\eta}_s \\&\qquad\qquad\qquad\quad+\int_t^T \Big((Q_s+\bar Q_s)  \big(\mathbb{E}[X^{(\bar{x}^0,(\bar{x}^0)')}_s] -\bar{x}^0_s\big) + A_s^{\top}  \big(\mathbb{E}[Y^{(\bar{x}^0,(\bar{x}^0)')}_s] -\bar{y}^0_s\big) \Big) \ud s.
\end{aligned}\right.
\end{equation*}
This last system has the form of \eqref{eq:rFBSDE} with $\xi=0,\Sigma_t=0,C^1_t=0, (C^1)'_t=0, C_t=0,S_t=0, S=0.$ In view of Theorem \ref{thm sol_rough_FBSDE} it has a unique solution which is equal to $\big((0,0),(0,0),0\big)$. Thus, $(\bar{x}^0,(\bar{x}^0)')$ is a fixed point of $\widetilde{F}$ and $$\big((\bar{x}^0,(\bar{x}^0)'),(\bar{y}^0,(\bar{y}^0)')\big) = \big((\bar{x}^0,(\bar{x}^0)'),\big(\mathbb{E}[Y^{(\bar{x}^0,(\bar{x}^0)')}_\cdot],-(A^1_\cdot)^\top \mathbb{E}[Y^{(\bar{x}^0,(\bar{x}^0)')}_\cdot]\big)\big).$$
\end{proof}

\subsection{Solving the FBRDE}\label{Section:FBRDE}

From Theorem \ref{thm:VolterraEquivalence} and Theorem \ref{thm:RMFGfixedPointEquivalence}, the rough LQ MFG problem \eqref{eq rough_affineI}-\eqref{eq:rConsistency} has a (unique) solution (per Definition \ref{dfn 1.1}), provided 
that \eqref{eq:FBRDE} has a (unique) solution. 

In this section we solve the latter
 through a decoupling field $\bar{v}.$
After inspecting the form of the terminal condition in \eqref{eq:FBRDE} we choose $\bar{v}$ to be linear.
Thus we make the ansatz 
\begin{gather}
\bar{v}(t,x):=  \bar{p}_t x,\hspace{.1cm}\forall(t,x)\in[0,T]\times \mathbb{R}^d\nonumber\\
\text{and}\nonumber\\
    {\bar{y}}_t = \bar{p}_t {\bar{x}}_t.\label{eq:FBRDEansatz}
\end{gather}
Of course, as we have already mentioned, in order for the decoupling field to be properly defined, we need to uniquely characterize the process $\bar{p}$. One way to do that is to find a proper characteristic equation for $\bar{p}$.

If $\pmb{\eta}$ is smooth, we can absorb the rough dependence of \eqref{eq:FBRDE} in the Riemann integrals, and the form of the characteristic equation for $\bar{p}$ is known by classical theory. This indicates that, in the rough case, the correct characteristic equation takes the form of the non-symmetric \textit{rough Riccati equation}
\begin{equation}\label{eq:rRiccati}
    \begin{aligned}
        \bar{p}_t&={Q}+{\bar Q}-{\bar{Q}}{S}+\int_t^T \big(\bar{p}_s(A^1_s +{C}^{1}_s)+(A^1_s)^\top\bar{p}_s\big) \ud\pmb{\eta}_s\\&+\int_t^T\Big(\bar{p}_s({A}_s + {C}_s) + ({A}_s)^{\top} \bar{p}_s - \bar{p}_s{B}_s {R}_s^{-1} {B}_s^\top \bar{p}_s + {Q}_s+{\bar Q}_s-{\bar{Q}}_s{S}_s\Big)\ud s.
    \end{aligned}
\end{equation}

\begin{enumerate}[label=(S4-R), ref=S4-R]
    \item\label{(S4-R)} We assume that $Q_t+\bar Q_t-\bar Q_t S_t\succeq 0$ and $Q+\bar Q-\bar Q S\succeq 0$. Additionally, let $(e,e')\in \mathscr{D}^{\gamma,\beta}_{\eta}\big([0,T];\mathbb{R}^{d\times d}\big)$ be the unique solution to the linear RDE\footnote{Uniqueness is standard, e.g.  by straightforward extension of \cite[Sec 8.9]{friz2020course}, or as special case of \cite[Thm 3.11] {bugini2024malliavin}, or \cite[Thm 3.5]{horst2025pontryagin}.}
    \begin{equation}\label{eq:eRDE}
 e_t=\text{Id}_{d\times d}+\int_0^t ({C}_s)^\top e_s\ud s + \int_0^t ({C}^{1}_s)^{\top}e_s\ud \pmb{\eta}_s.
    \end{equation}
    We assume that there exists $\lambda:[0,T]\rightarrow (0,\infty)$ such that, for all $t\in[0,T],$
    \begin{equation*}
        e_t = \lambda_t\text{Id}_{d\times d}.
    \end{equation*}
    So, 
    we get that $e'_t= \lambda_t({C}^{1}_t)^{\top}$.
\end{enumerate}

\begin{rmk}\label{rmk 4.12}
Assumption (\ref{(S4-R)}) can be viewed as a natural "rough" extension of an analogous condition that is typically assumed in classical LQ MFGs, see e.g. \cite[Proposition 3.35]{carmona2018probabilistic}. It is satisfied when for example there exist $\lambda^{(1)}:[0,T]\rightarrow\mathbb{R}$, $\big(\lambda^{(2)},(\lambda^{(2)})'\big)\in \mathscr{D}^{\beta,\beta'}_{\eta}\big([0,T];\mathbb{R}^{1\times p}\big)$, such that $C_t=\lambda^{(1)}_t\text{Id}_{d\times d}$ and ${C}^{1}_t(\cdot) = (\lambda^{(2)}_t\cdot)\text{Id}_{d\times d}$. Indeed then we have
\begin{align*}
    ({C}_t)^{\top} &= \lambda^{(1)}_t\text{Id}_{d\times d},\\
    ({C}^{1}_t(\cdot))^{\top}&= (\lambda^{(2)}_t\cdot)\text{Id}_{d\times d}, 
\end{align*}
  and so $e_t = \exp{\Big(\int_0^t\lambda^{(1)}_s\ud s+\int_0^t\lambda^{(2)}_s\ud\pmb{\eta}_s\Big)}\text{Id}_{d\times d}$.
 \end{rmk}

\begin{thm}\label{thm:FBRDEwellposedness}
Let (\ref{S1-R}), (\ref{S2}), (\ref{S3}) and (\ref{(S4-R)}) hold true.
Then, \eqref{eq:FBRDE} has a unique solution, and equivalently, the rough LQ MFG \eqref{eq rough_affineI}-\eqref{eq:rConsistency} also has a unique solution.
\end{thm} 

\begin{proof}
\noindent \textbf{Step 1. (Existence)}
By Theorem \ref{thm: sym_rough_riccati} we get a unique global solution for the symmetric rough Riccati 

\begin{equation}\label{eq:symrRiccati}
\left\{\begin{aligned}
        &\widetilde{p}_t={(Q}+{\bar Q}-{\bar{Q}}{S})e^{-1}_T+\int_t^T \Big(\widetilde{p}_s(A^1_s +{C}^{1}_s)+\big((A^1_s)^\top+(C^{1}_s)^\top\big)\widetilde{p}_s\Big) \ud\pmb{\eta}_s\\&\quad\quad+\int_t^T\Big(\widetilde{p}_s({A}_s + {C}_s) + ({A}_s)^{\top} \widetilde{p}_s - \widetilde{p}_s{B}_s {R}^{-1}_s {B}_s^\top e_s \widetilde{p}_s + ({Q}_s+{\bar Q}_s-{\bar{Q}}_s{S}_s)e^{-1}_s\Big)\ud s.
    \end{aligned}\right.
\end{equation}
 
\noindent We define, $\bar{p}_t:= e_t\widetilde{p}_t$ and we have

\begin{gather*}
({C}^{1}_t)^{\top}e_t \widetilde{p}_t - e_t\big(\widetilde{p}_t (A^1_t+{C}^{1}_t)+\big(({A}^{1}_t)^\top+({C}^{1}_t)^\top\big) \widetilde{p}_t\big)
= -\big(\bar{p}_t (A^1_t+{C}^{1}_t)+(A^{1}_t)^\top\bar{p}_t\big),\\
({C}_t)^\top e_t \widetilde{p}_t - e_t \big(\widetilde{p}_t({A}_t + {C}_t) + ({A}_t+{C}_t)^{\top} \widetilde{p}_t - \widetilde{p}_t{B}_t {R}^{-1}_t  {B}_t^\top e_t \widetilde{p}_t \\+({Q}_t+{\bar Q}_t-{\bar{Q}}_t{S}_t)e_t^{-1}\big)\\
= - \big(\bar{p}_t({A}_t + {C}_t) + ({A}_t)^{\top} \bar{p}_t - \bar{p}_t{B}_t {R}_t^{-1} {B}_t^\top \bar{p}_t +{Q}_t+{\bar Q}_t-{\bar{Q}}_t{S}_t\big).
\end{gather*}
Then, by the rough product formula, Theorem \ref{rough_product_formula}, 
we get that $\bar{p}_t$ satisfies \eqref{eq:rRiccati}.
So, $\bar{v}(t,x)$ is well defined\footnote{Uniqueness of \eqref{eq:rRiccati} comes by iterating local uniqueness.}.
Then, using the ansatz we decouple the first equation of \eqref{eq:FBRDE} and write it as
\begin{equation}\label{eq 034}
    \bar{x}_t = \mathbb{E}[\xi]+\int_{0}^t (A^1_s+ C^1_s)\bar{x}_s\ud\pmb{\eta}_s+ \int_0^t (A_s+C_s- B_s R_s^{-1} B_s^\top \bar{p}_s) \bar{x}_s \,\ud s.
\end{equation}
As a linear RDE, 
equation \eqref{eq 034} has a unique global solution in $\mathscr{D}^{\gamma,\beta}_{\eta}\big([0,T];\mathbb{R}^{d}\big)$. Finally, we define $${\bar{y}}_t:= \bar{v}(t,\bar{x}_t)=\bar{p}_t {\bar{x}}_t,$$ where ${\bar{x}}_t$ solves \eqref{eq 034}. To finish with existence, we need to verify that it satisfies the backward equation 
\begin{equation*}
 {\bar{y}}_t = ({Q}+{\bar Q}- {\bar Q} {S}){\bar{x}}_T+\int_t^T (A^1_s)^\top \bar{p}_s \bar{x}_s \ud \pmb{\eta}_s +\int_t^T \Big({Q}_s+{\bar Q}_s-{\bar Q}_s {S}_s +({A}_s)^{\top} \bar{p}_s\Big) {\bar{x}}_s\ud s,
\end{equation*}
in other words, we want to holds
\begin{align*}
   \bar{p}_t{\bar{x}}_t &= ({Q}+{\bar Q}- {\bar Q} {S}){\bar{x}}_T+\int_t^T (A^1_s)^\top \bar{p}_s \bar{x}_s \ud \pmb{\eta}_s +\int_t^T \Big({Q}_s+{\bar Q}_s-{\bar Q}_s {S}_s +{A}_s^{\top} \bar{p}_s\Big) {\bar{x}}_s\ud s\\
    &= \bar{p}_0\mathbb{E}[\xi]-\int_0^t (A^1_s)^\top \bar{p}_s \bar{x}_s \ud \pmb{\eta}_s-\int_{0}^t\Big({Q}_s+{\bar Q}_s-{\bar Q}_s {S}_s +{A}_s^{\top} \bar{p}_s\Big) {\bar{x}}_s\ud s,
\end{align*}
where $\bar{p}_t$ satisfies equation \eqref{eq:rRiccati} and $\bar{x}_t$ satisfies \eqref{eq 034}. This can be proven immediately by using the rough product formula Theorem \ref{rough_product_formula} as above. 
Indeed, we have
\begin{gather*}
-\Big(\bar{p}_t({A}_t + {C}_t) + {A}_t^{\top} \bar{p}_t - \bar{p}_t{B}_t {R}_t^{-1} {B}_t^\top \bar{p}_t +({Q}_t+{\bar Q}_t-{\bar{Q}}_t{S}_t)\Big){\bar{x}}_t\\
+\bar{p}_t\big({A}_t + {C}_t - {B}_t {R}_t^{-1} {B}_t^\top \bar{p}_t\big){\bar{x}}_t\\
=-\Big( {A}_t^{\top} \bar{p}_t +({Q}_t+{\bar Q}_t-{\bar{Q}}_t{S}_t)\Big){\bar{x}}_t
\end{gather*}
and so the proof for existence of a solution of \eqref{eq:FBRDE} is complete.

\noindent \textbf{Step 2. (Uniqueness)}
To this end we define
$\Xi^{\pmb{\eta}}_{t\leftarrow s}, \Xi^{\pmb{\eta}}_{s\leftarrow t}$ as the unique solutions of 
\begin{gather*}
     \Xi^{\pmb{\eta}}_{t\leftarrow s} = \text{Id}_{d\times d} + \int_s^t
     (A^1_r+{C}^{1}_r)\Xi^{\pmb{\eta}}_{r\leftarrow s}\ud \pmb{\eta}_r,\\
    \Xi^{\pmb{\eta}}_{s\leftarrow t} = \text{Id}_{d\times d} - \int_s^t \Xi^{\pmb{\eta}}_{s\leftarrow r}(A^1_r+{C}^{1}_r)\ud \pmb{\eta}_s.
\end{gather*}
Additionally, we define $$\Lambda^{\pmb{\eta}}(t,s):= \Xi^{\pmb{\eta}}_{t\leftarrow 0}\Xi^{\pmb{\eta}}_{0\leftarrow s}.$$ By virtue  of Theorem \ref{rough_product_formula}, and by repeating the proof of Theorem \ref{thm sol_rough_FBSDE} (without worrying about integrability properties as we are in a deterministic setting), we see that \eqref{eq:FBRDE} is equivalent to 
\begin{equation}\label{eq 0033}
\left\{\begin{aligned}
{\bar{x}}_t& = \Lambda^{\pmb{\eta}}(t,0)\mathbb{E}[\xi]+ \int_0^t \Lambda^{\pmb{\eta}}(t,s)\big(({A}_s+{C}_s) {\bar{x}}_s - {B}_s {R}_s^{-1} {B}_s^\top {\bar{y}}_s \big)\,\ud s,\\
{\bar{y}}_t
 &= \mathcal{K}^{\pmb{\eta}}(T,t)^\top({Q}+{\bar Q}- {\bar Q} {S}){\bar{x}}_T +\int_t^T \mathcal{K}^{\pmb{\eta}}(s,t)^\top\Big(({Q}_s+{\bar Q}_s-{\bar Q}_s {S}_s) {\bar{x}}_s+({A}_s)^{\top} {\bar{y}}_s\Big) \ud s,
\end{aligned}\right.
\end{equation}
and by defining $\widehat{\bar{x}}_t:=\Lambda^{\pmb{\eta}}(0,t)\bar{x}_t,\hspace{.1cm} \widehat{\bar{y}}_t:=\mathcal{K}^{\pmb{\eta}}(t,0)^\top \bar{y}_t$,
system \eqref{eq 0033} is equivalent to
\begin{equation}\label{eq 0034}
\left\{
\begin{aligned}
 &\widehat{\bar{x}}_t =\mathbb{E}[\xi]+ \int_0^t \Lambda^{\pmb{\eta}}(0,s)\big(({A}_s+{C}_s) \Lambda^{\pmb{\eta}}(s,0) \widehat{\bar{x}}_s - {B}_s {R}_s^{-1} {B}_s^\top \mathcal{K}^{\pmb{\eta}}(0,s)^\top \widehat{\bar{y}}_s \big)\,\ud s,\\&
\widehat{\bar{y}}_t
 = \mathcal{K}^{\pmb{\eta}}(T,0)^\top({Q}+{\bar Q}- {\bar Q} {S})\Lambda^{\pmb{\eta}}(T,0)\widehat{\bar{x}}_T \\&\quad\quad+\int_t^T \mathcal{K}^{\pmb{\eta}}(s,0)^\top\Big(({Q}_s+{\bar Q}_s-{\bar Q}_s {S}_s) \Lambda^{\pmb{\eta}}(s,0) \widehat{\bar{x}}_s+({A}_s)^{\top} \mathcal{K}^{\pmb{\eta}}(0,s)^\top\widehat{\bar{y}}_s\Big) \ud s.
\end{aligned}\right.
\end{equation}
From the analysis at the beginning of the subsection we have shown that \eqref{eq:FBRDE} has a solution of the form $(\bar{x}_t,\bar{p}_t \bar{x}_t)$.
 Hence, system \eqref{eq 0034} has a solution in the form $\big(\widehat{\bar{x}}_t,\mathcal{K}^{\pmb{\eta}}(t,0)^\top\bar{p}_t\Lambda^{\pmb{\eta}}(t,0) \widehat{\bar{x}}_t\big)$. In other words, there exists a decoupling field that gives a solution of \eqref{eq 0034}, so, from \cite[Proposition 4.8]{carmona2018probabilistic} this solution is unique. Thus, going backwards, \eqref{eq:FBRDE} also has a unique solution, and the proof is complete. 
\end{proof} 

\begin{rmk}
  \eqref{eq 0034} is a standard FBODE that is equivalent to \eqref{eq:FBRDE} and one may attempt to solve it directly. However, in order to deal with the asymmetric terms, existence of global solutions would require an additional condition of the form:
    \begin{enumerate}[label=(S5-R), ref=S5-R]
    \item\label{(S5-R)} $Q_t+\bar Q_t-\bar Q_t S_t\succeq 0$ and $Q+\bar Q-\bar Q S\succeq 0$. Additionally, with $
        e:[0,T]\rightarrow \mathbb{R}^{d\times d}$ unique solution of
    \begin{align*}
 e_t=\text{Id}_{d\times d}+\int_0^t (\widehat{C}_s)^\top e_s\ud s,
    \end{align*}
    and $\widehat{C}_t:= \mathcal{K}^{\pmb{\eta}}(0,t)C_t\mathcal{K}^{\pmb{\eta}}(t,0),$ there exist $\lambda^{(1)}, \lambda^{(2)}:[0,T]\rightarrow (0,\infty)$ such that for all $t\in[0,T],$
    \begin{align*}
        \Lambda^{\pmb{\eta}}(t,0)&= \lambda^{(1)}_t\mathcal{K}^{\pmb{\eta}}(t,0),\\
        e_t &= \lambda^{(2)}_t\text{Id}_{d\times d}.
    \end{align*}
    \end{enumerate}
\end{rmk}

\section{Stability of the rough LQ MFG}\label{sec:Stability}

Let $p\in\mathbb{N}, \gamma\in(\frac{1}{3},\frac{1}{2})$ and  $\mathcal{A}$ as in \eqref{eq:Arough}. In this section we establish stability of the solution to the MFG problem \eqref{eq rough_affineI}-\eqref{eq:rConsistency} with respect to initial conditions and the rough common noise. In particular, we prove (local Lipschitz) continuity properties for the equilibrium law, optimal value function and optimal control, respectively given by 

$$\mathscr{C}^{0,\gamma}_g([0,T];\mathbb{R}^p)\times \bigcap_{k=2}^{\infty}\mathbb{L}^k(\mathcal{F}_0,\mathbb{R}^d) \ni ( \pmb{\eta}, \xi)\longmapsto \mu^{ \pmb{\eta},\xi}=\mathcal{L}(X^{\pmb{\eta},\alpha^*(\pmb{\eta},\xi), \xi}_\cdot)\in  \mathcal{W}_k(C^\gamma([0,T];\R^d)),\;k\in[2,\infty),      $$
$$\alpha^*(\pmb{\eta},\xi) = \mathrm{argmin}_{\alpha\in  \mathcal A} \, \mathcal{J}(\alpha\, ;\pmb{\eta},(m,m')(\pmb{\eta},\mathbb{E}[\xi]),\xi),\footnotemark    $$
\footnotetext{Where $(m,m')=(m,m')(\pmb{\eta},\mathbb{E}[\xi])$ is the unique solution of \eqref{eq:FBRDE}.}and
$$\mathscr{C}^{0,\gamma}_g([0,T];\mathbb{R}^p)\ni \pmb{\eta}\longmapsto \mathcal{J}(\alpha^*;  \pmb{\eta},(m,m'),\xi)\in\R.$$ 
Above and throughout this section $\mathcal{W}_k(\mathcal{X}), k\in[2,\infty),$ denotes the $k-$Wasserstein space of probability measures over a Polish space $\mathcal{X}$ with topology induced by a metric $d.$ For $\mu^1,\mu^2\in\mathcal{W}_k(\mathcal{X})$ the $k-$Wasserstein distance is given by $\mathcal{W}^k_{k,\mathcal{X}}(\mu^1, \mu^2)=\inf_{\pi\in\Pi(\mu^1,\mu^2)}\iint_{\mathcal{X}\times \mathcal{X}}d^k(x,y)\ud \pi(x,y),$ where $\Pi$ denotes the family of couplings between $\mu^1, \mu^2.$

\begin{rmk}
As we have already mentioned, the kernel $\mathcal{K}^{\pmb{\eta}}$ and initial path $M^{\pmb{\eta}}$ depend on $(A^1,(A^1)')$ and $(A^1,(A^1)'),(C^1,(C^1)')$ respectively. For ease on notation, we continue to omit these dependencies. Nevertheless, they are implicitly assumed and are, in fact, important for the following stabilility analysis.
\end{rmk}

The following auxiliary lemma establishes continuity properties of the mean flows $(m, m')$ with respect to initial conditions and common noise.

\begin{lemma}\label{lem:mStability} Let $\xi_1,\xi_2\in \bigcap_{k=2}^{\infty}\mathbb{L}^k(\mathcal{F}_0,\mathbb{R}^d), T>0,$ $\gamma\in(\frac{1}{3},\frac{1}{2})$ and 
 $0<\beta'\leq \beta\leq\gamma$ such that $\gamma+\beta+\beta'>1$. For $\pmb{\eta}_1, \pmb{\eta}_2\in \mathscr{C}^{0,\gamma}_g([0,T];\mathbb{R}^p)$ that satisfy Assumption \eqref{(S4-R)} let $(m_i, (A^{1,i}+C^{1,i})m_i)\in \mathscr{D}^{\gamma,\beta}_{\eta_i}\big([0,T];\mathbb{R}^d\big)$ solve \eqref{eq:FBRDE} with coefficients $(A^{1,i}, (A^{1,i})'), (C^{1,i}, (C^{1,i})')\in \mathscr{D}^{\beta,\beta'}_{\eta_i}\big([0,T];L(\mathbb{R}^d,\mathbb{R}^{d\times p})\big)$ and initial conditions $\E[\xi_i],$ $i=1,2.$ Under Assumptions (\ref{S1-R}), (\ref{S2}), (\ref{S3}), if  $M>0$ is a constant such that 
$$
\sum_{i=1}^2\Big[\tnorm{\pmb{\eta}_i}_\gamma+\|(A^{1,i}, (A^{1,i})') \|_{\eta_i,\beta,\beta'} +\|(C^{1,i}, (C^{1,i})') \|_{\eta_i,\beta,\beta'} \Big]\leq M
$$
then up to a nonnegative multiplicative constant that depends on $M, T, \gamma, \beta,$  
\begin{equation*}
    \begin{aligned}
        \|m_1-m_2\|_{\infty;[0,T]}&+\|(m_1, (A^{1,1}+C^{1,1})m_1); (m_2, (A^{1,2}+C^{1,2})m_2) \|_{\eta_1, \eta_2,\gamma,\beta}\\&   \lesssim \|\xi_2-\xi_1\|_{k}+ \rho_{\gamma}(\pmb{\eta}_1, \pmb{\eta}_2)
        +\|(A^{1,1}, (A^{1,1})'); (A^{1,2}, (A^{1,2})') \|_{\eta_1, \eta_2,\beta, \beta'}\\&
        +\|(C^{1,1}, (C^{1,1})'); (C^{1,2}, (C^{1,2})')\|_{\eta_1, \eta_2,\beta,\beta'}.
        \end{aligned}
        \end{equation*}
\end{lemma}

\begin{proof} In view of the consistency condition \eqref{eq:rRconsistencyCondition}, any solution $(m, m')$ of \eqref{eq:FBRDE} satisfies 
$$  ({m}, {m}')=(\E[ X^{\alpha^*}], \E[ (X^{\alpha^*})'])=(\bar{x}, \bar{x}')=(\bar{x}, (A^1+C^1)\bar{x})    $$
on $[0,T]$ and $(\bar{x}, \bar{x}')\in \mathscr{D}^{\gamma,\beta}_{\eta}\big([0,T];\mathbb{R}^d\big)$ is the forward component of the FBRDE \eqref{eq:FBRDE}. From Theorem \ref{thm:FBRDEwellposedness}, the unique solution $(\bar{x}, \bar{y})$ of the latter satisfies the ansatz \eqref{eq:FBRDEansatz}. Consequently, we can write the solution pair in the form $( \bar{x}, \bar{p}\bar{x}),$ where $\bar{p}$ solves the (non-symmetric) rough Riccati equation \eqref{eq:rRiccati}. In turn, as shown in Section \ref{Section:FBRDE},  $\bar{p}=e\tilde{p},$ where $e$ is the (unique, global) solution of the linear RDE \eqref{eq:eRDE} and $\tilde{p}$ solves (uniquely and globally) the symmetric rough Riccati \eqref{eq:symrRiccati}. Thus, $(\bar{x}, \bar{y})$ takes the form $( \bar{x}, e\tilde{p}\bar{x})$ and, in view of Theorem \ref{thm: sym_rough_riccati}, the maps $\pmb{\eta}\longmapsto e, \tilde{p}$ are locally Lipschitz continuous. Returning to the FBRDE,  we substitute $\bar{y}=e\hat{p}\bar{x}$ to the dynamics of the forward component and obtain the decoupled linear RDE
\begin{equation}\label{eq:barxDecoupled}
     \bar{x}_t = \mathbb{E}[\xi]+\int_{0}^t (A^1_s+ C^1_s)\bar{x}_s\ud\pmb{\eta}_s+ \int_0^t \big((A_s+C_s) \bar{x}_s - B_s R_s^{-1} B_s^\top e_s\tilde{p}_s\bar{x}_s \big)\,\ud s.  
\end{equation}
To obtain the desired stability estimate, let $(\bar{x}_i, \bar{x}'_i)$ 
and $\bar{p}^i = e^i\tilde{p}^i, i=1,2,$ denote respectively solutions to  \eqref{eq:barxDecoupled} and \eqref{eq:rRiccati}   driven by  $\pmb{\eta}_1, \pmb{\eta}_2\in \mathscr{C}^{0,\gamma}_g([0,T];\mathbb{R}^p)$ and with initial conditions $\xi_1, \xi_2.$ Applying a stability estimate for (linear, deterministic) RDEs (see Proposition \ref{prop:LinearRDEestimate}), and noting that the drift terms also depend on the rough drivers, we obtain
\begin{equation*}
    \begin{aligned}
        \|\bar{x}_1-\bar{x}_2\|_{\infty;[0,T]}&+\|(\bar{x}_1, \bar{x}'_1); (\bar{x}_2, \bar{x}'_2) \|_{\eta_1,\eta_2,\gamma,\beta}\\&    \lesssim |\E[\xi_2]-\E[\xi_1]|+\rho_{\gamma}(\pmb{\eta}_1, \pmb{\eta}_2)\\&
        +\|(A^{1,1}, (A^{1,1})'); (A^{1,2}, (A^{1,2})') \|_{\eta_1, \eta_2,\beta,\beta'}\\&
        +\|(C^{1,1}, (C^{1,1})'); (C^{1,2}, (C^{1,2})') \|_{\eta_1, \eta_2,\beta,\beta'}\\&+\big\|e^2\tilde{p}^2-e^1\tilde{p}^1\big\|_{\infty;[0,T]},
    \end{aligned}
\end{equation*}
up to a nonnegative multiplicative constant that depends on $\tnorm{\pmb{\eta}_i}_\gamma, i=1,2$ and norms of the matrices associated to the linear vector fields. We bound the last term from above by using the triangle inequality along with local-Lipschitz estimates for the (symmetric) rough Riccati and linear RDE; see Proposition \ref{prop:LinearRDEestimate}, Theorem \ref{thm: sym_rough_riccati} respectively. In particular, up to constants depending continuously on the aforementioned parameters 
\begin{equation*}
    \begin{aligned}
\big\|e^2\tilde{p}^2-e^1\tilde{p}^1\big\|_{\infty;[0,T]}&\lesssim
\rho_{\gamma}(\pmb{\eta}_1, \pmb{\eta}_2)
        +\|(A^{1,1}, (A^{1,1})'); (A^{1,2}, (A^{1,2})') \|_{\eta_1,\eta_2,\beta,\beta'}
        \\&+\|(C^{1,1}, (C^{1,1})'); (C^{1,2}, (C^{1,2})')\|_{\eta_1, \eta_2,\beta,\beta' }.
     \end{aligned}
\end{equation*}
The proof is complete. \end{proof}

We turn to the proof of our main stability result which was described in the introduction and we state in detail below.

\begin{thm}[Stability of the rough MFG solution]\label{thm:StabilityMain} Let $\xi_1,\xi_2\in \bigcap_{k=2}^{\infty}\mathbb{L}^k(\mathcal{F}_0,\mathbb{R}^d), T>0,$ $\gamma\in(\frac{1}{3},\frac{1}{2})$ and 
 $0<\beta'\leq \beta\leq\gamma$ such that $\gamma+\beta+\beta'>1$. For $\pmb{\eta}_1, \pmb{\eta}_2\in \mathscr{C}^{0,\gamma}_g([0,T];\mathbb{R}^p)$ that satisfy Assumption \eqref{(S4-R)} let $(m_i, m'_i)\in \mathscr{D}^{\gamma,\beta}_{\eta_i}\big([0,T];\mathbb{R}^d\big)$ solve \eqref{eq:FBRDE} with coefficients $(A^{1,i}, (A^{1,i})'),$   $ (C^{1,i}, (C^{1,i})')\in \mathscr{D}^{\beta,\beta'}_{\eta_i}\big([0,T];L(\mathbb{R}^d,\mathbb{R}^{d\times p})\big)$ and initial conditions $\E[\xi_i]$  and $(X^{\pmb{\eta}_i, \alpha^{*, i}, (m_i, m_i')}, \alpha^{*, i}),$ $i=1,2$ be the corresponding solutions to the MFG \eqref{eq rough_affineI}-\eqref{eq:rConsistency} (per Definition \ref{dfn 1.1}). Under Assumptions (\ref{S1-R}), (\ref{S2}), (\ref{S3}) and for any  $M>0$ such that 
\begin{equation}\label{eq:LipLocM}
\sum_{i=1}^2\Big[\tnorm{\pmb{\eta}_i}_\gamma+\|(A^{1,i}, (A^{1,i})') \|_{\eta_i,\beta,\beta'} +\|(C^{1,i}, (C^{1,i})') \|_{\eta_i,\beta,\beta'} \Big]\leq M
\end{equation}
then, up to a nonnegative multiplicative constant that depends on $M, T, \gamma, \beta,$ we have 
\begin{equation}\label{eq:MainStabilityEstimate}
    \begin{aligned}       \sup_{t\in[0,T]}&\|\alpha^{*,2}_t-\alpha^{*,1}_t\|_{k}+  \|\|X^2-X^1\|_{\gamma}\|_{k}+\mathcal{W}_{k, C^\gamma}(\mu^1, \mu^2)\\&+\big|\mathcal{J}(\alpha^{*,2};\pmb{\eta}_2,(m_2,m'_2),\xi_2)-\mathcal{J}(\alpha^{*,1};\pmb{\eta}_1,(m_1,m_1'),\xi_1)\big|
        \\&\lesssim \|\xi_2-\xi_1\|_k+\rho_{\gamma}(\pmb{\eta}_2, \pmb{\eta}_1)+\|(A^{1,1}, (A^{1,1})'); (A^{1,2}, (A^{1,2})') \|_{\eta_1, \eta_2,\beta,\beta' }
        \\&+\|(C^{1,1}, (C^{1,1})'); (C^{1,2}, (C^{1,2})')\|_{\eta_1, \eta_2,\beta,\beta' },
    \end{aligned}
\end{equation}
where $\mu^i=\mathcal{L}(X^i)=\mathcal{L}(X^{\pmb{\eta}_i, \alpha^{*, i}, (m_i, m'_i),\xi_i})\in \mathcal{W}_k( C^\gamma([0,T];\R^d  ),$ and the optimal controls $$\alpha^{*,i}=\mathrm{argmin}_{\alpha\in  \mathcal A} \, \mathcal{J}(\alpha\,  ;\pmb{\eta}_i,(m_i,m_i'),\xi_i),\;\;i=1,2$$
take the explicit form $ \alpha^{*,i}_t = - R_t^{-1} B_t^\top Y^{i}_t, t\in[0,T]$ with $Y^i$ solving \eqref{eq:rFBSDE}. 

In addition, if the controlled rough paths $A^{1,i}, C^{1,i}, i=1,2$ satisfy 
\begin{equation}\label{eq:ControlledVectorFieldsContinuity}
    \|(A^{1,1}, (A^{1,1})'); (A^{1,2}, (A^{1,2})') \|_{\eta_1, \eta_2,\beta,\beta'  }+\|(C^{1,1}, (C^{1,1})'); (C^{1,2}, (C^{1,2})')\|_{\eta_1,\eta_2,\beta,\beta'  }\lesssim  \rho_{\gamma}(\pmb{\eta}_2, \pmb{\eta}_1), 
\end{equation}
whenever $  \tnorm{\pmb{\eta}_1}_\gamma+\tnorm{\pmb{\eta}_2}_\gamma\leq M,   $
then we obtain the local Lipschitz continuity estimate
\begin{equation*}
    \begin{aligned}       \sup_{t\in[0,T]}&\|\alpha^{*,2}_t-\alpha^{*,1}_t\|_k+  \|\|X^2-X^1\|_{\gamma}\|_{k}+\mathcal{W}_{k, C^\gamma}(\mu^1, \mu^2)\\&+\big|\mathcal{J}(\alpha^{*,2};\pmb{\eta}_2,(m_2,m_2'),\xi_2)-\mathcal{J}(\alpha^{*,1};\pmb{\eta}_1,(m_1,m_1'),\xi_1)\big|\lesssim \|\xi_2-\xi_1\|_k+\rho_{\gamma}(\pmb{\eta}_2, \pmb{\eta}_1),
    \end{aligned}
\end{equation*} 
 up to a nonnegative constant that depends on $M,T,\gamma, \beta.$
\end{thm}

\begin{proof} Starting from the optimal state, we adopt the shorter notation $$ X^i=X^{\pmb{\eta}_i,\alpha^{*, i}, (m_i, m'_i),\xi_i}, i=1,2,$$ with $\alpha^{*, i}, (m_i, m'_i)$ denoting the optimal control and mean paths of the MFG with corresponding common rough noise $\pmb{\eta}_i.$ In view of Theorem \ref{thm:VolterraEquivalence}, we express the dynamics in Volterra form and, for $0\leq s<t\leq T,$ 
\begin{equation}\label{eq:OptimalStateIncrement}
    \begin{aligned}
     \delta X^i_{s,t}\equiv X^i_t&-X^i_s=  \bigg(\mathcal{K}^{\pmb{\eta_i}}(t,0)-\mathcal{K}^{\pmb{\eta_i}}(s,0)\bigg)\xi_i+M^{\pmb{\eta}_i}\big( (m_i, m'_i)\big)_t-M^{\pmb{\eta}_i}\big( (m_i, m'_i)\big)_s\\&
     +\int_0^t\mathcal{K}^{\pmb{\eta_i}}(t,r)\bigg(A_r X^i_r+ \alpha^{*, i}_r+C_rm^i_r\bigg)\ud r-\int_0^s\mathcal{K}^{\pmb{\eta_i}}(s,r)\bigg(A_r X^i_r+ \alpha^{*, 1}_r+C_rm^i_r\bigg)\ud r\\&
     +\int_0^t\mathcal{K}^{\pmb{\eta_i}}(t,r)\Sigma_r\ud W_r-\int_0^s\mathcal{K}^{\pmb{\eta_i}}(s,r)\Sigma_r\ud W_r\\&
=:\delta(\mathcal{K}^{\pmb{\eta_i}})_{s,t}\xi_{i}+\delta M^{\pmb{\eta}_i}\big( (m_i, m'_i)\big)_{s,t}\\&
+\int_s^t\mathcal{K}^{\pmb{\eta_i}}(t,r)\bigg(A_r X^i_r+ \alpha^{*, i}_r+C_rm^i_r\bigg)\ud r\\&+\delta(\mathcal{K}^{\pmb{\eta_i}})_{s,t}\int_0^s\mathcal{K}^{\pmb{\eta_i}}(0,r)\bigg(A_r X^i_r+ \alpha^{*, i}_r+C_rm^i_r\bigg)\ud r\\&
    +\int_s^t\mathcal{K}^{\pmb{\eta_i}}(t,r)\Sigma_r\ud W_r+\delta(\mathcal{K}^{\pmb{\eta_i}})_{s,t}\int_0^s\mathcal{K}^{\pmb{\eta_i}}(0,r)\Sigma_r\ud W_r.
    \end{aligned}
\end{equation}

Note that the first two terms are deterministic and state-independent. Since the Volterra kernel \eqref{eq:RoughVolterraKerneldef} solves the linear matrix-valued RDE \eqref{eq:KkernelRDE} and $M^{\eta_i}$ is given by the (deterministic) rough integral \eqref{eq:MroughIntegral} we have the estimates
\begin{equation}\label{eq:KHolderEstimate}
    \begin{aligned}       \|\delta(\mathcal{K}^{\pmb{\eta_2}})\xi_2-\delta(\mathcal{K}^{\pmb{\eta_1}})\xi_1\|_{\gamma}\lesssim |\xi_2-\xi_1|+\rho_{\gamma}(\pmb{\eta}_2, \pmb{\eta}_1)+ \|(A^{1,1}, (A^{1,1})'); (A^{1,2}, (A^{1,2})')\|_{\eta_1, \eta_2, \beta, \beta' },
    \end{aligned}
\end{equation}
\begin{equation}\label{eq:StabilityT1bnd}
    \begin{aligned}       &\|\delta M^{\pmb{\eta}_2}\big( (m_2, m'_2)\big)-\delta M^{\pmb{\eta}_1}\big( (m_1, m'_1)\big)\|_{\gamma}\\&\lesssim  \rho_{\gamma}(\pmb{\eta}_2, \pmb{\eta}_1)+ \|(A^{1,1}, (A^{1,1})'); (A^{1,2}, (A^{1,2})')\|_{\eta_1, \eta_2, \beta, \beta' }\\&
    + \bigg\|\int_{0}^\cdot\big(Z_2^{\pmb{\eta_2},(m_2,m_2')},(Z_2')^{\pmb{\eta}_2,(m_2,m_2')}\big)_s\ud \pmb{\eta}_{2,s}-\int_{0}^\cdot\big(Z_1^{\pmb{\eta_1},(m_1,m_1')},(Z_1')^{\pmb{\eta}_1,(m_1,m_1')}\big)_s\ud \pmb{\eta}_{1,s}\bigg\|_{\gamma}.
    \end{aligned}
\end{equation}
By classical stability estimates for rough integrals (see e.g. \cite[Chapter 4]{friz2020course}), the last term is bounded from above by 
\begin{equation*}
    \begin{aligned}
        \rho_{\gamma}(\pmb{\eta}_2, \pmb{\eta}_1)+\big\|\big(Z_2^{\pmb{\eta_2},(m_2,m_2')},(Z_2')^{\pmb{\eta}_2,(m_2,m_2')}\big);\big(Z_1^{\pmb{\eta_1},(m_1,m_1')},(Z_1')^{\pmb{\eta}_1,(m_1,m_1')}\big)    \big\|_{\eta_1, \eta_2, \beta, \beta'}.
    \end{aligned}
\end{equation*}
In turn, in view of Definition \eqref{eq:ZetaControlledPathDef}, the second term above is bounded from above by 
\begin{equation*}
    \begin{aligned}
        \big\|&(C^{1,1}, (C^{1,1})'); (C^{1,2}, (C^{1,2})')    \big\|_{\eta_1, \eta_2, \beta, \beta' }+\big\|\big(m_2, m_2'\big);\big(m_1, m_1'\big)    \big\|_{\eta_1,\eta_2, \beta, \beta' }\\&
        +\big\| \big(\mathcal{K}^{\pmb{\eta_2}}(0,\cdot), (\mathcal{K}^{\pmb{\eta_2}})'(0,\cdot)\big); \big(\mathcal{K}^{\pmb{\eta_1}}(0,\cdot), (\mathcal{K}^{\pmb{\eta_1}})'(0,\cdot)\big)    \|_{\eta_1, \eta_2,\beta, \beta' }.
    \end{aligned}
\end{equation*}
We complete the estimate for the second term by invoking Lemma \ref{lem:mStability} along with stability estimates for the linear RDE \eqref{eq:KkernelRDE}. These bounds then furnish
\begin{equation}\label{eq:StabilityT2bnd}
    \begin{aligned}
        \|\delta M^{\pmb{\eta}_2}\big( (m_2, m'_2)\big)&-\delta M^{\pmb{\eta}_1}\big( (m_1, m'_1)\big)\|_{\gamma}\\&\lesssim \rho_{\gamma}(\pmb{\eta}_2, \pmb{\eta}_1)+ \big\|(C^{1,1}, (C^{1,1})'); (C^{1,2}, (C^{1,2})')   \big\|_{\eta_1, \eta_2, \beta, \beta' }
        \\&+ \big\|(A^{1,1}, (A^{1,1})'); (A^{1,2}, (A^{1,2})')    \big\|_{\eta_1, \eta_2, \beta, \beta' }.
    \end{aligned}
\end{equation}
Turning to the third term in \eqref{eq:OptimalStateIncrement} we have 
\begin{equation}\label{eq:StabilityT3prebnd}
    \begin{aligned}
\bigg|\int_s^t\mathcal{K}^{\pmb{\eta_2}}&(t,r)\bigg(A_r X^2_r+ \alpha^{*, 2}_r+C_rm^2_r\bigg)\ud r- \int_s^t\mathcal{K}^{\pmb{\eta_1}}(t,r)\bigg(A_r X^1_r+ \alpha^{*, 1}_r+C_rm^1_r\bigg)\ud r\bigg|\\&
\lesssim   \|\delta(\mathcal{K}^{\pmb{\eta_2}})-\delta(\mathcal{K}^{\pmb{\eta_1}})\|_{\gamma}\int_s^t(t-r)^\gamma\bigg(1+|\alpha^{*,2}_r|+|m_r^2|\bigg)\ud r\\&
+\int_s^t\bigg( |A_r||X^1_r-X^2_r|+|\alpha^{*,2}_r-\alpha^{*,1}_r|+|m^{2}_r-m^{1}_r|  \bigg)\ud r\\&
\lesssim \bigg[\rho_{\gamma}(\pmb{\eta}_2, \pmb{\eta}_1)+ \|(A^{1,1}, (A^{1,1})'); (A^{1,2}, (A^{1,2})')\|_{\eta_1, \eta_2,\beta,\beta';[0,T] }\bigg](t-s)^\gamma\int_0^T\bigg(1+|\alpha^{*,2}_r|\bigg)\ud r\\&
+\int_0^T|X^1_r-X^2_r|\ud r+\int_s^t|\alpha^{*,2}_r-\alpha^{*,1}_r|\ud r+(t-s)\|m^2-m^1\|_{\infty;[0,T]},
    \end{aligned}
\end{equation}
almost surely, where we used \eqref{eq:KHolderEstimate} as well as standard estimates for $\|m^2\|_{\infty;[0,T]}, \| U^{\pmb{\eta}_1}_{0\leftarrow \cdot}\|_{\infty;[0,T]}$ which are absorbed in the multiplicative constant,  to obtain the fourth line. Furthermore notice that, by virtue of Lemma \eqref{lem:mStability} the last term is upper bounded by 
\begin{equation*}
    \begin{aligned}
        (t-s)\bigg( \|\xi_2-\xi_1\|_{k}+ &\rho_{\gamma}(\pmb{\eta}_1, \pmb{\eta}_2)
        +\|(A^{1,1}, (A^{1,1})'); (A^{1,2}, (A^{1,2})')\|_{\eta_1, \eta_2,\beta,\beta' }
        \\&+\|(C^{1,1}, (C^{1,1})'); (C^{1,2}, (C^{1,2})') \|_{\eta_1, \eta_2,\beta,\beta' }   \bigg).  
    \end{aligned}
\end{equation*}    
In order to proceed, we need an estimate on the optimal controls. To this end, recall from \eqref{stoch_max_principle rough FBSDE} that
$$     \alpha^{*,i}_t = - R_t^{-1} B_t^\top Y^{i}_t   $$
and the unique, global solution $Y^i$ of the backward RDE satisfies the ansatz \eqref{eq:ansatzRFBSDE}
$$   Y_t^i=\mathcal{K}^{\pmb{\eta}_i}(0,t)^\top P^{\pmb{\eta}_i}_t  \mathcal{K}^{\pmb{\eta}_i}(0,t)\big({X}^{i}_t -M^{\pmb{\eta}_i}\big((m_i,m_i')\big)_t\big) + \mathcal{K}^{\pmb{\eta}_i}(0,t)^\top\Pi^{\pmb{\eta}_i,(m_i,m_i')}_t.$$
Thus, since $R, B$ are deterministic and bounded uniformly over $t\in[0,T],$ we have
\begin{equation*}
    |\alpha^{*,2}_r-\alpha^{*,1}_r|\lesssim |Y^2_r-Y^1_r|
\end{equation*}
almost surely for all $r\in[0, t].$
Next, notice that $P^{\pmb{\eta}_i},\Pi^{\pmb{\eta}_i,(m_i,m_i')}$ are (differentiable-in-time) solutions of the deterministic system of backward matrix equations  \eqref{eq symmetric_Riccati} with coefficients $\widehat{A}^i, \widehat{B}^i, \widehat{Q}^i, \widehat{\bar{Q}}^i,\mathcal{K}^{\pmb{\eta}_i}$ depending on $\pmb{\eta}_i, i=1,2.$ Thus, up to multiplicative constants that depend polynomially on the rough path norms $\tnorm{\pmb{\eta}_i}_\gamma$ we have the almost-sure estimate
\begin{equation}\label{eq:controlStabilityPrebnd}
    \begin{aligned}
        |\alpha^{*,2}_r-\alpha^{*,1}_r|&\lesssim \|\xi_2-\xi_1\|_{k}+ \rho_{\gamma}(\pmb{\eta}_2, \pmb{\eta}_1)+ \big\|(C^{1,1}, (C^{1,1})'); (C^{1,2}, (C^{1,2})')   \big\|_{\eta_1, \eta_2, \beta, \beta' }
        \\&+ \big\|(A^{1,1}, (A^{1,1})'); (A^{1,2}, (A^{1,2})')    \big\|_{\eta_1, \eta_2, \beta, \beta'}+|X^2_r-X^1_r|.
    \end{aligned}
\end{equation}
Substituting the latter back to \eqref{eq:StabilityT3prebnd} we obtain, up to deterministic constants, 
\begin{equation}\label{eq:StabilityT3bnd}
    \begin{aligned}
&\bigg|\int_s^t\mathcal{K}^{\pmb{\eta_2}}(t,r)\bigg(A_r X^2_r+ \alpha^{*, 2}_r+C_rm^2_r\bigg)\ud r- \int_s^t\mathcal{K}^{\pmb{\eta_1}}(t,r)\bigg(A_r X^1_r+ \alpha^{*, 1}_r+C_rm^1_r\bigg)\ud r\bigg|\\&
\lesssim (t-s)^\gamma\bigg( \|\xi_2-\xi_1\|_{k}+\rho_{\gamma}(\pmb{\eta}_1, \pmb{\eta}_2)
        +\|(A^{1,1}, (A^{1,1})'); (A^{1,2}, (A^{1,2})') \|_{\eta_1, \eta_2, \beta, \beta'}
       \\& +\|(C^{1,1}, (C^{1,1})'); (C^{1,2}, (C^{1,2})') \|_{\eta_1, \eta_2,\beta, \beta' }   \bigg)
        +\int_0^T|X^2_r-X^1_r|\ud r
\end{aligned}
\end{equation}
and an identical pathwise bound holds for the fourth term in \eqref{eq:OptimalStateIncrement}; its proof is omitted. 

It remains to estimate the It\^o integral terms in \eqref{eq:OptimalStateIncrement}. Both can be treated similarly via Doob's inequality. For example, for the first term on the last line we obtain, for any $k\geq 1$
\begin{equation}\label{eq:StabilityT5bnd}
    \begin{aligned}
      \bigg\|  \int_s^t(\mathcal{K}^{\pmb{\eta_2}}&(t,r)-\mathcal{K}^{\pmb{\eta_2}}(t,r))\Sigma_r\ud W_r  \bigg\|_{k}\\& \lesssim (t-s)^{1/2}\bigg( \rho_{\gamma}(\pmb{\eta}_2, \pmb{\eta}_1)+ \|(A^{1,1}, (A^{1,1})'); (A^{1,2}, (A^{1,2})')\|_{\eta_1, \eta_2, \beta, \beta' }  \bigg).
    \end{aligned}
\end{equation}
Collecting the estimates \eqref{eq:StabilityT1bnd}, \eqref{eq:StabilityT2bnd}, \eqref{eq:StabilityT3bnd}, \eqref{eq:StabilityT5bnd}
and applying a standard Gr\"onwall inequality yields the following stability bound for the optimal state:
\begin{equation}\label{eq:XstabilityEstimate}
    \begin{aligned}
        \|\|X^2-X^1\|_{\gamma}\|_{k}&\lesssim \|\xi_2-\xi_1\|_{k}+ \rho_{\gamma}(\pmb{\eta}_2, \pmb{\eta}_1)+ \big\|(C^{1,1}, (C^{1,1})'); (C^{1,2}, (C^{1,2})')   \big\|_{\eta_1, \eta_2, \beta, \beta'}
        \\&+ \big\|(A^{1,1}, (A^{1,1})'); (A^{1,2}, (A^{1,2})')    \big\|_{\eta_1, \eta_2, \beta, \beta' }.
    \end{aligned}
\end{equation}
Returning to \eqref{eq:controlStabilityPrebnd} we obtain the same bound for
$$\sup_{t\in[0,T]}\|\alpha^{*,2}_t-\alpha^{*,1}_t\|_k.$$
The estimate for the optimal costs $J(\alpha^{i,*}, \xi_i)$ follows immediately from the previous estimates and the fact that the cost function is quadratic in the state, mean vectors and controls.
Finally, we consider the $k-$Wasserstein distance between the equilibrium laws $\mu^i, i=1,2,$ viewed as measures on the H\"older space $C^\gamma([0,T];\R^d)$  (with norm $\|\cdot\|_\gamma$ defined by the sum of supremum norm and H\"older seminorm). The desired bound follows from \eqref{eq:XstabilityEstimate} along with the straightforward inequality
$$\mathcal{W}_{k, C^\gamma}(\mu^1, \mu^2)^k\leq \E[ \| X^2-X^1     \|^k_{\gamma}  ]     $$
which follows from independent coupling of the marginal laws $\mathcal{L}(X^i)=\mu^i, i=1,2.$
\end{proof}

\begin{rmk}\label{rmk 5.4}
An example where  condition \eqref{eq:ControlledVectorFieldsContinuity} is trivially satisfied is given in the case of time-independent coefficients $A^{1,i}_t\equiv A^{i}, C^{1,i}_t\equiv C^{i}$ for all $t$ and for $i=1,2.$ In this case, the last two terms on the right-hand side of the stability estimate \eqref{eq:MainStabilityEstimate} vanish. Finally, by virtue of classical continuity properties of the It\^o-Lyons map, the condition is also satisfied when the coefficients are solutions to RDEs driven by $\pmb{\eta}_i.$ 
\end{rmk}


\section{Randomization}\label{sec:Randomization}
For this section let
 $
 \Omega:=\Omega'\times\Omega'', \mathbb{P}=\overline{\mathbb{P}'\otimes\mathbb{P}''}$\footnote{$\mathbb{P}', \mathbb{P}''$ are assumed complete over appropriate $\sigma-$algebras.} that supports two (independent) Brownian motions $W(\omega)=W(\omega'), {W}^\perp(\omega)= {W}^\perp (\omega'')$. We denote by $(\mathcal{F}_t)_{t\in\mathbb{R_+}}, (\mathcal{F}^{\perp}_t)_{t\in\mathbb{R_+}}$ the augmented filtrations generated by $W, W^\perp$ respectively on $\Omega'\times\mathbb{R}_+$ and $\Omega''\times\mathbb{R}_+$. It is well known that these filtrations satisfy the usual conditions, see for example \cite[Section 2.7]{karatzas2014brownian}. By abusing notation we are going to denote the filtrations $(\mathcal{F}_t\times\Omega'')_{t\in\mathbb{R_+}}, (\Omega'\times \mathcal{F}^{\perp}_t)_{t\in\mathbb{R_+}}$ on $\Omega\times\mathbb{R}_+$ again with $(\mathcal{F}_t)_{t\in\mathbb{R_+}}, (\mathcal{F}^{\perp}_t)_{t\in\mathbb{R_+}}$. The same notation is also used for the corresponding $\sigma-$algebras. Again, note that the filtrations $(\mathcal{F}_t\vee\mathcal{F}^{\perp}_t)_{t\in\mathbb{R}_+}, (\mathcal{F}_t\vee\mathcal{F}^{\perp}_T)_{t\in\mathbb{R}_+}$ also satisfy the usual conditions.

Now we are going to lift $W^\perp$ to a geometric rough path.  To 
do that we define for each $(s,t)\in \Delta$, the Stratonovich integral 
$$
    (\mathbb{W}^{\perp}_{s,t})^{\mathrm{Strato}}:=\int_s^t \delta W^\perp_{s,r}\circ \ud W^{\perp}_r. 
$$ 
It is well known (see \cite[Chapter 3]{friz2020course}) that for all $\omega''$ outside a $\mathbb{P}''$-null set $\mathcal{N}''$ of $\Omega''$, the {\sl Brownian rough path}
$$
    \textbf{W}^{\perp} (\omega''):=\big(W^\perp(\omega''),(\mathbb{W}^\perp)^{\mathrm{Strato}}(\omega'')\big) 
$$    
belongs to the space $\mathscr{C}^{0,\gamma}_g$ for all $\gamma \in (1/3,1/2)$ and that the following {\it lifting map} is $(\mathcal{F}^\perp_t)_{t\in\mathbb{R}_+}-$prog. measurable:
\begin{equation*}\label{ito-lift}
\begin{array}{llll}
\textbf{W}^{\perp}: & [0,T]\times \Omega'' & \rightarrow & (\mathscr{C}^{0,\gamma}_g, \mathcal{B}(\mathscr{C}^{0,\gamma}_g))	\\
 & (t,\omega'') &\mapsto & \textbf{W}^{\perp}(\omega'')_{.\wedge t}.
\end{array}
\end{equation*}
For $\omega''\in \mathcal{N}''$ we set $\textbf{W}^{\perp}(\omega''):=(0,0)$, and the same convention will hold for every new process defined below whenever we randomize with respect to $\textbf{W}^{\perp}$.
 
\begin{rmk}\label{rmk 6.1}
 When we write $\mathbb{E}, \mathbb{E}', \mathbb{E}''$ we mean that we take expectations with respect to $\mathbb{P}, \mathbb{P}', \mathbb{P}''$.
\end{rmk}

Next, we define the following spaces of controls and vector flows.
\begin{defn}\label{dfn 6.2}  Let $p, d\in\mathbb{N}, \gamma\in(\frac{1}{3},\frac{1}{2})$ and 
 $0<\beta'\leq \beta\leq\gamma$ such that $\gamma+\beta+\beta'>1$.
    \begin{align*}
    \mathcal{A}&:= \big\{\alpha:\Omega'\times[0,T]\rightarrow\mathbb{R}^\kappa:\text{prog. measurable with respect to}\hspace{.1cm}\nonumber\\&\hspace{2.3cm}(\mathcal{F}_t)_{t\in[0,T]},\hspace{.1cm}\text{and}\hspace{.1cm} \sup_{t\in[0,T]}\mathbb{E}'[|\alpha_t|^k]<\infty,\hspace{.1cm}\text{for every}\hspace{.1cm}k\in [2,\infty)\big\},\\
       \mathcal{A}^{1}&:= \big\{\alpha:\Omega'\times\mathscr{C}^{0,\gamma}_g\times[0,T]\rightarrow\mathbb{R}^\kappa:\text{prog. measurable with respect to}\hspace{.1cm}\nonumber\\&\hspace{2.3cm}\big(\mathcal{F}_t\otimes\{\emptyset, \mathscr{C}^{0,\gamma}_g\}\big)_{t\in[0,T]},\footnotemark\hspace{.1cm}\text{and}\hspace{.1cm} {\alpha}(\pmb{\eta})\in \mathcal{A}\hspace{.1cm}\text{for every}\hspace{.1cm}\pmb{\eta}\in \mathscr{C}^{0,\gamma}_g\big\},\\
        \mathcal{A}^{2}&:= \big\{\alpha:\Omega'\times\mathscr{C}^{0,\gamma}_g\times[0,T]\rightarrow\mathbb{R}^\kappa:\text{prog. measurable with respect to}\hspace{.1cm}\nonumber\\&\hspace{2.3cm}\big(\mathcal{F}_t\otimes\mathcal{B}\big(\mathscr{C}^{0,\gamma}_g\big)\big)_{t\in[0,T]},\hspace{.1cm}\text{and}\hspace{.1cm} {\alpha}(\pmb{\eta})\in \mathcal{A}\hspace{.1cm}\text{for every}\hspace{.1cm}\pmb{\eta}\in \mathscr{C}^{0,\gamma}_g\big\},\\ 
        \mathscr{M}^{\beta,\beta'}_\bullet([0,T];\mathbb{R}^d)&:=\big\{(m,m'): \big[m:\mathscr{C}^{0,\gamma}_g\rightarrow C^{\beta}\big([0,T];\mathbb{R}^d\big)\hspace{.1cm}\text{measurable with respect to}\hspace{.1cm} \mathcal{B}\big(\mathscr{C}^{0,\gamma}_g\big)\hspace{0.1cm}\nonumber\\&\hspace{1.3cm}\text{and}\hspace{.1cm}m':\mathscr{C}^{0,\gamma}_g\rightarrow C^{\beta'}\big([0,T];\mathbb{R}^{d\times p}\big)\hspace{.1cm}\text{measurable with respect to}\hspace{.1cm} \mathcal{B}\big(\mathscr{C}^{0,\gamma}_g\big),\hspace{0.1cm}\nonumber\\&\hspace{1.3cm}\text{such that}\hspace{.1cm}(m(\pmb{\eta}),m'(\pmb{\eta}))\in \mathscr{D}^{\beta,\beta'}_{\eta}\big([0,T];\mathbb{R}^d\big)\hspace{.1cm}\text{for every}\hspace{.1cm}\pmb{\eta}\in \mathscr{C}^{0,\gamma}_g\big]\big\}.
    \end{align*}
    \end{defn}
    \footnotetext{$\mathcal{F}_t\otimes\{\emptyset,\mathscr{C}^{0,\gamma}_g\}=\{S\times \mathscr{C}^{0,\gamma}_g:S\in\mathcal{F}_t\}$.}

\begin{rmk}\label{rmk 6.3}
One can also define the following space of $\pmb{\eta}-$causal controls
\begin{align*}
    \mathcal{A}^{\text{Causal}}&:= \big\{\alpha:\Omega'\times\mathscr{C}^{0,\gamma}_g\times[0,T]\rightarrow\mathbb{R}^\kappa:\text{prog. measurable with respect to}\hspace{.1cm}\nonumber\\&\hspace{0.7cm}\big(\mathcal{F}_t\otimes \mathcal{B}(\mathscr{C}^{0,\gamma}_g)\big)_{t\in[0,T]}, {\alpha}(\pmb{\eta})\in \mathcal{A}\hspace{.1cm}\text{for every}\hspace{.1cm}\pmb{\eta}\in \mathscr{C}^{0,\gamma}_g,\hspace{.1cm}\text{and}\hspace{.1cm}\alpha_t(\omega',\pmb{\eta})=\alpha_t(\omega',\pmb{\eta}_{\cdot\wedge t})\big\}.
\end{align*}
Note that $\mathcal{A}^1\subseteq\mathcal{A}^{\text{Causal}}\subseteq \mathcal{A}^2$, and the randomized versions of controls $\alpha\in\mathcal{A}^{\text{Causal}}$, $\overline{\alpha}$, are $(\mathcal{F}_t\vee\mathcal{F}^{\perp}_t)_{t\in\mathbb{R}_+}$ prog. measurable processes.
\end{rmk}

At this point we introduce the following assumption.

\begin{enumerate}[label=(S6-R), ref=S6-R]
\setcounter{enumi}{5}
    \item\label{(S6-R)} There exist $A^1, C^1:\mathscr{C}^{0,\gamma}_g\rightarrow C^\beta\big([0,T];L(\mathbb{R}^d,\mathbb{R}^{d\times p})\big)$ measurable with respect to $\mathcal{B}\big(\mathscr{C}^{0,\gamma}_g\big)$ and $(A^1)', (C^1)':\mathscr{C}^{0,\gamma}_g\rightarrow C^{\beta'}\big([0,T];L\big(\mathbb{R}^p;L(\mathbb{R}^d,\mathbb{R}^{d\times p})\big)\big)$ measurable with respect to $\mathcal{B}\big(\mathscr{C}^{0,\gamma}_g\big)$, such that $(A^1(\pmb{\eta}),(A^{1})'(\pmb{\eta})), (C^1(\pmb{\eta}),(C^1)'(\pmb{\eta}))\in \mathscr{D}^{\beta,\beta'}_{\eta}\big([0,T];L(\mathbb{R}^d,\mathbb{R}^{d\times p})\big)\hspace{.1cm}\text{for every}\hspace{.1cm}\pmb{\eta}\in \mathscr{C}^{0,\gamma}_g$
    \footnote{We assume that $(A^{1}(0), (A^{1})'(0))=(C^{1}(0), (C^{1})'(0))=(0,0)$.},
    and, for any $M>0,$
    \begin{equation*}
        \|(A^{1}, (A^{1})'); (A^{1}, (A^{1})') \|_{\eta_1, \eta_2,\beta,\beta'}+\|(C^{1}, (C^{1})'); (C^{1}, (C^{1})') \|_{\eta_1, \eta_2,\beta,\beta'}\lesssim  \rho_{\gamma}(\pmb{\eta}_2, \pmb{\eta}_1),
    \end{equation*}
    for every $\pmb{\eta}_1,\pmb{\eta}_2\in \mathscr{C}^{0,\gamma}_g$ with $  \tnorm{\pmb{\eta}_1}_\gamma+\tnorm{\pmb{\eta}_2}_\gamma\leq M $. Furthermore, $(C^1(\pmb{\eta}),(C^1)'(\pmb{\eta}))$ satisfies assumption (\ref{(S4-R)}) for every $\pmb{\eta}\in \mathscr{C}^{0,\gamma}_g$.
\end{enumerate}

\begin{rmk}
\begin{enumerate}
    \item[(i)] Following Remark \ref{rmk 4.12}, one can alternatively make the following assumption with respect to $(C^1,(C^1)')$.
\begin{enumerate}[label=(S7-R), ref=S7-R]
    \item\label{(S7-R)} There exist $\lambda^{(2)}:\mathscr{C}^{0,\gamma}_g\rightarrow C^\beta\big([0,T];\mathbb{R}^{1\times p}\big)$ measurable with respect to $\mathcal{B}\big(\mathscr{C}^{0,\gamma}_g\big)$ and $(\lambda^{(2)})':\mathscr{C}^{0,\gamma}_g\rightarrow C^{\beta'}\big([0,T];L(\mathbb{R}^p,\mathbb{R}^{1\times p})\big)$ measurable with respect to $\mathcal{B}\big(\mathscr{C}^{0,\gamma}_g\big)$, such that $(\lambda^{(2)}(\pmb{\eta}),$ $(\lambda^{(2)})'(\pmb{\eta}))\in \mathscr{D}^{\beta,\beta'}_{\eta}\big([0,T];\mathbb{R}^{1\times p}\big)\hspace{.1cm}\text{for every}\hspace{.1cm}\pmb{\eta}\in \mathscr{C}^{0,\gamma}_g$, and for any $M>0$
    \begin{equation*}
        \|(\lambda^{(2)}(\pmb{\eta}_1), (\lambda^{(2)})'(\pmb{\eta}_1)); (\lambda^{(2)}(\pmb{\eta}_2), (\lambda^{(2)})'(\pmb{\eta}_2)) \|_{\eta_1, \eta_2,\beta,\beta'}\lesssim  \rho_{\gamma}(\pmb{\eta}_2, \pmb{\eta}_1),
    \end{equation*}
    for every $\pmb{\eta}_1,\pmb{\eta}_2\in \mathscr{C}^{0,\gamma}_g$ with  $\tnorm{\pmb{\eta}_1}_\gamma+\tnorm{\pmb{\eta}_2}_\gamma\leq M.$
\end{enumerate} 

    \item[(ii)] From (\ref{(S7-R)}) we immediately get that $$(\lambda^{(2)}(\pmb{\eta})\text{Id}_{d\times d}, (\lambda^{(2)})'(\pmb{\eta})\text{Id}_{d\times d})\in \mathscr{D}^{\beta,\beta'}_{\eta}\big([0,T];L(\mathbb{R}^d,\mathbb{R}^{d\times p})\big)$$
    $\text{for every}\hspace{.1cm}\pmb{\eta}\in \mathscr{C}^{0,\gamma}_g$, and that for any $M>0$
    \begin{equation*}
        \|(\lambda^{(2)}(\pmb{\eta}_1)\text{Id}_{d\times d}, (\lambda^{(2)})'(\pmb{\eta}_1)\text{Id}_{d\times d}); (\lambda^{(2)}(\pmb{\eta}_2)\text{Id}_{d\times d}, (\lambda^{(2)})'(\pmb{\eta}_2)\text{Id}_{d\times d}) \|_{\eta_1, \eta_2,\beta,\beta'}\lesssim  \rho_{\gamma}(\pmb{\eta}_2, \pmb{\eta}_1),
    \end{equation*}
    for every $\pmb{\eta}_1,\pmb{\eta}_2\in \mathscr{C}^{0,\gamma}_g$ with  $\tnorm{\pmb{\eta}_1}_\gamma+\tnorm{\pmb{\eta}_2}_\gamma\leq M.$

    \item[(iii)] An example when assumption (\ref{(S6-R)}) is satisfied is given in the case that $A^1, C^1$ take values in $ C^{\beta+\beta'}\big([0,T];L(\mathbb{R}^d,\mathbb{R}^{d\times p})\big)$ and are constant with respect to $\mathscr{C}^{0,\gamma}_g$. Then, $(A^{1},0),$   $ (C^{1},0)\in \mathscr{D}^{\beta,\beta'}_{\eta}\big([0,T];L(\mathbb{R}^d,\mathbb{R}^{d\times p})\big)$ for every $\pmb{\eta}\in\mathscr{C}^{0,\gamma}_g([0,T];\R^p)$. For more examples one can apply the results of \cite[Chapter 7]{friz2020course}.
\end{enumerate}
\end{rmk}

The next result shows that the per $\pmb{\eta}$ rough LQ MFG solutions we get from the analysis of the previous sections, with the help of the stability estimates from Section \ref{sec:Stability}, can be nicely packed together in terms of measurability by using the spaces of Definition \ref{dfn 6.2}. 

\begin{prop}\label{thm 6.4}
Let $p,d\in\mathbb{N}, \xi\in \mathbb{R}^d, \gamma\in(\frac{1}{3},\frac{1}{2})$ and 
 $0<\beta'\leq \beta\leq\gamma$ such that $\gamma+\beta+\beta'>1$. Additionally, let (\ref{S2}), (\ref{S3}), (\ref{(S6-R)}) hold true, $Q_t+\bar Q_t-\bar Q_t S_t\succeq 0$ and $Q+\bar Q-\bar Q S\succeq 0$. Then, there exist unique $(m^*,(m^*)')\in   \mathscr{M}^{\beta,\beta'}_\bullet([0,T];\mathbb{R}^d)$ and $\alpha^*\in \mathcal{A}^2$ such that, for every $\pmb{\eta}\in\mathscr{C}^{0,\gamma}_g$, we have
 \begin{equation*}
    \begin{aligned}
        \E'\!\left[\int_0^T \ell^{m^*(\pmb{\eta})}_t\big({X}^{*,\pmb{\eta}}_t,\alpha^*(\pmb{\eta})_t\big)\ud t + \Phi^{m^*(\pmb{\eta})}\big({X}^{*,\pmb{\eta}}_T\big)\right] = \min_{\alpha\in\mathcal{A}}\E'\!\left[\int_0^T \ell^{m^*(\pmb{\eta})}_t\big({X}^{\pmb{\eta},\alpha}_t,\alpha_t\big)\ud t + \Phi^{m^*(\pmb{\eta})}\big({X}^{\pmb{\eta},\alpha}_T\big)\right]
    \end{aligned}
\end{equation*}
while
\begin{equation*}
    \begin{aligned}
        m^*(\pmb{\eta})_t = \mathbb{E}'\big[ X^{*,\pmb{\eta}}_t\big],\hspace{.2cm}(m^*)'(\pmb{\eta})_t= (A^1(\pmb{\eta})_t+C^1(\pmb{\eta})_t)\mathbb{E}'\big[ X^{*,\pmb{\eta}}_t\big].
    \end{aligned}
\end{equation*}
\end{prop}
\begin{proof}
Let $\xi\in \mathbb{R}^d$, $\pmb{\eta}\in \mathscr{C}^{0,\gamma}_g([0,T];\R^p)$, $\alpha\in \mathcal{A}^{2}$ and $(m,m')\in \mathscr{M}^{\beta,\beta'}_\bullet([0,T];\mathbb{R}^d)$. Then, under \ref{(S6-R)}, from Theorem \ref{thm sol_rough_affine} and \cite[Proposition 3.10]{horst2025pontryagin}, we get that there exists a unique prog. measurable with respect to $\big(\mathcal{F}_t\otimes\mathcal{B}\big(\mathscr{C}^{0,\gamma}_g\big)\big)_{t\in[0,T]}$ process $X^{\pmb{\eta}}_t(\omega'):=X^{\pmb{\eta},\alpha(\pmb{\eta}),(m(\pmb{\eta}),m'(\pmb{\eta})),\xi}_t(\omega')$ which solves 
\begin{equation*}
\begin{aligned}
    X^{\pmb{\eta}}_t=\xi+\int_{0}^t\big(A_s X^{\pmb{\eta}}_s + B_s \alpha(\pmb{\eta})_s + C_s m(\pmb{\eta})_s \big)\ud s +\int_0^t \Sigma_s\ud W_s+\int_{0}^t \big(A^1_s X^{\pmb{\eta}}_s + C^1_s m(\pmb{\eta})_s\big)\ud\pmb{\eta}_s,
\end{aligned}
\end{equation*}
with $(X^{\pmb{\eta}},A^1X^{\pmb{\eta}}+C^1m(\pmb{\eta}))\in \bigcap_{k =2}^\infty\textbf{D}^{\gamma,\beta}_{\eta}\mathbb{L}^k\big([0,T];\Omega';\mathbb{R}^d\big)$.

Next, for every $\pmb{\eta}\in \mathscr{C}^{0,\gamma}_g$, under (\ref{(S6-R)}), there exist unique $(m^*(\pmb{\eta}),(m^*)'(\pmb{\eta}))\in \mathscr{D}^{\beta,\beta'}_{\eta}\big([0,T];\mathbb{R}^d\big)$ and $\alpha^*(\pmb{\eta})\in\mathcal{A}$
such that 
\begin{equation*}
    \begin{aligned}
        \E'\!\left[\int_0^T \ell^{m^*(\pmb{\eta})}_t\big({X}^{*,\pmb{\eta}}_t,\alpha^*(\pmb{\eta})_t\big)\ud t + \Phi^{m^*(\pmb{\eta})}\big({X}^{*,\pmb{\eta}}_T\big)\right] = \textrm{min}_{\alpha\in\mathcal{A}}\E'\!\left[\int_0^T \ell^{m^*(\pmb{\eta})}_t\big({X}^{\pmb{\eta},\alpha}_t,\alpha_t\big)\ud t + \Phi^{m^*(\pmb{\eta})}\big({X}^{\pmb{\eta},\alpha}_T\big)\right]
    \end{aligned}
\end{equation*}
where
\begin{align*}
    X^{*,\pmb{\eta}}_t(\omega'):=X^{\pmb{\eta},\alpha^*(\pmb{\eta}),(m^*(\pmb{\eta}),(m^*)'(\pmb{\eta})),\xi}_t(\omega'),\hspace{.2cm}
 X^{\pmb{\eta},\alpha}_t(\omega'):=X^{\pmb{\eta},\alpha,(m^*(\pmb{\eta}),(m^*)'(\pmb{\eta})),\xi}_t(\omega'),
\end{align*}
while
\begin{equation*}
    \begin{aligned}
        m^*(\pmb{\eta})_t = \mathbb{E}'\big[ X^{*,\pmb{\eta}}_t\big],\qquad(m^*)'(\pmb{\eta})_t= (A^1(\pmb{\eta})_t+C^1(\pmb{\eta})_t)\mathbb{E}'\big[ X^{*,\pmb{\eta}}_t\big].
    \end{aligned}
\end{equation*}
From Lemma \ref{lem:mStability}, we have that $m^*(\pmb{\eta}),(m^*)'(\pmb{\eta})$ depend continuously on $\pmb{\eta}$, thus the decoupling coefficients $ \mathcal{K}^{\pmb{\eta}}(\cdot,\cdot), M^{\pmb{\eta}}\big((m^*(\pmb{\eta}),(m^*)'(\pmb{\eta}))\big)_\cdot,P^{\pmb{\eta}}_\cdot, \Pi^{\pmb{\eta},(m^*({\pmb{\eta}}),(m^*)'(\pmb{\eta}))}_\cdot$ also depend continuously on $\pmb{\eta}$, and are deterministic. Moreover, we remind that $\mathscr{C}^{0,\gamma}_g$ is Polish. Hence Carath\'{e}odory's Lemma can be applied, see for instance \cite[4.51 Lemma]{aliprantis2006infinite}, with respect to the coefficients. So, by combining all the above, $(m^*,(m^*)')\in \mathscr{M}^{\beta,\beta'}_\bullet([0,T];\mathbb{R}^d)$, $X^{*,\pmb{\eta}}$ is prog. measurable with respect to $\big(\mathcal{F}_t\otimes\mathcal{B}\big(\mathscr{C}^{0,\gamma}_g\big)\big)_{t\in[0,T]}$ and $\alpha^*\in\mathcal{A}^2$, thus we have proved the result.
\end{proof}

Moving on, for $(m,m')\in\mathscr{M}^{\beta,\beta'}_\bullet([0,T];\mathbb{R}^d), \alpha\in\mathcal{A}^2$, we now define the following randomized versions as
\begin{equation*}
    \begin{aligned}
   \hspace{3.5cm} \overline{\alpha}(\omega)&:= \alpha(\pmb{\eta})|_{\pmb{\eta}=\textbf{W}^{\perp} (\omega'')},\\
    \overline{m}(\omega)&:= m(\pmb{\eta})|_{\pmb{\eta}=\textbf{W}^{\perp} (\omega'')},
    \\
    \overline{m}'(\omega)&:= m'(\pmb{\eta})|_{\pmb{\eta}=\textbf{W}^{\perp} (\omega'')},
    \\
      \overline{X}^{\overline{\alpha}}_t(\omega) &:=  X^{\pmb{\eta},\alpha(\pmb{\eta}),(m(\pmb{\eta}),m'(\pmb{\eta})),\xi}_t(\omega')\big|_{\pmb{\eta}=\textbf{W}^{\perp} (\omega'')}.
    \end{aligned}
\end{equation*}
From the above\footnote{See also \cite[Theorem 3.6]{FLZ25}.} we get that $\overline{\alpha}\in\overline{\mathcal{A}},\hspace{.1cm} \overline{m}\hspace{.1cm}\text{is}\hspace{.1cm}\mathcal{F}^{\perp}_T-\text{measurable},\hspace{.1cm}
\overline{m}'\hspace{.1cm}\text{is}\hspace{.1cm}\mathcal{F}^{\perp}_T-\text{measurable}\hspace{.1cm},(\overline{m},\overline{m}')\in\overline{\mathscr{M}}^{\beta,\beta'}_\bullet([0,T];\mathbb{R}^d)\hspace{.1cm}\text{and}\hspace{.1cm}
        \overline{X}\hspace{.1cm}\text{is}\hspace{.1cm}(\mathcal{F}_t\vee\mathcal{F}^{\perp}_T)_{t\in[0,T]}-\text{prog. measurable},$
where
\begin{equation*}
\begin{aligned}
   \overline{\mathcal{A}}&:= \big\{\overline{\alpha}:\Omega'\times\Omega''\times[0,T]\rightarrow\mathbb{R}^\kappa: \text{there exists}\hspace{.1cm}{\alpha}\in \mathcal{A}^{2}\\&\hspace{2cm}\text{such that}\hspace{.1cm} \overline{\alpha}_t(\omega',\omega'')={\alpha}_t(\omega',\textbf{W}^{\perp} (\omega''))\big\},\\
   \overline{\mathscr{M}}^{\beta,\beta'}_\bullet([0,T];\mathbb{R}^d)&:=\big\{(\overline{m},\overline{m}'): \text{there exists}\hspace{.1cm}(m,m')\in \mathscr{M}^{\beta,\beta'}_\bullet([0,T];\mathbb{R}^d)\\&\hspace{2cm}\text{such that}\hspace{.1cm} (\overline{m}(\omega''),\overline{m}'(\omega''))= \big(m\big(\textbf{W}^{\perp}(\omega'')\big),m'\big(\textbf{W}^{\perp} (\omega'')\big)\big)\big\}.
    \end{aligned}
\end{equation*}
Moreover, by \cite[Theorem 3.6]{FLZ25}, we get that for every $t\in[0,T]$ it holds
\begin{equation*}
    \mathcal{L}(\overline{X}^{\overline{\alpha}}_t)(\omega)= \mathcal{L}(\overline{X}^{\overline{\alpha}}_t|\mathcal{F}^\perp_T)(\omega'')=\mathcal{L}(X^{\pmb{\eta}}_t)\big|_{\pmb{\eta}=\textbf{W}^{\perp} (\omega'')}.
\end{equation*}

Next, for $\pmb{\eta}\in\mathscr{C}^{0,\gamma}_g, \alpha\in \mathcal{A}^2$ and $(m,m')\in \mathscr{M}^{\beta,\beta'}_\bullet([0,T];\mathbb{R}^d)$, we have
\begin{equation*}
\begin{aligned}\ell^{m(\pmb{\eta})}_t\big({X}^{\pmb{\eta}}_t,\alpha(\pmb{\eta})_t\big)\geq 0,& \hspace{.2cm} \Phi^{m(\pmb{\eta})}\big({X}^{\pmb{\eta}}_T\big)\geq0\\
\ell^{\overline{m}}_t\big(\overline{X}^{\overline{\alpha}}_t,\overline{\alpha}_t\big)\geq 0,&\hspace{.2cm} \Phi^{\overline{m}}\big(\overline{X}^{\overline{\alpha}}_T\big)\geq 0
\end{aligned}
\end{equation*}
so, by \cite[Proposition 7.2]{FLZ24} or \cite[Proposition 5.2]{horst2025pontryagin}, we get for every $i\in\mathbb{N}$
\begin{align*}
     &\E'\!\left[\int_0^T \Big(i\wedge\ell^{m(\pmb{\eta})}_t\Big)\big({X}^{\pmb{\eta}}_t,\alpha(\pmb{\eta})_t\big)\ud t + \Big(i\wedge\Phi^{m(\pmb{\eta})}\Big)\big({X}^{\pmb{\eta}}_T\big)\right]\Bigg|_{\pmb{\eta}=\textbf{W}^{\perp} (\omega'')}\\
     =&\E\!\left[\int_0^T \Big(i\wedge\ell^{\overline{m}}_t\Big)\big(\overline{X}^{\overline{\alpha}}_t,\overline{\alpha}_t\big)\ud t + \Big(i\wedge\Phi^{\overline{m}}\Big)\big(\overline{X}^{\overline{\alpha}}_T\big)\bigg|\mathcal{F}^{\perp}_T\right].
\end{align*}
Hence, by the monotone convergence theorem we get 
\begin{equation}\label{eq 72}
\begin{aligned}
     &\E'\!\left[\int_0^T \ell^{m(\pmb{\eta})}_t\big({X}^{\pmb{\eta}}_t,\alpha(\pmb{\eta})_t\big)\ud t + \Phi^{m(\pmb{\eta})}\big({X}^{\pmb{\eta}}_T\big)\right]\Bigg|_{\pmb{\eta}=\textbf{W}^{\perp} (\omega'')}\\
     =&\E\!\left[\int_0^T \ell^{\overline{m}}_t\big(\overline{X}^{\overline{\alpha}}_t,\overline{\alpha}_t\big)\ud t + \Phi^{\overline{m}}\big(\overline{X}^{\overline{\alpha}}_T\big)\bigg|\mathcal{F}^{\perp}_T\right].
     \end{aligned}
\end{equation}
We note that with a trivial lift $\mathcal{A}$ can be seen as $\mathcal{A}^1$. So, in case that ${a}\in\mathcal{A}$, equation \eqref{eq 72} is written as
\begin{equation}\label{eq 73}
\begin{aligned}
     &\E'\!\left[\int_0^T \ell^{m(\pmb{\eta})}_t\big({X}^{\pmb{\eta}}_t,\alpha_t\big)\ud t + \Phi^{m(\pmb{\eta})}\big({X}^{\pmb{\eta}}_T\big)\right]\Bigg|_{\pmb{\eta}=\textbf{W}^{\perp} (\omega'')}\\
     =&\E\!\left[\int_0^T \ell^{\overline{m}}_t\big(\overline{X}^{{\alpha}}_t,{\alpha}_t\big)\ud t + \Phi^{\overline{m}}\big(\overline{X}_T^{{\alpha}}\big)\bigg|\mathcal{F}^{\perp}_T\right].
     \end{aligned}
\end{equation}
For every $(m,m')\in \mathscr{M}^{\beta,\beta'}_\bullet([0,T];\mathbb{R}^d)$ we define the following value functions
\begin{align}\label{eq:ValueFunctionsRandomisation}
\mathcal{V}^{(m,m'),1}(\omega)&:= \mathrm{essinf}_{\alpha\in \mathcal{A}^1}\E\!\left[\int_0^T \ell^{\overline{m}}_t\big(\overline{X}^{\overline{\alpha}}_t,\overline{\alpha}_t\big)\ud t + \Phi^{\overline{m}}\big(\overline{X}^{\overline{\alpha}}_T\big)\bigg|\mathcal{F}^{\perp}_T\right],\\
    \mathcal{V}^{(m,m'),2}(\omega)&:= \mathrm{essinf}_{\alpha\in \mathcal{A}^2}\E\!\left[\int_0^T \ell^{\overline{m}}_t\big(\overline{X}^{\overline{\alpha}}_t,\overline{\alpha}_t\big)\ud t + \Phi^{\overline{m}}\big(\overline{X}^{\overline{\alpha}}_T\big)\bigg|\mathcal{F}^{\perp}_T\right],\\
    \overline{\mathcal{V}}^{(m,m')}(\omega)&:= \left(\textrm{inf}_{\alpha\in\mathcal{A}}\E'\!\left[\int_0^T \ell^{m(\pmb{\eta})}_t\big({X}^{\pmb{\eta}}_t,\alpha_t\big)\ud t + \Phi^{m(\pmb{\eta})}\big({X}^{\pmb{\eta}}_T\big)\right]\right)\Bigg|_{\pmb{\eta}=\textbf{W}^{\perp} (\omega'')}.
\end{align}
Now we are ready to state our final result. 
\begin{thm}\label{thm 6.7}
Let $p, d\in\mathbb{N}, \xi\in \mathbb{R}^d, \gamma\in(\frac{1}{3},\frac{1}{2})$ and 
 $0<\beta'\leq \beta\leq\gamma$ such that $\gamma+\beta+\beta'>1$. Additionally, let (\ref{S2}), (\ref{S3}), (\ref{(S6-R)}) hold true, $Q_t+\bar Q_t-\bar Q_t S_t\succeq 0$ and $Q+\bar Q-\bar Q S\succeq 0$. Then, there exist $(\overline{m^*},(\overline{m^*})')\in \overline{\mathscr{M}}^{\beta,\beta'}_\bullet([0,T];\mathbb{R}^d)$ and $\overline{\alpha^*}\in  \overline{\mathcal{A}}$ such that 
 \begin{equation*}
      \overline{\mathcal{V}}^{(m^*,(m^*)')}(\omega)=\mathcal{V}^{(m^*,(m^*)'),2}(\omega)=\mathcal{V}^{(m^*,(m^*)'),1}(\omega)= \E\!\left[\int_0^T \ell^{\overline{m^*}}_t\big(\overline{X}^{\overline{\alpha^*}}_t,\overline{\alpha^*}_t\big)\ud t + \Phi^{\overline{m^*}}\big(\overline{X}^{\overline{\alpha^*}}_T\big)\bigg|\mathcal{F}^{\perp}_T\right]
 \end{equation*}
 while 
 \begin{equation*}
 \overline{m^*}_t = \mathbb{E}\Big[ \overline{X}^{\overline{\alpha^*}}_t\Big|\mathcal{F}^{\perp}_T\Big], \qquad (\overline{m^*})'_t = ( A_t^1  + C_t^1)|_{\pmb{\eta}=\textbf{W}^{\perp} (\omega'')}  \mathbb{E}\Big[ \overline{X}^{\overline{\alpha^*}}_t\Big|\mathcal{F}^{\perp}_T\Big].
\end{equation*}
\end{thm}
\begin{proof}
Because $\mathcal{A}^1\subseteq \mathcal{A}^2$, and by the first part of the proof of \cite[Theorem 7.4]{FLZ24}, we have 
\begin{equation}\label{eq. 70}
\overline{\mathcal{V}}^{(m,m')}(\omega)\leq\mathcal{V}^{(m,m'),2}(\omega)\leq \mathcal{V}^{(m,m'),1}(\omega),\hspace{.2cm}\mathbb{P}-a.s..
\end{equation}
Let $(m^*,(m^*)')\in \mathscr{M}^{\beta,\beta'}_\bullet([0,T];\mathbb{R}^d)$ , $\alpha^*\in  \mathcal{A}^{2}$ be the objects described in Proposition \ref{thm 6.4}. Then, by \eqref{eq 72} and \eqref{eq 73}, we immediately get that 
\begin{equation*}
    \begin{aligned}
       \overline{\mathcal{V}}^{(m^*,(m^*)')}(\omega)= \E\!\left[\int_0^T \ell^{\overline{m^*}}_t\big(\overline{X}^{\overline{\alpha^*}}_t,\overline{\alpha^*}_t\big)\ud t + \Phi^{\overline{m^*}}\big(\overline{X}^{\overline{\alpha^*}}_T\big)\bigg|\mathcal{F}^{\perp}_T\right].
    \end{aligned}
\end{equation*}
We remark that this is a major simplification in contrast with the situation in \cite{FLZ24}, where one does not know \emph{a priori} that $\overline{\mathcal{V}}^{(m^*,(m^*)')}$ is $\mathcal{F}^{\perp}_T-$measurable.
Thus, we have 
\begin{equation*}
    \begin{aligned}
        \overline{\mathcal{V}}^{(m^*,(m^*)')}(\omega)= \mathrm{essinf}_{\alpha\in \mathcal{A}}\E\!\left[\int_0^T \ell^{\overline{m^*}}_t\big(\overline{X}^{{\alpha}}_t,{\alpha}_t\big)\ud t + \Phi^{\overline{m^*}}\big(\overline{X}^{{\alpha}}_T\big)\bigg|\mathcal{F}^{\perp}_T\right],
    \end{aligned}
\end{equation*}
and by viewing again $\mathcal{A}$ as $\mathcal{A}^1$, we get by \eqref{eq. 70} that
\begin{equation*}
    \overline{\mathcal{V}}^{(m^*,(m^*)')}(\omega)=\mathcal{V}^{(m^*,(m^*)'),2}(\omega)=\mathcal{V}^{(m^*,(m^*)'),1}(\omega),\hspace{.2cm}\mathbb{P}-a.s..
\end{equation*}
\end{proof}

\vspace{.5cm}
\begin{rmk}[From rough MFG to causal conditional and unconditional MFGs]

\begin{enumerate} 
\item[(i)] It is worth mentioning that $\big(\overline{\alpha}^*,(\overline{m^*}, (\overline{m^*})')\big)$ also attain the optimal value of the conditional cost functional over the family of causal controls; note however that in general ${\alpha}^*$ is not a priori expected to be causal (see for example Section 1.1. Pathwise stochastic control problems of \cite{horst2025pontryagin}), hence, strictly speaking, $\big(\overline{\alpha}^*,(\overline{m^*}, (\overline{m^*})')\big)$ is not a solution of the causal conditional MFG. In other words, from Remark \ref{rmk 6.3} and Theorem \ref{thm 6.7} we 
immediately get that
\begin{equation*}
   \overline{\mathcal{V}}^{(m^*,(m^*)'),\text{Causal}}(\omega)  = \mathbb{E}\left[\int_0^T \ell^{\overline{m^*}}_t\big(\overline{X}^{\overline{\alpha^*}}_t,\overline{\alpha^*}_t\big)\ud t + \Phi^{\overline{m^*}}\big(\overline{X}^{\overline{\alpha^*}}_T\big)\bigg|\mathcal{F}^{\perp}_T\right],
\end{equation*}
where 
\begin{equation*}
     \overline{\mathcal{V}}^{(m,m'),\text{Causal}}(\omega) :=\mathrm{essinf}_{\alpha\in \mathcal{A}^{\text{Causal}}}\E\!\left[\int_0^T \ell^{\overline{m}}_t\big(\overline{X}^{\overline{\alpha}}_t,\overline{\alpha}_t\big)\ud t + \Phi^{\overline{m}}\big(\overline{X}^{\overline{\alpha}}_T\big)\bigg|\mathcal{F}^{\perp}_T\right].
\end{equation*}

\item[(ii)] Moreover, solutions of the conditional MFG (according to Theorem \ref{thm 6.7}) also solve the unconditional MFG over the family $\overline{\mathcal{A}}$ (which of course includes controls that depend on the whole path of the common noise). Indeed, for all $\overline{\alpha}\in  \overline{\mathcal{A}}$ we have
\begin{align*}
    \mathbb{E}\left[\int_0^T \ell^{\overline{m^*}}_t\big(\overline{X}^{\overline{\alpha}}_t,\overline{\alpha}_t\big)\ud t + \Phi^{\overline{m^*}}\big(\overline{X}^{\overline{\alpha}}_T\big)\right] &= \mathbb{E}\Bigg[\E\!\left[\int_0^T \ell^{\overline{m^*}}_t\big(\overline{X}^{\overline{\alpha}}_t,\overline{\alpha}_t\big)\ud t + \Phi^{\overline{m^*}}\big(\overline{X}^{\overline{\alpha}}_T\big)\bigg|\mathcal{F}^{\perp}_T\right]\Bigg]\\
    \geq \mathbb{E}\Big[ \mathcal{V}^{(m^*,(m^*)'),2}(\omega)\Big] &=   \mathbb{E}\left[\int_0^T \ell^{\overline{m^*}}_t\big(\overline{X}^{\overline{\alpha^*}}_t,\overline{\alpha^*}_t\big)\ud t + \Phi^{\overline{m^*}}\big(\overline{X}^{\overline{\alpha^*}}_T\big)\right].
\end{align*}
\end{enumerate}
\end{rmk}

\section{Outlook and conclusion}
In this paper we introduced a new class of Linear-Quadratic MFG with common noise that is modeled as a generic geometric rough path. Apart from the well-posedness theory, which rests upon a new rough consistency condition \eqref{eq:rConsistency} and (non-symmetric) rough Riccati equations \eqref{eq:rRiccati}, we also obtained stability properties of equilibria with respect to the common noise rough path (Theorem \ref{thm:main} $(i),(ii)$). Finally, we replaced the driving rough path by a Brownian geometric rough path, independent of the idiosyncratic noise, and showed that solutions to this randomized rough MFG correspond to solutions of a stochastic MFG that is conditioned on the common noise (Theorem \ref{thm:main} $(iii)$).

The partially anticipating feature that arises after randomization of the rough common noise may be viewed from several complementary perspectives:

 In the spirit of the pathwise program initiated by Lyons and, in the PDE setting, by Lions–Souganidis, one motivation is mathematical in its own right: allowing a common signal to be treated pathwise leads naturally to new solution concepts, stability questions, and analytical structures. From this perspective, this work can be seen as the first step towards several natural extensions such as MFGs with nonlinear state dynamics, $N-$player games, propagation of chaos, rates of convergence etc.

    Moreover, anticipative or pathwise control problems arise naturally even when the original stochastic control problem is entirely nonanticipative. This viewpoint goes back at least to Davis–Burstein \cite{DB92} and appears prominently in dual formulations of stochastic control (see e.g. \cite{bank2026duality}): optimization with knowledge of a complete noise trajectory can provide bounds, representations, and computational tools for the original adapted problem. Thus, pathwise problems of the type considered here may also be useful as auxiliary objects in duality arguments, independently of whether anticipation is part of the primitive economic model.
    
    Finally, the parameter $t$ need not always be interpreted as chronological time. In machine-learning applications, for instance, it may represent learning time, an optimization parameter, or network depth. In such settings, data associated with later parameter values may already be available when a decision at an earlier parameter value is made. Partial or even complete access to the driving signal is then considerably less exotic than it would be under a literal physical-time interpretation.
 
These perspectives are neither exhaustive nor intended to provide fully specified modeling applications for the present framework. Rather, they are meant to emphasize that the mathematical structure studied here is potentially useful beyond the standard adapted common-noise paradigm, it admits several interesting extensions, and that the partially anticipating formulation arising from randomization should not in itself be viewed as an obstruction to meaningful applications.

\appendix

\section{Appendix}\label{sec ap.}

\subsection{Linear RDEs}

\begin{prop}\label{prop:LinearRDEestimate} Let $p, d\in\mathbb{N}, \gamma\in(\frac{1}{3},\frac{1}{2})$ and 
 $0<\beta'\leq \beta\leq\gamma$ such that $\gamma+\beta+\beta'>1$. Additionally, let $\xi_1, \xi_2\in\R^d$ and $\pmb{\eta}_1,\pmb{\eta}_2\in \mathscr{C}^{0,\gamma}_g([0,T];\mathbb{R}^p).$ For $T>0,$ $i=1,2$ and given $B_i\in C([0,T];\R^d),$ $(A_i, A_i')\in \mathscr{D}^{\beta, \beta'}_{\eta_i}\big([0,T];L(\mathbb{R}^d,\mathbb{R}^{d\times p})\big)$,  
  there is a unique (global)
  solution 
  $(X_i, A_i X_i)$ to
  \[ \ud X_i(t) = B_{i}(t)X_{i}(t)\ud t+(A_i, A_i')(t) X_i(t) \ud \pmb{\eta}_i, X_{i}(0) = \xi_i \]
  on $[0,T].$ Then for a constant $M>0$ such that 
  $$
\sum_{i=1}^2\Big[\tnorm{\pmb{\eta}_i}_\gamma+\|B_i\|_{\infty;[0,T]}+\| (A_i, A_i')
     \|_{ \eta_i,\beta, \beta'}\Big]\leq M
  $$
  we have
  \begin{equation*}
      \begin{aligned}
           \|X_1 - X_2 \|_{\infty;[0,T]} &+ \| (X_1, A_1 X_1) ; (X_2, A_2 X_2) \|_{\eta_1, \eta_2,\gamma,\beta} \\&\lesssim  | \xi_1 - \xi_2 | + \rho_{\gamma}
     (\pmb{\eta}_1, \pmb{\eta}_2) + \| (A_1,A_1') ; (A_2,A_2')
     \|_{\eta_1, \eta_2,\beta, \beta'}+\|B_1-B_2\|_{\infty;[0,T]},
      \end{aligned}
  \end{equation*}
  up to a nonnegative multiplicative constant that depends on $M.$
\end{prop}
\begin{proof} In the case $B_i=0, i=1,2,$ and bounded coefficients, such estimates are well-known (e.g. as a very special case of the estimates in \cite{FHL21}). In combination with estimates for linear RDEs, e.g. along the lines of 
\cite[Ch.8]{friz2020course} the statement then follows by localization. For non-zero drifts $B^i$, the estimate follows by either repeating the proof with the terms
$\int B_iX_i d t$ included in the Picard map, or by extracting the estimate, a posteriori, from Duhamel plus
Gr\"onwall's inequality.
\end{proof}

\subsection{Backward symmetric rough Riccati}

\begin{thm}\label{thm: sym_rough_riccati}
Let $p\in\mathbb{N}, \pmb{\eta} \in \mathscr{C}^{0,\gamma}_g([0,T];\mathbb{R}^p)$, with $\gamma\in(\frac{1}{3},\frac{1}{2})$ and 
 $0<\beta'\leq \beta\leq\gamma$ such that $\gamma+\beta+\beta'>1$. Additionally, let $\bar{Q}\in \mathbb{R}^{d\times d}$ and $A,B,Q,R:[0,T]\rightarrow \mathbb{R}^{d\times d}$ bounded and Borel measurable on \([0,T]\), with
\(R_t\succ \lambda \text{Id}_{d\times d}\), for some $\lambda >0$, and \(\bar Q, Q_t\succeq 0\), and $\big({A}^{1},(A^1)'\big)\in \mathscr{D}^{\beta,\beta'}_{\eta}\big([0,T];L(\mathbb{R}^d,\mathbb{R}^{d\times p})\big)$ deterministic controlled rough path with respect to ${\eta}$. Then, the backward symmetric rough Riccati equation
\begin{equation}\label{eq:SymmetricRRiccati}
     {p}_t= \bar{Q}+\int_t^T \big({p}_sA^1_s +(A^1_s)^\top {p}_s\big) \ud\pmb{\eta}_s
    +\int_t^T\Big({p}_s {A}_s  + ({A}_s)^{\top} {p}_s - {p}_s{B}_s {R}_s^{-1} ({B}_s)^\top {p}_s + {Q}_s\Big)\ud s,
\end{equation}
has a unique (global) solution 
with $(p,p')\in \mathscr{D}^{\gamma,\beta}_{\eta}\big([0,T];\mathbb{R}^{d\times d}\big)$. Moreover, the solution is locally Lipschitz in its data and in particular in $\pmb{\eta}$ and $(A^1, (A^1)').$
\end{thm}

\begin{proof}
Immediate from  \cite[Theorem 4.2]{bugini2025rough}. Note however that \eqref{eq:SymmetricRRiccati} has a terminal condition and evolves backward in time. This poses no issues since, on the one hand, the Picard iteration scheme is symmetric with respect to time. On the other hand, we can proceed by analyzing the "backward" increments with the usual time flow, i.e. from the left point in time to the right. Thus, the aforementioned result applies directly to backward rough Riccati equations. Of course, such a strategy works since we are only dealing with deterministic equations. 
\end{proof}

\subsection{Rough stochastic product rule}

As before, we assume that \((\Omega,\mathcal F,\mathbb{P})\) is a complete probability space and $(\mathcal{F}_t)_{t\in\mathbb{R}_+}$ the usual augmentation of a filtration generated by a $q-$dimensional Brownian motion $W$ and a $\sigma$-algebra $\mathcal{F}_0^0$ independent to the Brownian motion.

\begin{thm}\label{rough_product_formula}
 Let $p\in\mathbb{N}, \pmb{\eta} \in \mathscr{C}^{0,\gamma}_g([0,T];\mathbb{R}^p)$, with $\gamma\in(\frac{1}{3},\frac{1}{2})$ and 
 $0<\beta''\leq \beta'\leq\gamma$ such that $\gamma+\beta'+\beta''>1$. Furthermore, let $X, Y$ be progressively measurable processes with values in $\mathbb{R}^{d\times d}, \mathbb{R}^d (\text{or}\hspace{.1cm} \mathbb{R}^{d\times d})$ respectively, $(X',X'')\in \bigcap_{k =2}^\infty\textbf{D}^{\beta',\beta''}_{\eta}\mathbb{L}^k\big([0,T];L(\mathbb{R}^p,\mathbb{R}^{d\times d})\big),$ $ (Y',Y'')\in \bigcap_{k =2}^\infty\textbf{D}^{\beta',\beta''}_{\eta}\mathbb{L}^k\big([0,T];\mathbb{R}^{d\times p}\big)$ $\big(\text{or}\hspace{.1cm}\bigcap_{k =2}^\infty\textbf{D}^{\beta',\beta''}_{\eta}\mathbb{L}^k\big([0,T];L(\mathbb{R}^p,\mathbb{R}^{d\times d})\big)\big)$ stochastic controlled rough paths with respect to ${\eta}$, $X_0,Y_0\in \mathcal{F}_0$, and $b^1,b^2,\sigma^1,\sigma^2$ progressively measurable processes with $$\mathbb{E}[|X_0|^k]+\mathbb{E}[|Y_0|^k]+\sum_{i=1}^2\Big[\sup_{t\in[0,T]}\mathbb{E}\big[|b^i_t|^k\big] + \sup_{t\in[0,T]}\mathbb{E}\big[|\sigma^i_t|^k\big]\Big] < \infty,\hspace{.2cm}\text{for every}\hspace{.1cm} k\in [2,\infty),$$ such that
 \begin{align*}
     X_t &= X_0+\int_0^t b^1_s\ud s +\int_0^t \sigma^1_s\ud W_s+\int_0^t(X',X'')_s\ud \pmb{\eta}_s,\\
     Y_t &= Y_0+\int_0^t b^2_s\ud s +\int_0^t \sigma^2_s\ud W_s+\int_0^t(Y',Y'')_s\ud \pmb{\eta}_s.
 \end{align*}
 Then, for every $t\in[0,T]$, $\mathbb{P}-$a.s., we have
 \begin{gather*}
     X_tY_t = X_0Y_0 + \int_0^t \Big(X_sb^2_s+b^1_sY_s+ \sum_{i=1}^q (\sigma^2_i)_s(\sigma^1_i)_s\Big)\ud s + \int_0^t \big( X_s \sigma^2_s + \sigma^1_s Y_s \big) \ud W_s + \int_0^t ({Z}',{Z}'')_s \ud \pmb{\eta}_s,
 \end{gather*}
 where $({Z}',{Z}''):=\big(X Y' + X'Y,2X'Y'+  X Y'' + X'' Y\big) \in \bigcap_{k=2}^{\infty}\textbf{D}^{\beta',\beta''}_{\eta}\mathbb{L}^k\big([0,T];\mathbb{R}^d\big)$.
\end{thm}

\begin{proof}
We treat the case where $Y_t$ takes values in $\mathbb{R}^d;$ the case where it takes values in $\mathbb{R}^{d\times d}$ is completely similar.

We define
$(\pi^1)^{-1}:\mathbb{R}^{d\times d}\rightarrow \mathbb{R}^{d\times d}\times \mathbb{R}^{d}, (\pi^2)^{-1}:\mathbb{R}^d\rightarrow\mathbb{R}^{d\times d}\times \mathbb{R}^{d}$ where $(\pi^1)^{-1}(s):= (s,0), (\pi^2)^{-1}(s):= (0,s)$. Then, the aggregate vector $Z_t:=(X_t,Y_t)=(\pi^1)^{-1}(X_t)+ (\pi^2)^{-1}(Y_t)$ is progressively measurable with values in $\mathbb{R}^{d\times d}\times \mathbb{R}^d$, $\mathbb{E}[|Z_0|^k]<\infty$, for every $k\in [2,\infty)$, and $(\widetilde{Z}',\widetilde{Z}''):=\big((\pi^1)^{-1}(X')+ (\pi^2)^{-1}(Y'), (\pi^1)^{-1}(X'')+ (\pi^2)^{-1}(Y'')\big) \in \bigcap_{k=2}^{\infty} \textbf{D}^{\beta',\beta''}_{\eta}\mathbb{L}^k$, is a stochastic controlled rough path with respect to $\pmb{\eta}$ such that
\begin{gather*}
    Z_t=Z_0+\int_{0}^t b_s\ud s + \int_0^t \sigma_s \ud W_s + \int_{0}^t (\widetilde{Z}',\widetilde{Z}'')_s \ud \pmb{\eta}_s,
\end{gather*}
where $b_t:= (\pi^1)^{-1}(b^1_t)+ (\pi^2)^{-1}(b^2_t), \sigma_t:= (\pi^1)^{-1}(\sigma^1_t)+ (\pi^2)^{-1}(\sigma^2_t)$.
Now, let us define the function $G:\mathbb{R}^{d\times d}\times \mathbb{R}^{d}\rightarrow \mathbb{R}^d$ where $G(s_1,s_2):= s_1s_2$. Because ${G}$ is bilinear we get, for every $(s_1,s_2)\in \mathbb{R}^{d\times d}\times \mathbb{R}^d$, that
\begin{align*}
    \big[DG(s_1,s_2)\big](j_1,j_2) &= s_1 j_2 + j_1 s_2,\\
     \big[D^2G(s_1,s_2)\big](j_1,j_2)(h_1,h_2)&= j_1h_2 + h_1j_2  \hspace{.2cm}\text{and}\\
     \big[D^\ell G(s_1,s_2)\big](\cdot)&= 0, \hspace{.1cm}\ell\geq 3.
\end{align*}
Obviously, $DG$ is linear and $D^\ell G$ constant for ever $\ell\geq 2$, so $|D^\ell G(s)|\leq |D^2G|(1+|s|)$ for every $\ell \in\mathbb{N}$. 
By direct calculations we have
\begin{gather*}
    \big[DG(Z_t)\big](b_t) = X_tb^2_t+b^1_tY_t,\hspace{.2cm}
    \big[DG(Z_t)\big](\sigma_t) = X_t \sigma^2_t + \sigma^1_tY_t,\\
    \big[DG(Z_t)\big](Z'_t) = X_t Y'_t + X'_tY_t,\hspace{0.02cm}\big[D^2G(Z_t)\big](Z'_t)(Z'_t)= 2X'_t Y'_t\\
     \big[DG(Z_t)\big](Z''_t) =  X_t Y''_t + X''_t Y_t,\\    \sum_{i=1}^q\big[D^2G(Z_t)\big]\big((\sigma_i)_t\big)\big((\sigma_i)_t\big)= 2\sum_{i=1}^q (\sigma^2_i)_t(\sigma^1_i)_t.
\end{gather*}
We apply \cite[Proposition 2.16]{bugini2024malliavin} to $G(Z_t)$ and the proof is complete.
\end{proof}

\section*{Acknowledgments}
\smallskip
\noindent \textbf{Funding}  
PKF, UH and ST acknowledge support from 
Deutsche Forschungsgemeinschaft (DFG, German Research Foundation) – CRC/TRR 388 
"Rough Analysis, Stochastic Dynamics and Related Fields" – Project ID 516748464 (B04).

\printbibliography[heading=bibintoc]

\end{document}